\pdfmapfile{=<name>.map}
\documentclass[12pt,oneside,reqno,fleqn]{amsart}
\usepackage{amssymb}
\usepackage{bbm}
\usepackage{cases}
\usepackage{amsmath,mathtools}
\usepackage{graphicx}
\usepackage{mathrsfs}
\usepackage{stmaryrd}
\usepackage{color}
\usepackage{soul}
\usepackage[dvipsnames]{xcolor}
\usepackage{amsfonts}
\usepackage{amsmath,amssymb,amsthm}
\usepackage{libertine}
\usepackage[T1]{fontenc}
\usepackage[libertine,varbb]{newtxmath}
\usepackage[colorlinks,linkcolor=Red,citecolor=Red,urlcolor=Red]{hyperref}
\usepackage[shortlabels] {enumitem}
\usepackage[nameinlink,capitalize]{cleveref}
\usepackage[spacing=true,kerning=true,tracking=true,letterspace=50]{microtype}
\usepackage{orcidlink}
\usepackage[disable]{todonotes}
\hypersetup{pdftitle={Taming singular SDEs}, pdfauthor={Khoa L\^e and Chengcheng Ling}}
\numberwithin{equation}{section}
\newcommand{\be}{\begin{eqnarray}}
\newcommand{\ee}{\end{eqnarray}}
\newcommand{\ce}{\begin{eqnarray*}}
\newcommand{\de}{\end{eqnarray*}}
\newtheorem{theorem}{Theorem}[section]
\newtheorem{lemma}[theorem]{Lemma}
\newtheorem{proposition}[theorem]{Proposition}
\newtheorem{corollary}[theorem]{Corollary}
\theoremstyle{remark}

\newtheorem{example}[theorem]{Example}
\newtheorem{remark}[theorem]{Remark}
\newtheorem{definition}[theorem]{Definition}
\newtheorem{innercustomthm}{Condition}
\newenvironment{customcon}[1]
  {\renewcommand\theinnercustomthm{#1}\innercustomthm}
  {\endinnercustomthm}
\crefname{eqn}{Equation}{Equations}
\crefname{assumption}{Assumption}{Assumptions}
\crefname{innercustomthm}{Condition}{Conditions}
\crefrangelabelformat{innercustomthm}{#3#1#4-#5#2#6}

\def\b{\beta}
\def\p{\partial}

\def\<{{\langle}}
\def\>{{\rangle}}
\def\({{\Big(}}
\def\){{\Big)}}

\def\bx{{\mathbf{x}}}

\def\dif{d}

\def\min{{\mathord{{\rm min}}}}

\def\={&\!\!=\!\!&}
\def\bt{\begin{theorem}}
\def\et{\end{theorem}}
\def\bl{\begin{lemma}}
\def\el{\end{lemma}}
\def\br{\begin{remark}}
\def\er{\end{remark}}
\def\bd{\begin{definition}}
\def\ed{\end{definition}}
\def\bp{\begin{proposition}}
\def\ep{\end{proposition}}
\def\bc{\begin{corollary}}
\def\ec{\end{corollary}}
\def\bx{\begin{example}}
\def\ex{\end{example}}

\def\cB{{\mathcal B}}

\def\cF{{\mathcal F}}
\def\cff{\cF}
\def\cG{{\mathcal G}}
\def\cgg{\cG}
\def\cH{{\mathcal H}}
\def\chh{\cH}

\def\cM{{\mathcal M}}
\newcommand\cmm{\cM}

\def\mE{{\mathbb E}}
\def\E{\mE}

\def\mI{{\mathbb I}}

\def\mL{{\mathbb L}}

\def\mP{{\mathbb P}}

\def\mR{{\mathbb R}}

\def\bP{{\mathbf P}}

\def\geq{\geqslant}
\def\leq{\leqslant}
\def\les{\lesssim}

\def\bP{{\mathbf P}}

\newcommand{\R}{{\mathbb R}}
\newcommand{\Rd}{{\R^d}}
\newcommand{\tand}{\quad\text{and}\quad}
\newcommand{\LL}{{\mathbb{L}}}

\newcommand{\1}{{\mathbf 1}}
\newcommand{\caa}{{\mathcal A}}
\newcommand{\wei}[1]{\langle#1\rangle}
\newcommand{\norm}[1]{{\left\vert\kern-0.25ex\left\vert\kern-0.25ex\left\vert #1 
    \right\vert\kern-0.25ex\right\vert\kern-0.25ex\right\vert}}
\DeclareMathOperator*{\esssup}{ess\,sup}
\renewcommand{\le}{\leq}
\renewcommand{\ge}{\geq}
\allowdisplaybreaks
\begin{document}
	\title[Taming singular SDEs]{Taming singular stochastic differential equations: A numerical method}
	\date{\today}
	
\author{Khoa L{\^e} 
\and Chengcheng Ling 
}
\address{
School of Mathematics, University of Leeds, Leeds, United Kingdom.}
\email{k.le@leeds.ac.uk}
 \address{Institut f\"ur Mathematik, Universit\"at Augsburg,  Augsburg, Germany}
 \email{chengcheng.ling@uni-a.de}

	\begin{abstract}
	We consider a generic and explicit  tamed Euler--Maruyama scheme for multidimensional time-inhomogeneous stochastic differential equations with multiplicative Brownian noise. The diffusive coefficient is uniformly elliptic, H\"older continuous and weakly differentiable in the spatial variables while the drift satisfies the strict Ladyzhenskaya--Prodi--Serrin condition, as considered by Krylov and R\"ockner (2005).
	In the discrete scheme, the drift is tamed by replacing it by an approximation.
	A strong rate of convergence of the scheme is provided in terms of the approximation error of the drift in a suitable and possibly very weak topology.
	A few examples of approximating drifts are discussed in detail. The parameters of the approximating drifts can vary and---under suitable conditions---
	be fine-tuned to achieve a strong convergence rate which is arbitrarily close to the benchmark $0.5$ rate.
	The result is then applied to provide numerical solutions for stochastic transport equations with singular vector fields satisfying the aforementioned condition. 
		
		\bigskip
		

	\end{abstract}
	
	\maketitle
\section{Introduction} 
\label{sec:introduction}
	The aim of this article is to devise a numerical scheme and  obtain its strong convergence rate for stochastic differential equations (SDEs) with integrable drift coefficients and elliptic regular diffusive coefficients.
	We consider the SDE
	\begin{align}\label{sde00}
	\dif X_t=b(t,X_t)\dif t+\sigma(t,X_t) \dif B_t,\quad X_0=x_0,\quad t\in[0,1],
	\end{align}
	where $d\geq1$, $b: [0,1]\times\mR^d\rightarrow\mR^d$ is a Borel measurable function satisfying
	\begin{align}
		\int_0^1\left[\int_\Rd|b(t,x)|^pdx\right]^{\frac qp}dt<\infty \quad\text{with}\quad q,p\in[2,\infty)\quad\text{and}\quad  \frac dp+\frac2q<1,    \label{pq}
	\end{align}
	and $\sigma:[0,1]\times\Rd\to \Rd\times\Rd$ is a bounded Borel measurable function, continuous in the spatial variables and uniformly elliptic, $(B_t)_{t\geq0}$ is a $d$-dimensional standard Brownian motion defined on some complete filtered probability space $(\Omega,\cff,(\cff_t)_{t\geq 0},\mP)$ and $x_0$ is a $\cff_0$-random variable.
	With $\le1$ in place of $<1$, \eqref{pq} is known in the fluid dynamics' literature as the Ladyzhenskaya--Prodi--Serrin  condition.

	In the seminal paper \cite{MR2117951},  Krylov and R\"ockner, building upon \cite{Z,Ve}, show that \eqref{sde00} has a unique strong solution assuming that $\sigma$ is the identity matrix and $b$ satisfies \eqref{pq}.
	This result is later extended by Zhang \cite{Zhang2011} (complemented by \cite{XXZZ})\footnote{\cite[Theorem 5.1]{Zhang2011} is non-trivial whose proof is provided in \cite{XXZZ}.} for variable diffusive coefficients which are weakly differentiable, uniformly elliptic, uniformly bounded and uniformly continuous in $x$ locally uniformly in time.

	While theoretical solutions of \eqref{sde00} are well understood since \cite{MR2117951}, numerical analysis of \eqref{sde00} under condition \eqref{pq} has been an open problem.
	At the moment of writing,  we are aware of two publications on the topic.
	Jourdain and Menozzi consider in \cite{jourdain2021convergence} the case $\sigma$ is the identity matrix and show that the marginal density of a tamed Euler--Maruyama scheme with truncated drifts converges  at the rate $\frac12-\frac d{2p}-\frac1q$. Gy\"ongy and Krylov in
	\cite{gyongy2021existence} recently show that the tamed Euler--Maruyama scheme with truncated drifts converges in probability to the exact solution, albeit without any rate. 
	 Needless to say, a strong convergence rate is desirable and is of independent interest. 
	For this purpose, we consider the discrete scheme defined by
	 \begin{equation}\label{eqn.EMscheme}
		 dX_t^n=b^n(t,X^n_{k_n(t)})dt+ \sigma(t,X^n_{k_n(t)}) dB_t,\quad X_0^n=x_0^n, \quad t\in[0,1],
	 \end{equation}
	 where $x^n_0$ is a $\cff_0$-random variable and $b^n$ is an approximation of the vector field $b$ and
	 \begin{align*}
	 	k_n(t)=\frac jn \text{ whenever } \frac jn\le t<\frac {j+1}n \text{ for some integer } j\ge0.
	 \end{align*}
	 We note that \eqref{eqn.EMscheme} with the choice $b^n=b$ is the usual Euler--Maruyama scheme, which, however, is not well-behaved for a merely integrable function $b$ even when $b$ is replaced by $b\1_{(|b|<\infty)}$. This is because  the simulation for the usual Euler--Maruyama scheme may enter a neighborhood of a singularity of $b$, making the scheme unstable and uncontrollable. We thus have to tame the vector field $b$, replacing it by a suitable approximation $b^n$. Henceforth, we call \eqref{eqn.EMscheme} a \textit{tamed Euler--Maruyama scheme}.
	 The terminology is borrowed from \cite{MR2985171}, who consider a specific case of \eqref{eqn.EMscheme} to approximate SDEs with regular but super-linear drifts. 
	 The name ``tamed Euler--Maruyama'' thus should be understood in a broad sense, and in particular, \eqref{eqn.EMscheme} also includes the ``truncated Euler--Maruyama'' scheme considered in \cite{MR3370415}.

	 Natural choices for $b^n$ are the truncated vector fields
	 \begin{gather}\label{def.truncatedb}
		  	b^n_r(x)=b_r(x)\1_{(|b_r(x)|\le Cn^{\chi}\|b_r\|_{L_p(\Rd)} )},
		  	\\b^n_r(x)=b_r(x)\1_{(|b_r(x)|\le C n^\chi)},\label{def.bn2}
		\end{gather}
	 for some constants $C,\chi>0$.
	 Another practical choice is the regularized vector field
	 \begin{gather}
		b^n_r(x)=p_{1/n^\chi}\ast b_r(x),\label{def.smoothedb}
	 \end{gather}
	  where $\chi>0$, $p_t(x)$ is the Gaussian density of variance $t$ and $*$ is the spatial convolution.
	 Alternatively, multiresolution approximations by wavelet (\cite{MR1228209}) or the truncated discrete $\varphi$-transform (\cite{MR1107300}) could be used whenever desirable.

	The main results of the article, \cref{thm.main,thm.alpha} below, assert the strong convergence of \eqref{eqn.EMscheme} to \eqref{sde00} with an explicit rate under some mild regularity conditions on $\sigma$ and on the approximating drifts $b^n$. 
	When the approximating drifts take one of the forms \eqref{def.truncatedb}-\eqref{def.smoothedb}, \cref{cor.somebn} expresses the convergence rates which are proportional to $\chi$. For each form of $b^n$, a suitable validity range of $\chi$ is identified in terms of the parameters $p,q,d$. For the approximating  drifts \eqref{def.bn2} and Lipschitz $\sigma$, the parameter $\chi$ can be tuned within the interval $(0,1 /2)$ to obtain a strong convergence rate $\chi$, which is arbitrarily close to the benchmark $0.5$ rate. For the approximating drifts \eqref{def.truncatedb} and \eqref{def.smoothedb}, such sharp rate can be achieved under some restricted conditions on $p,q$.

	We expect that \cref{thm.main,thm.alpha} are useful in algorithm designs when the vector field $b$ is not explicitly available but rather arises from another analytic system which itself needs to be numerically evaluated. Such situations appear in hydrodynamic-type equations due to their fundamental connection with singular SDEs, see for instance, \cite{MR2376844,Zhang2010,MR4207449} where the SDE \eqref{sde00} is coupled with another analytic constraint on $b$.
	In such scenarios, $b^n$ does not have an explicit form, but nevertheless, \cref{thm.main,thm.alpha} could be implemented.
	While we leave this problem for future investigations, herein we focus on a  simpler application to stochastic transport equations with vector fields satisfying \eqref{pq} (see Eq. \eqref{STE}).
	While theoretical solutions for such equations have been considered in  \cite{FGP2010,FF13,NO15,BFGM19}, singularity of the  coefficients have prevented the study of numerical solutions by standard tools (\cite{C80,P02}).
	We propose in \cref{thm.STE} an explicit numerical scheme with rate  for such equations, based upon the method of characteristics.

	\smallskip
	Literatures on convergence of Euler--Maruyama schemes for SDEs is vast and expanding, for which we provide a brief and personalized overview.
	When the coefficients are continuous, convergence rates of the Euler--Maruyama scheme are well-studied.
	For Lipschitz continuous coefficients and non-trivial diffusive coefficients, the optimal strong rate of convergence is $1/2$, as shown in \cite{MR1119837,MR1617049}.
	Results on the strong rate of convergence for H\"older / Dini continuous drifts are discussed in \cite{GR,BHY,PT} and only settled recently by  Dareiotis and Gerencs\'er in \cite{DG}, who obtain  the $L_2(\Omega)$-rate $1/2- \varepsilon$, for any $\varepsilon\in(0,1/2)$, when $b$ is Dini continuous and $\sigma$ is the identity matrix. This result is extended in \cite{butkovsky2021approximation} for the case when $b$ is H\"older continuous and $\sigma$ is uniformly elliptic and twice continuously differentiable.
	For an in-depth overview and more complete lists of other contributions, see \cite{BBT,KP,KPS,MT} and the references therein.
	Results for discontinuous drifts are more sparse but are attracting attention. The case of piecewise Lipschitz drifts are considered in \cite{LS1,LS2,MY}.
	\cite{MR4246871} considers one-dimensional SDEs with additive noise and bounded measurable drifts with a positive Sobolev--Slobodecki-regularity.
	\cite{bao2020convergence} considers bounded measurable drifts with a certain Gaussian--Besov-regularity. For merely bounded measurable drifts without any regularity, the recent article \cite{DGL2021} obtains the $L_p(\Omega)$-rate $1/2- \varepsilon$, for any $p\ge2$ and $\varepsilon\in(0,1/2)$, extending the results of \cite{DG,butkovsky2021approximation}. 
	At last, we mention the work \cite{deangelis2020numerical} who consider similar tamed Euler--Maruyama schemes for one dimensional SDEs with distributional drifts. For comparison, our approach is different, our results are in a multidimensional setting and allow completely generic approximating drifts $b^n$.
	Furthermore, we emphasize that  one dimensional SDEs are more specific,  often well-posed even for distributional drifts, and usually require tailored techniques, \cite{MR777514,MR3652414}.
	This list is surely not exhaustive.

	\smallskip
	The article is organized as follows. In \cref{sec:main_results}, we state our standing assumptions and the main results. \cref{sec:preliminaries} contains auxiliary results which are collected and adapted from previous works.
	\cref{sec.BM} is pivotal and contains a case study of moment estimation for some relevant functionals of Brownian motion. While some results in this section will not be used directly to prove the main results, the section showcases our main estimates in a simpler setting. \cref{sec:regularizing_properties_of_the_discrete_paths,sec:regularizing_properties_of_the_continuum_paths}  extend the moment estimates in \cref{sec.BM} respectively to functionals of the solutions of \eqref{eqn.EMscheme} and \eqref{sde00}. The two sections contain most of the technical estimates of the paper which build up a foundation for the proofs of the main results.
	In \cref{sec:proof}, we give the proofs of \cref{thm.main,thm.alpha}, using the moment estimates from the prior sections. The application to numerical solutions for stochastic transport equations is discussed in \cref{sec.application}.
	The appendix contains maximal regularity estimates for parabolic equations with variable coefficients and distributional forcing, which are needed but independent from the main text.

\section{Main results} 
\label{sec:main_results}

 We first fix a few notation. Let $p,q\in[1,\infty]$ be some fixed parameters. $L_p(\Rd)$ and $L_p(\Omega)$ denote the Lebesgue spaces respectively on $\Rd$ and $\Omega$.
 The expectation with respect to $\mP$ is denoted by $\E$. For each $\nu\in\R$, $L_{\nu,p}(\Rd):=({\mI}-\Delta)^{-\nu/2}\big(L^{ p}(\mR^d)\big)$ is the usual Bessel  potential space on $\Rd$ equipped with the norm
$\|f\|_{L_{\nu, p}(\mathbb{R}^d)}:=\|(\mI-\Delta)^{\nu/2}f\|_{L_{ p}(\mathbb{R}^d)}$,
where  $(\mI-\Delta)^{\nu/2}f$ is defined through Fourier's  transform. $\LL^q_{\nu,p}([0,1])$ denotes the space of measurable function $f:[0,1]\to L_{\nu,p}(\Rd)$ such that $\|f\|_{\LL^q_{\nu,p}([0,1])}$ is finite. Here, for each $s,t\in[0,1]$ satisfying $s\le t$, we denote
 \begin{align*}
  	\|f\|_{\LL^q_{\nu,p}([s,t])}:=\left(\int_s^t\|f(r,\cdot)\|_{L_{\nu,p}(\Rd)}^qdr\right)^{\frac1q}
 \end{align*}
 with obvious modification when $q=\infty$.  When $\nu=0$, we simply write $\LL^q_p([0,1])$ instead of $\LL^q_{0,p}([0,1])$.
 In particular, $\LL^q_p([0,1])$ contains Borel measurable functions $f:[0,1]\times\Rd\to\R$ such that $\int_0^1\left[\int_\Rd|f(t,x)|^pdx\right]^{q/p}dt$ is finite.
 For each $\rho\in(0,1)$, $L_\rho(\Rd)$ denotes the space of all measurable functions $f$ on $\Rd$ such that $\|f\|_{L_\rho(\Rd)}:= (\int_\Rd|f(x)|^\rho dx)^{1/\rho}$ is finite. Note that in this case, $\|\cdot\|_{L_\rho(\Rd)}$ is not a norm.

 For each $\mathcal X\in\{\LL^q_{\nu,p}([0,1]),L_p(\Rd),L_p(\Omega)\}$, an $\R^m$-valued function $f=(f^1,\ldots,f^m)$ belongs to $\mathcal X$,  if all components $f^1,\ldots,f^m$ belong to $\mathcal X$, and we put $\|f\|_{\mathcal X}=\max_{i=1,\ldots,m}\{\|f^i\|_{\mathcal X}\}$. Since we only deal with either scalars or $\Rd$-valued functions and random variables, we conventionally drop the dimension of the range in the notation of the spaces $\LL^q_{\nu,p}([0,1]),L_p(\Rd),L_p(\Omega)$.
 
  Put $D_n=\{i/n:i=0,\ldots,n\}$.
  For each $S\le T$, we put $\Delta([S,T])=\{(s,t)\in [S,T]^2: s\le t\}$ and $\Delta_2([S,T])=\{(s,u,t)\in [S,T]^3:s\le u\le t\}$. We abbreviate $\Delta=\Delta([0,1])$ and $\Delta_2=\Delta_2([0,1])$.
	We say that a function $w:\Delta([S,T])\to[0,\infty)$ is a control if $w(s,u)+w(u,t)\le w(s,t)$ for every $(s,u,t)\in \Delta_2([S,T])$.
   For a $d\times d$-matrix $P$, $P^*$ denotes its transpose and  $\|P\|$ denotes its Hilbert--Schmidt norm.
   The following conditions are enforced throughout unless noted otherwise.
  \begin{customcon}{$\mathfrak{A}$}\label{con.A}
  	The diffusion coefficient $\sigma$ is a $d\times d$-matrix-valued measurable function on $[0,1]\times\Rd$ . There exists a constant $K_1\in[1,\infty)$ such that
  	for every $s\in[0,1]$ and $x\in\Rd$
  	\begin{align}\label{ellptic-con}
  		K_1^{-1}I\le (\sigma \sigma^*)(s,x)\le K_1I.
  	\end{align}
  	Furthermore, the following conditions hold.
  	\begin{enumerate}[label=$\mathbf{\arabic{*}}$., ref=$\mathbf{\arabic*}$]
  	 	\item\label{con.Aholder} There are constants $\alpha\in(0,1]$ and $K_2\in(0,\infty)$  such that for every $s\in[0,1]$ and $x,y\in\Rd$
  	 	\begin{align*}
  	 		|(\sigma \sigma^*)(s,x)-(\sigma \sigma^*)(s,y)|\le K_2|x-y|^\alpha.
  	 	\end{align*}
  	 	
  	 	\item\label{con.ASobolev} $\sigma(s,\cdot)$ is weakly differentiable for a.e. $s\in[0,1]$ and there are constants $p_0\in[2,\infty)$, $q_0\in(2,\infty]$ and $K_3\in(0,\infty)$  such that
	  	\begin{align*}
	  		\frac d{p_0}+\frac2{q_0}<1
	  		\tand\|\nabla \sigma\|_{\LL^{q_0}_{p_0}([0,1])}\le K_3.
	  	\end{align*}
  	 \end{enumerate}
  \end{customcon}
  \begin{customcon}{$\mathfrak{B}$}\label{con.B}	
  	$x_0$ belongs to $L_p(\Omega,\cff_0)$ and
  	$b$ belongs to $\LL^q_p([0,1])$ for some $p,q\in[2,\infty)$ satisfying $\frac dp+\frac2q<1$.
	For each $n$, $x^n_0$ belongs to $L_p(\Omega,\cff_0)$ and $b^n$ belongs to $\LL^q_p([0,1])\cap \LL^q_\infty([0,1])$ with $p,q$ as above.
	Furthermore, there exist finite positive constants $K_4,\theta$ and continuous controls $\{\mu^n\}_n$ such that $\sup_{n\ge1}(\|b^n\|_{\LL^q_p([0,1])}+\mu^n(0,1))\le K_4$ and
  		\begin{align}\label{con.locGK}
  		  	(1/n)^{\frac12-\frac1q}\|b^n\|_{\LL^q_\infty([s,t])}\le \mu^n(s,t)^\theta \quad \forall\ t-s\le1/n.	
  		\end{align}
  \end{customcon}
  In the above, $I$ denotes the identity matrix.
  If one replaces H\"older continuity by uniform continuity, \crefrange{con.A}{con.B} are comparable to those from \cite{Zhang2011,XXZZ}, who show strong uniqueness for \eqref{sde00}.
  Hence, hereafter, we assume that the solution to \eqref{sde00} exists and is strongly unique.\footnote{Actually the results from \cite{Zhang2011,XXZZ} are for deterministic $x_0\in\Rd$, however, can be easily extended to our case by conditioning and utilizing Markov property of Brownian motion. See also \cref{rmk.stronguniq} below.}
  Next, we define an important quantity which controls the strong convergence rate.
  \begin{definition}\label{def.rate}
  	Let $\lambda>0$ be a fixed number which is sufficiently large.
  	Let $U=(U^1,\ldots,U^d)$ where for each $h=1,\ldots,d$, $U^h$ is the solution to the following equation
  	\begin{align}\label{eqn.uKol}
  	\partial_t U^h+\sum_{i,j=1}^d\frac{1}{2}(\sigma \sigma^*)^{ij}\partial^2_{ij} U^h+b^{n,h}\cdot\nabla U^h= \lambda U^h-b^{n,h},\quad U^h(1,\cdot)=0.
  	\end{align}
  	Let $X$ be the solution to \eqref{sde00}.
  	For each $\bar p\in[1,\infty)$, we put
  	\begin{align*}
  		&\varpi_n(\bar p)=\Big\|\sup_{t\in[0,1]}\Big|\int_0^t(1+\nabla U) [b-b^n](r,X_r)dr\Big|\Big\|_{L_{\bar p}(\Omega)}.
  	\end{align*}
  \end{definition}
  In the above and hereafter, we omit the dependence of $U$ on $n$.
  Equation \eqref{eqn.uKol} arises from a Zvonkin transformation, which we postpone to \cref{sec:proof} for the details. It is known that when $\lambda$ is sufficiently large, equation \eqref{eqn.uKol} has a unique solution, see \cref{pde} below.
\begin{theorem}\label{thm.main}
Assume that \crefrange{con.A}{con.B} hold.
Let $(X_t^n)_{t\in[0,1]}$ be the solution to \eqref{eqn.EMscheme} and $(X_t)_{t\in[0,1]}$ be the solution to  SDE \eqref{sde00}. Then for any  $\bar p\in(1,p)\cap(1,\frac 2d(p\wedge p_0))
$  and any $\gamma\in(0,1)$,  there exists a finite constant $N(K_1,K_2,K_3,K_4,\alpha,p_0,q_0,p,q,d,\bar p,\gamma)$ such that
\begin{align}\label{maine}
\|\sup_{t\in[0,1]}|X_t^n-X_t|\|_{L_{\gamma\bar p}(\Omega)} \le N \left[\|x^n_0-x_0\|_{L_{\bar p}(\Omega)}+(1/n)^{\frac \alpha2}+ (1/n)^{\frac 12}\log(n) +\varpi_n(\bar p)\right].
\end{align}
\end{theorem}
Actually, by adding an exponential weight, moments up to order $p$-th can be estimated, see \cref{prop.weightedmoment} below. The condition $\bar p<\frac 2d(p\wedge p_0)$ ensures finiteness of the moments of the exponential weight and therefore deduces \eqref{maine} by an application of H\"older inequality.

 Under uniform ellipticity and H\"older regularity of $\sigma$, pathwise uniqueness for \eqref{sde00}, $p>d/\alpha$, $\lim_n b^n=b$ in $\LL^q_p([0,1])$ and the following condition
 \begin{align}\label{con.GK21}
 	\sup_{n\ge1}(1/n)^{\frac12-\frac1q}\|b^n\|_{\LL^q_\infty([0,1])}<\infty,
 \end{align}
 \cite[Theorem 2.11]{gyongy2021existence} recently shows that the tamed Euler--Maruyama scheme \eqref{eqn.EMscheme} converges in probability to the solution of \eqref{sde00}.
 It is evident that \eqref{con.GK21} implies \eqref{con.locGK} (with the choice $\mu^n(s,t)^{\frac1q}=(1/n)^{\frac12-\frac1q}\|b^n\|_{\LL^q_\infty([s,t])}$).
 However, because of the interchangeability between $t-s$ and $1/n$ in \eqref{con.locGK}, truncated vector fields with higher truncation levels, which yield better convergence rates, satisfy \eqref{con.locGK} but not \eqref{con.GK21}, see \cref{cor.somebn} below.
  Under \crefrange{con.A}{con.B}, the above result provides an upper bound for the moments of $\sup_{t\in[0,1]}|X^n_t-X_t|$ which depends on $n$ and $\varpi_n$.
  When $\lim_nb^n=b$ in $\LL^q_p([0,1])$ as in the setting of \cite{gyongy2021existence}, one can show that $\lim_n \varpi_n=0$ (cf. \cref{cor.somebn}). However, the topology of $\LL^q_p$ does not provide any explicit rate. There are, of course, many other topologies for  $\lim_n b^n=b$ so that one can actually obtain an explicit rate. The choice of a suitable topology depends on the approximating vector fields $b^n$. 
  Our next main result relates $\varpi_n$ with the convergence of $b^n$ to $b$ with respective to the topologies of $\LL^{q_1}_{p_1}$ (for some $p_1,q_1\in[1,\infty]$) and $\LL^q_{-\nu,p}$ (for some $\nu\in[0,1)$).
\begin{theorem}\label{thm.alpha}
	Assume that \crefrange{con.A}{con.B} hold.

	(i) Let $p_1,q_1\in[1,\infty]$ be such that $\frac d{p_1}+\frac2{q_1}<2$. Then for every $m\ge1$, there exists a constant $N$ depending  on $K_1$, $K_2$, $K_3$, $K_4$, $\alpha$, $p_0$, $q_0$, $p_1$, $q_1$, $p$, $q$, $d$, $m$ such that
	\begin{align}
		\varpi_n(m)\le N\|b-b^n\|_{\LL^{q_1}_{p_1}([0,1])}.
		\label{est.ap}
	\end{align}

	(ii) Assuming furthermore that $q_0=\infty$ and $\frac1p+\frac1{p_0}<1$.
	Let $\nu\in[0,1)$ be such that
	\begin{align}\label{con.nu}
		\nu<\frac 32- \frac d{2p}-\frac2q.
	\end{align}
	Then for every $\bar p\in[1,p)$, there exists a constant $N$ depending  on $K_1$, $K_2$, $K_3$, $K_4$, $\alpha$, $p_0$, $p$, $q$, $d$, $\bar p$, $\nu$ such that
	\begin{gather}
		\varpi_n(\bar p)\le N\|b-b^n\|_{\LL^q_{-\nu,p}([0,1])}.
		\label{est.anu}
	\end{gather}

	(iii) Assuming furthermore that $q_0=\infty$, $\frac1p+\frac1{p_0}<1$ and
	\begin{align}\label{con.nu1}
		\frac dp+\frac4q<1. 	
	\end{align}
	Suppose that there exists a continuous control $w_0$ on $\Delta$ and a constant $\Gamma\ge0$ such that 
	\begin{align}\label{con.bbn-1}
		\|b-b^n\|_{\LL^q_{-1,p}([s,t])}\le \Gamma w_0(s,t)^{\frac1q}
		\tand\|b-b^n\|_{\LL^q_p([s,t])}\le w_0(s,t)^{\frac1q}
	\end{align}
	for every $(s,t)\in \Delta$.
	Then for every $\bar p\in[1,p)$, there exists a constant $N$ depending  on $K_1$, $K_2$, $K_3$, $K_4$, $\alpha$, $p_0$, $p$, $q$, $d$, $\bar p$ such that
	\begin{align}
		\varpi_n(\bar p)\le N \Gamma\left(1+|\log \Gamma|\right)w_0(0,1)^{\frac1q}.
		\label{est.a-1}
	\end{align}
\end{theorem}
Using \cref{thm.main,thm.alpha}, we can derive explicit strong convergence rates for the scheme \eqref{eqn.EMscheme} when the approximating vector field $b^n$ take one of the forms \eqref{def.truncatedb}-\eqref{def.smoothedb}.
\begin{corollary}\label{cor.somebn}
	Assume that \cref{con.A} holds. Let $b,x_0,x_0^n$ be as in \cref{con.B}; and  $\bar p,\gamma$ be as in \cref{thm.main}.
	{}

	(a) Let $C>0$ and $\chi\in(0,1/2-1/q]$ be constants and define $b^n$ by \eqref{def.truncatedb}.
	Let $\rho\in(1,p]$ be a number such  that $\rho\frac d{p}+\frac 2q<2$.
	Then there exists a constant $N$ depending  on $K_1$, $K_2$, $K_3$, $K_4$, $\alpha$, $p_0$, $q_0$, $p$, $q$, $d$, $\bar p$, $\gamma$, $\rho$, $\chi$, $C$ such that 
	\begin{multline}\label{est.XXtruncated}
		\|\sup_{t\in[0,1]}|X^n_t-X_t|\|_{L_{\gamma\bar p}(\Omega)}
		\\ \le N\left[\|x_0-x^n_0\|_{L_{\bar p}(\Omega)}+ (1/n)^{\chi(\rho-1)}+ (1/n)^{\frac \alpha2}+ (1/n)^{\frac12}\log(n)\right].
	\end{multline}

	(b) Let $C>0$ and $\chi\in(0,1/2)$ be constants and define $b^n$ by \eqref{def.bn2}.
	Then there exists a constant $N$ depending  on $K_1$, $K_2$, $K_3$, $K_4$, $\alpha$, $p_0$, $q_0$, $p$, $q$, $d$, $\bar p$, $\gamma$, $\rho$, $\chi$, $C$ such that \eqref{est.XXtruncated} holds for any $\rho\in(1,p\wedge q]$ satisfying $\rho\left(\frac d{p}+\frac 2q\right)<2$.

	(c) Let $\chi\in\Big(0, \frac {p}d\left(1-\frac2q\right)\Big]$
	and define $b^n$ by \eqref{def.smoothedb}.
	Let $\nu\in(0,1)$ be any number satisfying \eqref{con.nu}.
	Assume furthermore that $q_0=\infty$ and $\frac1p+\frac1{p_0}<1$. Then there exists a constant $N$ depending  on $K_1$, $K_2$, $K_3$, $K_4$, $\alpha$, $p_0$,  $p$, $q$, $d$, $\bar p$, $\nu$, $\gamma$, $\chi$ such that
	\begin{align}\label{est.XXsmooth}
		\|\sup_{t\in[0,1]}|X^n_t-X_t|\|_{L_{\gamma\bar p}(\Omega)} \le N\left[\|x_0-x^n_0\|_{L_{\bar p}(\Omega)}+(1/n)^{\chi \frac \nu2} +(1/n)^{\frac \alpha2}+ (1/n)^{\frac12}\log(n)\right].
	\end{align}
\end{corollary}
\begin{proof}
	In view of \cref{thm.main}, it suffices to estimate $\varpi_n(\bar p)$.
	(a) It is obvious that $\|b^n\|_{\LL^q_p([0,1])}\le \|b\|_{\LL^q_p([0,1])}$. From the inequality $\|b^n_r\|_{L_\infty(\Rd)}\le C n^{\chi}\|b_r\|_{L_p(\Rd)}$, we see that $\|b^n\|_{\LL^q_\infty([s,t])}\les n^\chi\|b\|_{\LL^q_p([s,t])}$.
	It follows that for every $0\le t-s\le1/n$,
	\begin{align*}
		(1/n)^{\frac12-\frac1q}\|b^n\|_{\LL^q_\infty([s,t])}\les (1/n)^{\frac12-\frac1q- \chi}\|b\|_{\LL^q_p([s,t])}
		\les\|b\|_{\LL^q_p([s,t])},
	\end{align*}
	verifying \cref{con.B} with $\mu^n(s,t)=\|b\|_{\LL^q_p([s,t])}^q$ and $\theta=1/q$. Furthermore,
	\begin{align*}
	 	|b_r(x)-b^n_r(x)|
	 	\le|b_r(x)|\1_{(|b_r(x)|> C n^\chi\|b_r\|_{L_p(\Rd)})}
	 	\le C^{1- \rho} n^{-\chi(\rho-1)} \|b_r\|_{L_p(\Rd)}^{1- \rho}|b_r(x)|^\rho.
	\end{align*}
	The function $(r,x)\mapsto \|b_r\|_{L_p(\Rd)}^{1- \rho}|b_r(x)|^\rho$ belongs to $\LL^q_{p/\rho}([0,1])$
	and hence,
	\begin{align*}
		\|b-b^n \|_{\LL^q_{p/\rho}([0,1])}\les n^{-\chi(\rho-1)}\|b\|_{\LL^q_p([0,1])}.
	\end{align*}
	It follows from \eqref{est.ap} that $\varpi_n\les(1/n)^{\chi(\rho-1)}\|b\|_{\LL^q_p([0,1])}$. The stated estimate is then a consequence of \eqref{maine}.

	(b) For the vector field $b^n$ defined by \eqref{def.bn2}, we have $\|b^n_r\|_{L_\infty(\Rd)}\les n^\chi$ so that for every $0\le t-s\le 1/n$,
	\begin{align*}
	 	(1/n)^{\frac12-\frac1q}\|b^n\|_{\LL^q_\infty([s,t])}\les (1/n)^{\frac12-\frac1q- \chi}(t-s)^{\frac1q}
	 	\les (t-s)^{\theta}
	\end{align*}
	for any $\theta>0$ such that $\theta\le \min(1/q,1/2-\chi)$. This verifies \cref{con.B} with $\mu^n(s,t)=t-s$.
	On the other hand,
	$|b-b^n|\les n^{-\chi(\rho-1)}|b|^{\rho}$ so that $\|b-b^n\|_{\LL^{q/\rho}_{p/\rho}([0,1])}\les n^{-\chi(\rho-1)}\|b\|_{\LL^q_p([0,1])}^{\rho}$.
	It follows  from \eqref{est.ap} that $\varpi_n\les n^{-\chi(\rho-1)}\|b\|_{\LL^q_p([0,1])}^{\rho}$.

	(c) We have $\|b^n_r\|_{L_{\infty}(\Rd)}\les n^{\chi\frac d{2p}}\|b_r\|_{L_p(\Rd)}$ and hence, for every $0\le t-s\le 1/n$, 
	\begin{align*}
		(1/n)^{\frac12-\frac1q}\|b^n\|_{\LL^q_\infty([s,t])} \les (1/n)^{\frac12-\frac 1q- \chi\frac d{2p}}\|b\|_{\LL^q_p([s,t])}
		\les \|b\|_{\LL^q_p([s,t])},
	\end{align*}
	verifying condition \eqref{con.locGK} with $\mu^n(s,t)=\|b\|_{\LL^q_p([s,t])}^q$ and $\theta=1/q$.
	We also have
	\begin{align*}
		\|b-b^n\|_{\LL^q_{-\nu,p}}\les (1/n)^{\chi \nu/2 }\|b\|_{\LL^q_p}.
	\end{align*}
	Applying \eqref{est.anu}, we have $\varpi_n(\bar p)\les (1/n)^{\chi\frac \nu2}$.
\end{proof}

\begin{remark}
	Similar truncated vector fields to \eqref{def.bn2} with the values $\chi=1/2$ and $\chi=d/{(2p)}+1/q$ were considered in \cite{jourdain2021convergence}, in which a weak rate of convergence of order $\frac12-\frac d{2p}-\frac1q$ was obtained.
While \cref{cor.somebn}(b) excludes the value $\chi=1/2$, by choosing $\rho=2$, it yields the strong rate
	\begin{align*}
		\|x_0-x^n_0\|_{L_{\bar p}(\Omega)}+(1/ n)^{\chi}  +(1/n)^{\frac \alpha2}+ (1/n)^{\frac12}\log(n),
	\end{align*}
	in which $\chi$ can be as close as one desires to $1/2$.

	Similar (but different) regularized vector fields to \eqref{def.smoothedb} was considered in \cite{deangelis2020numerical} in a  different setting.
\end{remark}
\begin{remark}\label{rmk.stronguniq}
	The proof of \cref{thm.main} actually works for any adapted solution to \eqref{sde00}, see \cref{rmk.weakuniq} and \cref{sec:proof}. Consequently, \cref{thm.main} yields an alternative proof (\cite{Zhang2011,XXZZ}) of pathwise uniqueness for \eqref{sde00} under \crefrange{con.A}{con.B}.
\end{remark}
The restriction on the unit time interval in \cref{thm.main,thm.alpha} is of course artificial and it is straightforward to extend the above results on arbitrary finite time intervals. In such case, the constants in our estimates also depend on the length the time interval.
The logarithmic factor in \eqref{maine} arises from the stochastic Davie--Gr\"onwall lemma with critical exponents (see \cite{fhl1} or \cref{lem.Davie_iteration} herein).
The explicit estimation for square moments from \cite{DG} suggests that the logarithmic factor in \eqref{maine} could be improved.
Because of the role of the stochastic Davie--Gr\"onwall lemma in the study of rough/stochastic ordinary/partial differential equations (\cite{fhl1,athreya2021wellposedness,MR2377011,MR2387018}), it is an important problem to identify the sharpness of the logarithmic factor. However, we do not pursue this direction herein.

\smallskip
Let us briefly explain our general method and strategy.
Starting from \eqref{sde00} and \eqref{eqn.EMscheme}, we decompose the difference $X_t-X^n_t$ into three types of differences:
\begin{itemize}
	\item differences between functionals of $b(t,X_t)$ and $b^n(t,X_t)$,
	\item differences between functionals of $X_t$ and functionals of $X^n_t$,
	\item differences between functionals of $X^n_t$ and functionals of $X^n_{k_n(t)}$.
\end{itemize}
At this stage, our strategy aligns with the classical works \cite{MR1119837,MR1617049} for SDE's with Lipschitz coefficients. However, in order to utilize the regularizing effect of the noise in compensation for the lack of regularity of the drift, our treatments for these functionals are different and follow the recent approach of \cite{DGL2021}.
The differences of the first type can be easily estimated from above by $\varpi_n$.
For the differences of the second type, we use a Zvonkin-type transformation to show that they depend on $\sup_{t\in[0,1]}|X_t-X^n_t|$ in a Lipschitz sense.
The differences of the last type contain, for instance, the functional
\[
	\sup_{t\in[0,1]}\left|\int_0^t [b^n(s,X^n_{k_n(s)})-b^n(s,X^n_s)]ds\right|.
\]
Because $b$ and $b^n$ are not continuous (uniformly in $n$), estimation for the above functional is a challenging problem and one has to utilize the regularizing effect from the noise, an important observation made by Dareiotis and Gerencs\'er in \cite{DG}.
For these differences, we use stochastic sewing techniques---originated from \cite{MR4089788} and further extended in \cite{fhl1,SSLBanach}---to estimate them by a constant multiple of $(1/n)^{\alpha/2}+ (1/n)^{1/2}\log(n)$. 
From here, we obtain an integral inequality for the moment of $\sup_{t\in[0,1]}|X_t-X^n_t|$. An application of the stochastic Gr\"onwall inequality yields the desired estimate in \cref{thm.main}.
From this analysis, one observes that the strong rate of convergence for \eqref{eqn.EMscheme} is deduced from the rates of the estimations for the differences of the first and the last types.
The estimates for $\varpi_n$ in \cref{thm.alpha} are obtained by mean of Krylov estimates, Khasminskii estimates and stochastic sewing techniques, utilizing statistical properties of the solution to \eqref{sde00}.

 We make a few observations comparing with previous works. 
 Setting technicalities aside, our proof of \cref{thm.main} follows the approach of \cite{DGL2021}; and similar to \cite{butkovsky2021approximation}, we also apply stochastic sewing techniques to obtain moment estimates for the differences of the last type.
 However, the works \cite{butkovsky2021approximation,DGL2021} crucially rely on the fact that the drifts are either continuous or bounded, which is not available under \crefrange{con.A}{con.B}.
 In particular, the stochastic lemmas in \cite{butkovsky2021approximation,DGL2021} cannot be applied under \crefrange{con.A}{con.B} because the resulting H\"older exponents are strictly below $1/2$; and even if  control functions were  employed, one would end up with a regularity exponent of exactly $1/2$ (cf. \cref{prop.B1,prop.fXX2}).\footnote{We recall that an exponent of $1/2+\varepsilon$ is required in these stochastic sewing lemmas.}
 In other words, the situations considered herein are at the border line and are critical to a certain extent.
 To successfully adapt the method above to the current setting, to overcome criticality and to remove the $\varepsilon$-loss in the obtained rate, we have benefited from the recent stochastic Davie--Gr\"onwall lemma with critical exponents from \cite{fhl1},  the analysis for singular paths from \cite{bellingeri2021singular} and novel usage of control functions inspired by Lyon's theory of rough paths \cite{MR1654527}. 
 To the authors' knowledge, these tools, which are developed within rough path theory, have not been utilized previously in stochastic numerics. 
 Lastly, in order to verify the hypotheses for stochastic sewing and of independent interests, we have obtained some new and improved analytic estimates (\cite{Kim,XXZZ,jourdain2021convergence,gyongy2021existence,MR2735377,bao2020convergence}) for the probability laws of the solutions to the discrete scheme \eqref{eqn.EMscheme} and to equation \eqref{sde00} (see \cref{sec:regularizing_properties_of_the_discrete_paths,sec:regularizing_properties_of_the_continuum_paths}). 
 
 While this article was under review, progress had been made in improving our main results. Namely, \cite{le2022quantitative} removes the  moment restrictions in \cref{thm.main,thm.alpha} and consequently in \cref{cor.somebn}. This is accomplished by taking advantage of two other recent developments (after the first appearance of the current article). One is the stability results from Galeati and the second named author in \cite{MR4546633}. The other is the John--Nirenberg inequality for stochastic processes of bounded mean oscillations as discussed in \cite{le2022quantitative}.

 \smallskip
 \noindent\textbf{Convention.} Whenever convenience, we place temporal variables into subscript right after the function, e.g. $f_t(x)=f(t,x)$. The relation $A\les B$ means that $A\le CB$ for some finite constant $C\ge0$. The implicit constants $C$ may change from one inequality to another and their values may depend on other parameters which are clear from the context. We will also make use of  Einstein's convention of summation over repeated indices. 

\section{Preliminaries} 
\label{sec:preliminaries}
	In the current section, we collect and enhance some relevant results which appear separately in previous works from various authors.
	These results form a useful toolbox which is used in later sections to prove our main results.

	For any one-parameter process $t\mapsto Y_t$ and any two-parameter process $(s,t)\mapsto A_{s,t}$, we denote $\delta Y_{s,t}=Y_t-Y_s$ and $\delta A_{s,u,t}=A_{s,t}-A_{s,u}-A_{u,t}$ for every $s\le u\le t$. We say that $Y$ (resp. $A$) is $L_m$-integrable if $\|Y_t\|_{L_m(\Omega)}$ (resp. $\|A_{s,t}\|_{L_m(\Omega)}$) is finite for each $t$ (resp. $(s,t)$); we say $A$ is adapted if $A_{s,t}$ is $\cff_t$-measurable whenever $s\le t$.
	Let $v\in[0,1]$ and let $\mP|\cff_v$ be the probability measure conditioned on $\cff_v$. We denote by $L_p(\Omega|\cff_v)$ the space of random variables $Z$ such that
	\[
		\|Z\|_{L_p(\Omega|\cff_v)}:=\esssup_{\omega} [\E(|Z|^p|\cff_v)]^{1/p}<\infty.
	\]

	The advantages of considering the conditional moment norms over the usual moment norms are summarized in the following result, which is implicit in \cite{DGL2021,fhl1}.
	\begin{lemma}\label{lem.lenglart}
		Let $\caa=(\caa_t)_{t\in[0,1]}$ be a continuous adapted stochastic process and let $p,N\in(0,\infty)$ be some fixed constants. Assume that $\caa_0=0$ and
		\begin{align*}
		 	\sup_{0\le s\le t\le1}\|\delta\caa_{s,t}\|_{L_p(\Omega|\cff_s)}\le N.
		\end{align*}
		Then the following statements hold.
		\begin{enumerate}[(i)]
			\item\label{item.Atau} There exists a constant $c(p)$ such that  $\|\caa_\tau\|_{L_p(\Omega)}\le c(p) N$ for any stopping time $\tau\le1$.
			\item\label{item.Asup} For every $\bar p\in(0,p)$, there exists a constant $c(\bar p,p)$ such that
			\begin{align*}
				\|\sup_{t\in[0,1]}|\caa_{t}|\|_{L_{\bar p}(\Omega)}\le c(\bar p,p)N.
			\end{align*}
		\end{enumerate}
	\end{lemma}
	\begin{proof}
		Let $\tau$ be a stopping time taking finitely many values $\{t_j\}\subset[0,1]$. By assumption, we have
		\begin{align*}
			\E(|\delta\caa_{\tau,1}|^p)
			=\E\sum_{j}\1_{(\tau=t_j)}\E (|\delta\caa_{t_j,1}|^p|\cff_{t_j})
			\le \E\sum_{j}\1_{(\tau=t_j)}N^p
			\le N^p.
		\end{align*}
		Using the elementary inequality $(a+b)^p\les a^p+b^p$, we have
		\begin{align*}
			\|\caa_\tau\|_{L_p(\Omega)}\les\|\delta\caa_{\tau,1}\|_{L_p(\Omega)}+\|\caa_1\|_{L_p(\Omega)}\les N.
		\end{align*}
		By approximations and continuity of $\caa$, the above inequality also holds for all stopping times $\tau\le1$. This shows \ref{item.Atau}. Part \ref{item.Asup} is a consequence of part \ref{item.Atau} and Lenglart inequality. 
	\end{proof}
	The next result is a variant of the stochastic Davie--Gr\"onwall lemma from \cite{fhl1} and is closely related to the stochastic sewing lemmas from \cite{MR4089788,SSLBanach}. 
	\begin{lemma}[Stochastic sewing]\label{lem.Davie_iteration}
		Let $\varepsilon>0$; $v,S,T,C_1,C_2,C_3,\Gamma_1,\Gamma_2\ge0$ be fixed numbers such that $0\le v<S<T$. Let $w$ be a deterministic control on $\Delta([S,T])$ which is continuous. 
	    Let $J$ be a $L_m$-integrable adapted process indexed by $\Delta([S,T])$ such that
			\begin{gather}
				\|J_{s,t}\|_{L_m(\Omega|\cff_v)}\le C_2  w(s,t)^{\frac12+\varepsilon}
				\,,\quad\|\E_sJ_{s,t}\|_{L_m(\Omega|\cff_v)}\le C_1 w(s,t)^{1+\varepsilon},\label{con.sll1}
				\\\|\delta J_{s,u,t}\|_{L_m(\Omega|\cff_v)}\le \Gamma_2w(s,t)^{\frac12}+C_3 \Gamma_2w(s,t)^{\frac12+\varepsilon} \label{con.dJ}
				\\\shortintertext{and}
				\label{con.EdJ}
				 \|\E_s \delta J_{s,u,t}\|_{L_m(\Omega|\cff_v)}\le\Gamma_1w(s,t)^{1+\varepsilon}
			\end{gather}
			for every $(s,u,t)$ in $\Delta_2([S,T])$.
			Then there exists a constant $N=N(\varepsilon,m)$, in particular independent from $\Gamma_1,\Gamma_2,C_1,C_2,S,T,v$ and $w$,
			 such that for every $(s,t)\in \Delta([S,T])$
			\begin{align}
				\|J_{s,t}\|_{L_m(\Omega|\cff_v)}&\le N\Gamma_2\left[(1+|\log \Gamma_2|)w(s,t)^{\frac12}+C_1w(s,t)^{1+\varepsilon}+(C_2+C_3)w(s,t)^{\frac12+\varepsilon}\right]
				\nonumber\\&\quad+ N\Gamma_1w(s,t)^{1+\varepsilon}.
				\label{est.SDavie}
			\end{align}
	\end{lemma}
	\begin{proof}
		For each $(s,t)\in \Delta([S,T])$, define 
		\begin{align*}
			u=\inf\{r\in[s,t]:w(s,r)\ge\frac12 w(s,t)\}
		\end{align*}
		and call $u$ the $w$-midpoint of $[s,t]$. Since $t$ trivially belongs to the set defining $u$ above, such a point always exists and uniquely defined. If $u$ is a $w$-midpoint of $[s,t]$, then it follows from continuity of $w$ that
		\begin{align*}
			w(s,u)\le \frac12 w(s,t)
			\tand 
			w(u,t)\le\frac12 w(s,t).
		\end{align*}
		See \cite{SSLBanach} for more detail.
		For convenience, we denote $(s|t)$ for the $w$-midpoint of $[s,t]$.

		Let $(s,t)$ be in $\Delta([S,T])$. 
		Define $d^0_0(s,t)=s$ and $d^0_1(s,t)=t$. For each integers $h\ge0$ and $i=0,\dots,2^{h+1}$, we set $d^{h+1}_i(s,t)=d^h_{i/2}(s,t)$ if $i$ is even and $d^{h+1}_i(s,t)$ equal to the $w$-midpoint of $[d^h_{(i-1)/2}(s,t),d^h_{(i+1)/2}(s,t)]$ if $i$ is odd.
		Set $D^h_w(s,t):=\{d^h_i(s,t)\}_{i=0}^{2^h}$ for each $h\ge0$. 
		It is  readily checked that for every integers $h\ge0$ and $i=0,\dots,2^h-1$, we have
		\begin{align}
			&D^h_w(s,t)\subset D^{h+1}_w(s,t),
			\\&[d^h_i(s,t),d^h_{i+1}(s,t)]=[d^{h+1}_{2i}(s,t),d^{h+1}_{2i+1}(s,t)]\cup[d^{h+1}_{2i+1}(s,t),d^{h+1}_{2i+2}(s,t)],
			\label{id.nested}
			\\&w(d^{h}_{i}(s,t),d^{h}_{i+1}(s,t))\le 2^{-h} w(s,t).\label{est.wdh}
		\end{align}
		
		Herein, we abbreviate $\|\cdot\|$ for $\|\cdot\|_{L_m(\Omega|\cff_v)}$. The implicit constants below only depend on $\varepsilon$ and $m$.
		By triangle inequality, we have
	    \begin{align*}
	      \|J_{s,t}\|
	      \le
	      \big\|\sum_{[u,v]\in D^h_w(s,t)}J_{u,v} \big\|+\big\|J_{s,t}-\sum_{[u,v]\in D^h_w(s,t)}J_{u,v}\big\|.
	    \end{align*}
	    We estimate the first term using conditional BDG inequality (\cite[Section 2]{SSLBanach}), condition \eqref{con.sll1} and \eqref{est.wdh},
	     \begin{align*}
	      \|\sum_{[u,v]\in D^h_w(s,t)}J_{u,v}\|
	      &\lesssim\sum_{[u,v]\in D^h_w(s,t)}\|\E_uJ_{u,v}\|+\left(\sum_{[u,v]\in D^h_w(s,t)}\|J_{u,v}\|^2\right)^{1/2}
	     \\& \lesssim C_1  2^{- h \varepsilon}w(s,t)^{1+\varepsilon}+C_2 2^{-h \varepsilon}w(s,t)^{\frac12+\varepsilon}.
	    \end{align*}		
	    For the second term, we derive from \eqref{id.nested} (cf. \cite[Lemma 3.6]{SSLBanach}) and conditional BDG inequality that for $ h\ge1 $ 
	    \begin{align*}
	    	&\|J_{s,t}-\sum_{[u,v]\in D^h_w(s,t)}J_{u,v}\|
	    	=\|\sum_{k=0}^{h-1}\sum_{[u,v]\in D^k_w(s,t)}\delta J_{u,(u|v),v}\|
	    	\\&\lesssim\sum_{k=0}^{h-1}\sum_{[u,v]\in D^k_w(s,t)}\|\E_u\delta J_{u,(u|v),v}\| +\sum_{k=0}^{h-1}\left(\sum_{[u,v]\in D^k_w(s,t)}\|\delta J_{u,(u|v),v}\|^2 \right)^{1/2}.
	    \end{align*}
	    Applying \eqref{con.EdJ}, \eqref{con.dJ} and \eqref{est.wdh}, we have
	    \begin{align*}
	       \sum_{[u,v]\in D^k_w(s,t)}\|\E_u\delta J_{u,(u|v),v}\|
	      &\lesssim  2^{-k \varepsilon}\Gamma_1w(s,t)^{1+ \varepsilon}
	    \end{align*}
	    and
	    \begin{align*}
	    	\left(\sum_{[u,v]\in D^k_w(s,t)}\|\delta J_{u,(u|v),v}\|^2\right)^{\frac12}
	    	\lesssim  \Gamma_2w(s,t)^{\frac12}+2^{-k \varepsilon}C_3 \Gamma_2 w(s,t)^{\frac12+\varepsilon}.
	    \end{align*}
	    Summing in $k$, we have
	    \begin{equation*}
	      \big\|J_{s,t}-\sum_{[u,v]\in D^h_w(s,t)}J_{u,v}\big\|\lesssim \Gamma_1w(s,t)^{1+\varepsilon}+h\Gamma_2w(s,t)^{\frac12}+C_3 \Gamma_2 w(s,t)^{\frac12+\varepsilon}.
	    \end{equation*}
	    Combining the previous estimates, we have shown that for every integer $h\geq1$ and each  $(s,t)\in \Delta$
    	\begin{multline}\label{tmp.Jst}
    		\|J_{s,t}\|\lesssim   2^{- h \varepsilon}\left[C_1w(s,t)^{1+\varepsilon}+C_2 w(s,t)^{\frac12+\varepsilon}\right]
			\\+\Gamma_1w(s,t)^{1+\varepsilon}+h\Gamma_2w(s,t)^{\frac12}+C_3\Gamma_2w(s,t)^{\frac12+\varepsilon}\,.
    	\end{multline}
    	If $\Gamma_2\ge1$, we choose $h=1$  while if $\Gamma_2<1$, we choose $h$ such that $2^{-h \varepsilon}\approx \Gamma_2$.
		In both cases, we obtain \eqref{est.SDavie} from \eqref{tmp.Jst}.
	\end{proof}
	Some controls which are relevant for our purpose are given below.
	\begin{example}\label{ex.control} (a)  For any  $\phi\in L_q([0,1])$, $q\in[1,\infty)$, $w(s,t)=\|\phi\|_{L_q([s,t])}^q$ is a continuous control on $\Delta([0,1])$.
	(b) For any $\nu\ge0$, $w(s,t)=s^{-\nu}(t-s)$ is a continuous control on $\Delta([S,T])$ for any $0<S\le T$. 
	(c) For any controls $w_1,w_2$ and any number $\theta\in[0,1]$, $w=w^\theta_1 w^{1- \theta}_2$ is another control. For further examples and basic properties of controls, we refer to \cite[Chapter 5]{MR2604669}.
	\end{example}

	The following result is an excerpt from \cite[Lemma 2.3]{bellingeri2021singular}.
	\begin{lemma}\label{lem.Bellising}
		Let $(\mathcal E,\|\cdot\|)$ be a normed vector space, $s_{-1},\tau_i,\eta_i\in[0,1]$, $i=1,\ldots,h$ be fixed numbers and let $Y:(0,1]\to \mathcal E$ be a function such that
		\begin{align}\label{con.Ysing}
			\|Y_t-Y_s\|\le \sum_{i=1}^h C_is^{-\eta_i}(t-s)^{\tau_i}\quad\forall s_{-1}\le s\le t\le1, s\neq 0
		\end{align}
		for some constant $C_1,\ldots,C_h\ge0$.
		Assume that $\tau_i-\eta_i>0$ for each $i$. Then
		\begin{align*}
			\|Y_t-Y_s\|\le\sum_{i=1}^h(1-2^{\eta_i- \tau_i})^{-1} C_i(t-s)^{\tau_i- \eta_i}\quad\forall s_{-1}\le s\le t\le1, s\neq0.
		\end{align*}	
	\end{lemma}
	\begin{proof}
		Observe that \eqref{con.Ysing} implies that $Y$ is continuous on $[s_{-1},1]\setminus\{0\}$.
		We fix $s_{-1}\le s<t\le1$, $s\neq0$, and put $s_n=s+(t-s)2^{-n}$ for each integer $n\ge0$. By continuity and triangle inequality, we have
		\begin{align*}
			\|Y_t-Y_s\|
			\le \sum_{n=0}^\infty\|Y_{s_n}-Y_{s_{n+1}}\|
			\le \sum_{n=0}^\infty\sum_{i=1}^h C_is_{n+1}^{-\eta_i}(s_n-s_{n+1})^{\tau_i}.
		\end{align*}
		Note that $s_{n+1}\ge (t-s)2^{-n-1}$ and $s_n-s_{n+1}=(t-s)2^{-n-1}$. Hence, from the previous estimate, we have
		\begin{align*}
			\|Y_t-Y_s\|\le\sum_{i=1}^h\sum_{n=0}^{\infty}C_i(t-s)^{\tau_i- \eta_i}2^{-(n+1)(\tau_i- \eta_i)}.
		\end{align*}
		Because $\sum_{n=0}^{\infty}2^{-(n+1)(\tau_i- \eta_i)}\le (1-2^{\eta_i- \tau_i})^{-1}$ for each $i$, this yields the stated estimate.	
	\end{proof}
The following result is the Khasminskii's lemma\footnote{This result goes back at least to the paper \cite{MR123373} of Khasminskii, although in a less general form and with a smallness condition, then rediscovered without the smallness condition by Portenko \cite{MR0375483}, who considered \ref{con.Khasgamma} with $w(s,t)=t-s$. The general version here is based on \cite{MR1104660}. For a bit of history, see \cite[pg. 214]{MR644024}.} enhanced with some quantitative estimates.
\begin{lemma}[Quantitative Khasminskii's lemma]\label{lem.Khasminski}
Let $S,T$ be such that $0\le S\le T$ and let $\left\{\beta(t)\right\}_{t\in[S,T]}$ be a nonnegative measurable $(\mathcal{F}_t)$-adapted process. Assume that for all $S\leq s\leq t\leq T$,
\begin{align}\label{con.khas}
	\left\|\int_s^t \beta(r) dr\right\|_{L_1(\Omega|\cff_s)}\le \rho(s,t),
\end{align}
where $(s,t)\mapsto \rho(s,t)$ is a nonrandom function on $\Delta([S,T])$  satisfying the following conditions:
\begin{enumerate}[(i)]
	\item $\rho(t_1,t_2)\leq\rho(t_3,t_4)$ if $(t_1,t_2)\subset(t_3,t_4)$,
	\item $\lim_{h\downarrow 0}\sup_{S\leq s<t\leq T,|t-s|\leq h} \rho(s,t)=\kappa$, $\kappa\geq0$.
\end{enumerate}
Then for any real $\lambda<\kappa^{-1}$, (if $\kappa=0$, then $\kappa^{-1}=\infty$), and any integer $m\ge1$
\begin{gather*}
	\E\exp\Big(\lambda\int_S^T\beta(r)dr\Big)<\infty
	\tand
	\left\|\int_S^T \beta(r)dr\right\|_{L_m(\Omega)}\le  (m!)^{\frac1m}\rho(S,T).
\end{gather*}

Suppose additionally that there exist $\gamma>0$ and a continuous control $w$ on $\Delta([S,T])$ such that
\begin{enumerate}[resume*]
	\item\label{con.Khasgamma} $\rho(s,t)\le  w(s,t)^\gamma$ for each $(s,t)\in \Delta([S,T])$.
\end{enumerate}
Then for every $\lambda>0$,
\begin{align*}
	\E\exp\Big(\lambda\int_S^T\beta(r)dr\Big)\leq 2^{1+(2 \lambda )^{1/\gamma}w(S,T)}.
\end{align*}
\end{lemma}
\begin{proof}
	The former statement is an excerpt from {\cite[pg. 1 Lemma 1.1.]{MR1104660}}, which gives the following estimate
	\begin{align}\label{est.Portenko}
		\E\exp\left(\lambda\int_S^T \beta(r)dr\right)\le\prod_{k=1}^n(1- \lambda \rho(t_{k-1},t_k))^{-1}.
	\end{align}
	In the above, $S=t_0<t_1<\ldots<t_n=T$ are chosen so that $\sup_{k=1,\ldots,n}\lambda \rho(t_{k-1},t_k)<1$.

	To obtain the estimate in $L_m(\Omega)$-norm, we apply Tonelli theorem and the assumption to see that
	\begin{align*}
		\E\left(\int_S^T \beta(r)dr\right)^m
		&=m!\E\int_{S<r_1<\ldots<r_m<T}\beta(r_1)\ldots \beta(r_m)dr_1\ldots dr_m
		\\&\le m! \rho(0,T)\E\int_{S<r_1<\ldots<r_{m-1}<T}\beta(r_1)\ldots \beta(r_{m-1})dr_1\ldots dr_{m-1}.
	\end{align*}
	Iterating the above inequality, we obtain the stated estimate for $\|\int_S^T \beta(r)dr\|_{L_m(\Omega)}$.

	Under the additional condition \ref{con.Khasgamma}, we can choose $t_0=S$ and for each $k\ge1$,
	\begin{align*}
		t_k=\sup\{t\in [t_{k-1},T]: \lambda  w(t_{k-1},t)^\gamma\le 1/2\}.
	\end{align*}
	With this choice, we have $\lambda  w(t_{k-1},t_k)^\gamma=1/2$ for $k=1,\ldots,n-1$ and $\lambda  w(t_{n-1},t_n)^\gamma\le 1/2$. By definition of controls, we have
	\begin{align*}
		\frac{n-1}{(2 \lambda )^{\frac 1 \gamma}}\le \sum_{k=1}^n w(t_{k-1},t_k)\le w(S,T),
	\end{align*}
	which yields $n\le 1+ (2 \lambda )^{1/\gamma}w(S,T)$. Hence, from \eqref{est.Portenko}, we have
	\begin{align*}
		\E\exp\left(\lambda\int_S^T \beta(r)dr\right)\le 2^n\le 2^{1+(2 \lambda )^{1/\gamma}w(S,T)},
	\end{align*}
	completing the proof.	
\end{proof}
\begin{remark}\label{rmk.khas12}
	In the setting of \cref{lem.Khasminski}, if for each $(s,t)\in \Delta([S,T])$, $\rho(s,t)\le w_1(s,t)^{\gamma_1}+w_2(s,t)^{\gamma_2}$ for some continuous controls $w_1,w_2$ and some constants $0<\gamma_1\le \gamma_2$. Then we have
	\begin{align}\label{est.Khas12}
		\E\exp\left(\lambda\int_S^T \beta(r)dr\right)\le 2^{1+(4 \lambda)^{1/\gamma_1}\left(w_1(S,T)+w_2(S,T)^{\gamma_2/\gamma_1}\right)}.
	\end{align}
	Indeed, the function $w=w_1+w_2^{\gamma_2/\gamma_1}$ is a control (see \cite[Excersice 1.10]{MR2604669}) and we have $\rho(s,t)\le 2w(s,t)^{\gamma_1}$. Then \cref{lem.Khasminski}(iii) implies \eqref{est.Khas12}.
\end{remark}
\begin{remark}
	In \cref{lem.Khasminski}, we can assume without loss of generality that $\gamma\le 1$---for otherwise, condition \eqref{con.khas} implies the trivial identification $\beta\equiv0$. Furthermore, \cref{lem.Khasminski}(iii) implies that for every $\kappa>0$ and every $\rho\in(0,\frac1{1- \gamma})$, with $\frac1{1- \gamma}=\infty$ if $\gamma=1$, we have 
	\begin{align*}
		\E \exp\left(\kappa\left(\int_S^T \beta(r)dr\right)^\rho\right)<\infty.
	\end{align*}
	This follows from the same argument used in \cref{lem.expA} below.	
\end{remark}

	The next result is a kind of stochastic Gr\"onwall inequality, which is of independent interest.

	\begin{lemma}[Stochastic Gr\"onwall inequality]\label{lem.SGr\"onwall}
		Let $\xi_t,V_t$ be nonnegative nondecreasing processes,
		let $A_t$ be a continuous nondecreasing $\cff_t$-adapted process with $A_0=0$, and let $M_t$ be $\cff_t$-local martingale with $M_0=0$. Suppose that there exists a constant $\theta\in(0,\infty)$ such that with probability one,
		\begin{align}\label{Gron}
		\xi_t\leq \left(\int^t_0\xi_s^{1/ \theta}\dif A_s\right)^\theta+M_t+V_t,\quad \forall t\geq 0.
		\end{align}
		Then for any bounded stopping time $\tau$, we have
		\begin{gather}\label{est.gr\"onwall<1}
		 	\E 2^{-2^{1/\theta}A_\tau}\xi_\tau\le 2\E V_\tau \quad\text{when} \quad\theta\le1
		 	\\\shortintertext{and}
		 	\label{est.gr\"onwall>1}
		 	\E 2^{-2A_\tau^\theta}\xi_\tau\le 2\E V_\tau \quad\text{when} \quad\theta>1.
		\end{gather}		  
	\end{lemma}
	\begin{proof}
		We put $G=M+V$ and consider two cases.

		\textit{Case 1: when $\theta\le1$}. Define
		\begin{align*}
			\bar \xi_t=\left(\int^t_0\xi_s^{1/\theta}\dif A_s\right)^{\theta}+G_t
			\quad\text{so that}\quad 0\le \xi_t\le \bar \xi_t.
		\end{align*}
		We assume first that $M$ is a uniformly integrable martingale.
		For any $t\ge s\ge0$, we have
		\begin{align*}
			\delta \bar \xi_{s,t}
			=\left(\int^t_0\xi_r^{1/ \theta}\dif A_r\right)^{\theta}-\left(\int^s_0\xi_r^{1/ \theta}\dif A_r\right)^{\theta}+\delta G_{s,t}.
		\end{align*}
		We use the inequality $a^\theta-b^\theta\le (a-b)^\theta$ (valid for any $a\ge b\ge0$) to obtain from the previous identity that
		\begin{align}
			\delta \bar \xi_{s,t}
			\le\left(\int^t_s\xi_r^{1/ \theta}\dif A_r\right)^{\theta}+\delta G_{s,t}.
			\label{tmp.0926barxi}
		\end{align}
		
		Define $t_0=0$ and for each integer $j\ge1$, the stopping time
		\begin{align*}
		 	t_j=\inf\{t>t_{j-1}\ :\ A_t-A_{t_j-1}\ge2^{-1/\theta} \}.
		\end{align*}
		Let $j\ge1$ be fixed. For every $t\in[t_{j-1},t_j]$, we derive from \eqref{tmp.0926barxi} that
		\begin{align*}
			\delta\bar  \xi_{t_{j-1},t}
			\le \frac12\xi_{t}+\delta G_{t_{j-1},t}
			\le \frac12\bar\xi_{t}+\delta G_{t_{j-1},t}
		\end{align*}
		which yields $\bar \xi_{t}\le 2 \bar \xi_{t_{j-1}}+2 \delta G_{t_{j-1},t}$.
		By iteration and the fact that $\bar \xi_0=V_0$, we have
		\begin{align}\label{tmp.0926xitau}
			\bar \xi_{t}\le 2^jV_0+ \sum_{i=1}^{j}2^{j-i+1}\delta G_{t_{i-1},t_{i}\wedge t}, \quad \forall t\in[t_{j-1},t_j].
		\end{align}

		Next, let $\tau$ be a bounded stopping time and let $N$ be an (random) integer such that $t_N>\tau$. We have
		\begin{align*}
			2^{-2^{1/\theta}A_\tau}\bar \xi_\tau
			&=\sum_{j=1}^N\1_{[t_{j-1},t_j)}(\tau)2^{-2^{1/\theta}A_\tau}\bar \xi_\tau
			\le\sum_{j=1}^N\1_{[t_{j-1},t_j)}(\tau)2^{-(j-1)}\bar \xi_{\tau}.
		\end{align*}
		Using \eqref{tmp.0926xitau}, we have
		\begin{align*}
			2^{-2^{1/\theta}A_\tau}\bar \xi_\tau
			&\le2\sum_{j=1}^N \1_{[t_{j-1},t_j)}(\tau)\left(V_0+\sum_{i=1}^j2^{1-i}\delta G_{t_{i-1},\tau\wedge t_i}\right)
			\\&=2V_0+2\sum_{i=1}^N\sum_{j=i}^N\1_{[t_{j-1},t_j)}(\tau)2^{1-i}\delta G_{t_{i-1},\tau\wedge t_i}
			\\&=2V_0+2\sum_{i=1}^N\1_{[t_{i-1},t_N)}(\tau)2^{1-i}\delta G_{t_{i-1},\tau\wedge t_i},
		\end{align*}
		which is rewritten as
		\begin{multline}\label{tmp.0927Axi}
			2^{-2^{1/\theta}A_\tau}\bar \xi_\tau
			\le2V_0+2\sum_{i=1}^\infty\1_{[t_{i-1},\infty)}(\tau)2^{1-i}\delta V_{\tau\wedge t_{i-1},\tau\wedge t_i}
			\\+2\sum_{i=1}^\infty\1_{[t_{i-1},\infty)}(\tau)2^{1-i}\delta M_{\tau\wedge t_{i-1},\tau\wedge t_i}.
		\end{multline}
		By martingale property, boundedness of $\tau$ and uniform integrability, $\E | \delta M_{\tau\wedge t_{i-1},\tau\wedge t_i}|\le \E|M_{\tau\wedge t_{i-1}}|+\E|M_{\tau\wedge t_i}|\le 2\sup_{t\ge0}\E|M_t|<\infty$. 
		Hence, by Fubini theorem and martingale property,
		\begin{align*}
			\E \sum_{i=1}^\infty\1_{[t_{i-1},\infty)}(\tau)2^{1-i}\delta M_{\tau\wedge t_{i-1},\tau\wedge t_i}
			=\E \sum_{i=1}^\infty\1_{[t_{i-1},\infty)}(\tau)2^{1-i}\E(\delta M_{\tau\wedge t_{i-1},\tau\wedge t_i}|\cff_{\tau\wedge t_{i-1}})=0.
		\end{align*}
		Taking expectation in \eqref{tmp.0927Axi} gives
		\begin{align*}
			\E 2^{-2^{1/\theta}A_\tau}\bar \xi_\tau
			&\le \E\left(2 V_0+2\sum_{i=1}^\infty\1_{[t_{i-1},\infty)}(\tau)2^{1-i}\delta V_{\tau\wedge t_{i-1},\tau\wedge t_i} \right)
			\\&\le \E\left(2 V_0+2\sum_{i=1}^N\delta V_{\tau\wedge t_{i-1},\tau\wedge t_i} \right)=2\E V_\tau.
		\end{align*}
		In the general case when $M$ is a local martingale, let $\{\tau_n\}$ be a sequence of increasing stopping times such that $\lim_n \tau_n=\infty$ a.s. and for each $n$, $M_{\tau_n\wedge\cdot}$ is a uniformly integrable martingale. For a bounded stopping time $\tau$, the previous case implies that 
		\begin{align*}
			\E 2^{-2^{1/\theta}A_{\tau\wedge \tau_n}}\bar \xi_{\tau\wedge \tau_n}
			&\le2\E V_{\tau\wedge \tau_n}.
		\end{align*}
		Sending $n\to\infty$ yields \eqref{est.gr\"onwall<1}.

		\textit{Case 2: when $\theta>1$}. Using H\"older inequality and integration by parts
		\begin{align*}
			\left(\int_0^t \xi^{1/\theta} dA\right)^{\theta}\le \left(\int_0^t \xi dA \right)A_t^{\theta-1}
			=\int_0^t\int_0^s \xi dA (\theta-1)A_s^{\theta-2}dA_s+\int_0^t A^{\theta-1}\xi dA.
		\end{align*}		
		By monotonicity, $\int_0^s \xi dA \le \xi_s A_s$ so that
		\begin{align*}
		 	\int_0^t\int_0^s \xi dA (\theta-1)A_s^{\theta-2}dA_s\le \int_0^t \xi_s (\theta-1)A_s^{\theta-1}dA_s.
		 \end{align*}
		  Hence, we have
		  \begin{align*}
		 	\left(\int_0^t \xi^{1/\theta} dA\right)^{\theta}\le \int_0^t \xi_s \theta A_s^{\theta-1}dA_s=\int_0^t \xi_s dA^{\theta}_s.
		  \end{align*}
		  Together with \eqref{Gron}, we have
		  \begin{align*}
		  	\xi_t\le \int_0^t \xi dA^\theta+M_t+V_t, \quad \forall t\ge0.
		  \end{align*}
		  Using the result from the previous case, we obtain \eqref{est.gr\"onwall>1}.	
	\end{proof}
	\begin{remark}
		Stochastic Gr\"onwall inequality is useful in applications to obtain moment estimates for solutions to SDEs. Starting from \cite{MR3078830}, there have been several extensions,  for instance \cite{MR4260476,MR4079431,MR4165506}. The setting of \cref{lem.SGr\"onwall} is similar to that of \cite{MR4165506}, however, our result comes with less stringent hypotheses and stronger conclusions.  In particular,estimates \eqref{est.gr\"onwall<1} and \eqref{est.gr\"onwall>1} hold for any $\theta\in(0,\infty)$ and do not depend on the quadratic variation of the martingale part.
	\end{remark}

	\begin{lemma}\label{lem.intkn} Let $\varepsilon>0$, $s\in D_n$ and $r>s$. Then
		\begin{align}
			&\int_s^r(r- k_n(\theta))^{-1- \varepsilon}d \theta\le N_\varepsilon [\min(r-s,1/n)]^{-\varepsilon} +\1_{(r\notin D_n)}(r-k_n(r))^{-\varepsilon},
			\label{est.knminus}
			\\&\int_s^r(r- k_n(\theta))^{-1}d \theta\le \log(n(k_n(r)-s)) +2,
			\label{est.kn0}
			\\&\int_s^r(r- k_n(\theta))^{-1+\varepsilon}d \theta\le N_\varepsilon(r-s)^{\varepsilon}.
			\label{est.knplus}
		\end{align}		
	\end{lemma}
	\begin{proof}
		If $r-s\le1/n$, we have
		\begin{align*}
			\int_s^r(r- k_n(\theta))^{-1- \varepsilon}d \theta=\int_s^r(r- s)^{-1- \varepsilon}d \theta=(r-s)^{-\varepsilon}.
		\end{align*}
		We now assume that $r-s>1/n$. If $r\in D_n$ then
		\begin{align*}
			\int_s^r(r- k_n(\theta))^{-1- \varepsilon}d \theta\le (1/n)^{-\varepsilon} \sum_{j=1}^\infty j^{-1- \varepsilon}=N(1/n)^{-\varepsilon}.
		\end{align*}
		If $r\notin D_n$, then we have $r>k_n(r)$ and
		\begin{align*}
			\int_s^r(r- k_n(\theta))^{-1- \varepsilon}d \theta
			&=\int_s^{k_n(r)}(r- k_n(\theta))^{-1- \varepsilon}d \theta+\int_{k_n(r)}^r(r- k_n(\theta))^{-1- \varepsilon}d \theta
			\\&\le N (1/n)^{-\varepsilon}+(r-k_n(r))^{-\varepsilon}.
		\end{align*}
		This shows \eqref{est.knminus}.

		When $\varepsilon=0$, we argue analogously. The only notable difference is the following estimate
		\begin{align*}
			\int_s^{k_n(r)}(r-k_n(\theta))^{-1}d \theta\le \sum_{j=1}^{n(k_n(r)-s)}j^{-1}\le \log(n(k_n(r)-s))+1.
		\end{align*}
		This shows \eqref{est.kn0}.

		Since $r-k_n(\theta)\ge r- \theta$, estimate \eqref{est.knplus} is obvious.
	\end{proof}
	\begin{lemma}[{\cite[Proposition 2.7]{DGL2021}}]\label{lem.pSig}
		Let $K>0$ be a constant and let $\Sigma,\bar \Sigma$ be symmetric invertible matrices such that $K^{-1}I\le \Sigma\bar \Sigma^{-1}\le KI$. Then for all $x\in\Rd$, one has the bound
		\begin{align}\label{est.psigma1}
			|p_\Sigma(x)-p_{\bar \Sigma}(x)|\le N\|I- \Sigma\bar \Sigma^{-1}\|\left(p_{\Sigma/2}(x)+p_{\bar \Sigma/2}(x)\right)
		\end{align}
		where $N$ is a constant depending only on $d, K$.
	\end{lemma}
	

\section{Regularizing properties of Brownian paths---a case study}\label{sec.BM}
We obtain various moment estimates for the following functionals of Brownian paths
\begin{align*}
	\int_s^t [f(r,B_r)-f(r,B_{k_n(r)})]dr
	\tand
	\int_s^t g(r,B_r)dr
\end{align*}
where $f,g$ are measurable functions in $ \LL^q_p([0,1])$. In typical applications herein, we take $f=b^n$ and $g=b-b^n$. Hence,  $f$ usually has an additional property of being in $\LL^q_\infty([0,1])$. 
To extract a rate from $b-b^n$, one has to measure $g$ with respect to a norm which is weaker than that of $\LL^q_p$. For this purpose, we  measure $g$ in $\LL^q_{-\nu,p}([0,1])$ for some $\nu\in[0,1]$ or in $\LL^{q_1}_{p_1}([0,1])$ for some $q_1\le q$ and some $p_1\le p$.

In later sections, analogous functionals of the solutions to \eqref{sde00} and \eqref{eqn.EMscheme} will play a central role in the proofs of \cref{thm.main,thm.alpha}. 
While not being applied directly in proving the main results, the analysis in the current section are relatively simpler, mostly due to the fact that statistical properties of Brownian motion are well-understood. In addition, some of the arguments generalizes directly when $B$ is replaced by another stochastic process which has similar analytic estimates. 
Therefore, we present these results at an early stage in the hope of easing out the technicalities and outlining our method. Readers who are familiar with the stochastic sewing techniques, of course, may go directly to the following sections.

Let $p_t(x):=(2 \pi t)^{-d/2}e^{-|x|^2/(2t)}$ and $P_{s,t}f(x):=p_{t-s}\ast f(x)$.

\begin{lemma} \label{lem.heat1}
	Let $p\in[1,\infty]$, $\delta\in(0,1)$, for $0<s<t$, there exists $N=N(d,p,\delta)>0$ such that for any $f\in L_p(\mathbb{R}^d)$ and $0<s<t$,
	\begin{align}\label{BMest}
	\Vert f(B_t)\Vert_{L_p(\Omega)}\leq N t^{-\frac{d}{2p}}\Vert f\Vert_{L_p(\mathbb{R}^d)}
	\end{align}
	and
    \begin{align}\label{heat-difference}
    \Vert P_{0,t} f-P_{0,s} f\Vert_{L_p(\mathbb{R}^d)}\leq N|t-s|^\delta|s|^{-\delta}\Vert f\Vert_{L_p(\mathbb{R}^d)}.
    \end{align}
\end{lemma}
\begin{proof}
	Inequality \eqref{BMest} is taken from \cite[Lemma 2.5]{DGL2021}. We only show \eqref{heat-difference}.
	First
	\begin{align*}
	\Vert \nabla^2P_tf\Vert_{L_p(\mathbb{R}^d)}\leq &\Vert \nabla^2p_t\Vert_{L_{1}(\mathbb{R}^d)}\Vert f\Vert_{L_p(\mathbb{R}^d)}\les t^{-1}\Vert f\Vert_{L_p(\mathbb{R}^d)}.
	\end{align*}
	Then for $\delta\in(0,1)$, we have
	\begin{align*}
	\Vert P_{0,t} f-P_{0,s} f\Vert_{L_p(\mathbb{R}^d)}\leq &\int_s^t\Vert\partial_tP_{0,r}f\Vert_{L_p(\mathbb{R}^d)}\dif r =\int_s^t\Vert\Delta P_{0,r}f\Vert_{L_p(\mathbb{R}^d)}\dif r\\ \lesssim  &\int_s^t r^{-1+\delta }r^{-\delta}\dif r\Vert f\Vert_{L_p(\mathbb{R}^d)}\les s^{-\delta}(t-s)^{\delta}\Vert f\Vert_{L_p(\mathbb{R}^d)},
	\end{align*}
	completing the proof.
\end{proof}
While not being used directly, the following result is pivotal.
\begin{proposition}\label{prop.B1}
	Let $f\in\mathbb{L}_p^q([0,1])$, with $p,q\in[2,\infty)$ satisfying $\frac{d}{p}+\frac{2}{q}<1$.
	Then for all $2/n\le S\le T\le1$ and $n\in\mathbb{N}$ one has the bounds
	\begin{multline}\label{Davieest}
		\Vert \int_S^T(f(r,B_r)-f(r,B_{k_n(r)}))\dif r\Vert_{L_p(\Omega)}
		\\\leq N(1/n)^{\frac12}\log(n) \Vert f\Vert_{\mathbb{L}_p^q([S,T])}\left[S^{-\frac d{2p}} |T-S|^{\frac{1}{2}-\frac1q}+S^{-\frac d{p}} |T-S|^{1-\frac2q}\right]
	\end{multline}
	and
	\begin{align}\label{est.qdt2}
	\Vert \int_S^T(f(r,B_r)-f(r,B_{k_n(r)}))\dif r\Vert_{L_p(\Omega)}
	\leq N(1/n)^{\frac12}\log(n) \Vert f\Vert_{\mathbb{L}_p^q([S,T])} |T-S|^{\frac{1}{2}-\frac1q-\frac d{2p}}.
	\end{align}
\end{proposition}
\begin{proof}
	\eqref{est.qdt2} is a direct consequence of \eqref{Davieest} and \cref{lem.Bellising}.
	We show \eqref{Davieest} below.
	Let $2/n\le S\le T\le 1$ be fixed. By linearity, we can assume that $\|f\|_{\LL^q_p([S,T])}=1$.
	For $S\le s\le t\le T$, let
	\begin{align*}
	A_{s,t}:=\E_s\int_s^t(f(r,B_r)-f(r,B_{k_n(r)}))\dif r.
	\end{align*}
	We treat two cases $t\leq k_n(s)+\frac{2}{n}$ and $t>k_n(s)+\frac{2}{n}$ separately as following.
	
	\textit{Case 1.} For $t\in(s,k_n(s)+\frac{2}{n}]$, by triangle inequality and \eqref{BMest} we have
	\begin{align*}
	\|A_{s,t}\|_{L_p(\Omega)}
	&\leq \int_s^t\Vert f(r,B_r)\Vert_{L_p(\Omega)}+\Vert f(r,B_{k_n(r)})\Vert_{L_p(\Omega)} \dif r
	\\&\lesssim \int_s^t k_n(r)^{-\frac{d}{2p}}\Vert f(r,\cdot)\Vert_{L_p(\mathbb{R}^d)}\dif r.
	\end{align*}
	Note that $k_n(r)\ge k_n(s)\ge s/2$, applying H\"older inequality and the fact that $t-s\le2/n$, we have
	\begin{align*}
		\int_s^t k_n(r)^{-\frac{d}{2p}}\Vert f(r,\cdot)\Vert_{L_p(\mathbb{R}^d)}\dif r
		&\les s^{-\frac d{2p}}\|f\|_{\LL^q_p([s,t])}(t-s)^{1-\frac1q}
		\\&\les (1/n)^{1/2} s^{-\frac d{2p}}\|f\|_{\LL^q_p([s,t])}(t-s)^{\frac12-\frac1q}.
	\end{align*}
	This gives
	\begin{align}\label{est1}
		\|A_{s,t}\|_{L_p(\Omega)}\les (1/n)^{\frac12} s^{-\frac d{2p}}\|f\|_{\LL^q_p([s,t])}(t-s)^{\frac12-\frac1q}.
	\end{align}
	
	\textit{Case 2.} When $t\in(k_n(s)+\frac{2}{n},1]$, by triangle inequality,
	\begin{align*}
	\|A_{s,t}\|_{L_p(\Omega)}&\le\int_s^{k_n(s)+\frac{2}{n}}\Vert \E_s(f(r,B_r)-f(r,B_{k_n(r)}))\Vert_{L_p(\Omega)} \dif r
	\\&\quad+\int_{k_n(s)+\frac{2}{n}}^t\Vert \E_s(f(r,B_r)-f(r,B_{k_n(r)}))\Vert_{L_p(\Omega)} \dif r=:I_1+I_2.
	\end{align*}
	For $I_1$, from \eqref{est1} we know that
	\begin{align*}
	I_1\lesssim (1/n)^{\frac12} s^{-\frac d{2p}}\|f\|_{\LL^q_p([s,t])}\left(k_n(s)-s+\frac2n\right)^{\frac12-\frac1q}.
	\end{align*}
	Because  $k_n(s)-s+\frac2n\le t-s$,
	we get
	\begin{align*}
		I_1\lesssim (1/n)^{\frac12} s^{-\frac d{2p}}\|f\|_{\LL^q_p([s,t])}(t-s)^{\frac12-\frac1q}.
	\end{align*}
	By \eqref{BMest} and \eqref{heat-difference} we have for $I_2$
	\begin{align*}
	I_2&\lesssim\int_{k_n(s)+\frac{2}{n}}^t \Vert P_{s,r}f(r,B_s)-P_{s,k_n(r)}f(r,B_s)\Vert_{L_p(\Omega)}\dif r
	\\&\lesssim\int_{k_n(s)+\frac{2}{n}}^ts^{-\frac{d}{2p}} \Vert P_{s,r}f(r,\cdot)-P_{s,k_n(r)}f(r,\cdot)\Vert_{L_p(\mathbb{R}^d)}\dif r
	\\&\lesssim \int_{k_n(s)+\frac{2}{n}}^ts^{-\frac{d}{2p}} n^{-\frac{1}{2}}(k_n(r)-s)^{-\frac{1}{2}}\Vert f(r,\cdot)\Vert_{L_p(\mathbb{R}^d)}\dif r
	\\&\lesssim (1/n)^{\frac{1}{2}}s^{-\frac{d}{2p}}\|f\|_{\LL^q_p([s,t])}(t-s)^{\frac{1}{2}-\frac{1}{q}}.
	\end{align*}
	Combining these two cases together we obtain that for $S\le s\le t\le T$,
	\begin{align}\label{A}
	\Vert A_{s,t}\Vert_{L_p(\Omega)}\lesssim (1/n)^{\frac{1}{2}}s^{-\frac{d}{2p}}\|f\|_{\LL^q_p([s,t])}(t-s)^{\frac{1}{2}-\frac{1}{q}}.
	\end{align}
	Furthermore, for $u\in(s,t)$, we have $\E_s \delta A_{s,u,t}=0$.
	Let $w$ be the continuous control defined by
	\begin{align*}
		w(s,t)=\left[s^{-\frac d{2p}}\|f\|_{\LL^q_p([s,t])}(t-s)^{\frac12-\frac 1q}\right]^2
		+s^{-\frac d{2p}}\|f\|_{\LL^q_p([s,t])}(t-s)^{1-\frac1q}.
	\end{align*}
	(See \cref{ex.control} for  a justification that $w$ is a control.)
	Denote
	\begin{align*}
	 \mathcal{A}_{t}:=\int_0^t(f(r,B_r)-f(r,B_{k_n(r)}))\dif r,\quad J_{s,t}:=\delta\mathcal{A}_{s,t}-A_{s,t}.
	\end{align*}
	Using similar estimates leading to \eqref{est1}, we have
	\begin{align*}
		\|J_{s,t}\|_{L_p(\Omega)}\les s^{-\frac d{2p}}\|f\|_{\LL^q_p([s,t])}(t-s)^{1-\frac1q}\les w(s,t).
	\end{align*}
	Furthermore, $\delta J_{s,u,t}=-\delta A_{s,u,t}$ and we derive from \eqref{A} that
	\begin{align*}
		\|\delta J_{s,u,t}\|_{L_p(\Omega)}\les(1/n)^{\frac12} w(s,t)^{\frac12}.
	\end{align*}
	It is obvious that $\E_s J_{s,t}=0$ and hence $\E_s \delta J_{s,u,t}=0$.
	Applying \cref{lem.Davie_iteration}, we have
	\begin{align*}
		\|J_{s,t}\|_{L_p(\Omega)}\les (1/n)^{\frac12}\log(n)\left[w(s,t)^{\frac12}+w(s,t)\right]
	\end{align*}
	for every $S\le s\le t\le T$.
	By triangle inequality and \eqref{A}, this implies that
	\begin{align*}
		\|\delta\caa_{s,t}\|_{L_p(\Omega)}
		\les (1/n)^{\frac12}\log(n)\left[w(s,t)^{\frac12}+w(s,t)\right].
	\end{align*}
	Because $\|f\|_{\LL^q_p([s,t])}\le\|f\|_{\LL^q_p([S,T])}=1$ and $t-s\le 1$, we have $w(s,t)\le2 s^{-\frac dp}(t-s)^{1-\frac2q}$. Hence, we deduce \eqref{Davieest} from the above estimate.
\end{proof}
In the following result, the $L_p(\Omega)$-norm in \eqref{est.qdt2} is improved to $L_p(\Omega|\cff_v)$-norm.
\begin{proposition}
	Let $f\in\mathbb{L}_p^q([0,1])$, with $p,q\in[2,\infty)$ satisfying $\frac{d}{p}+\frac{2}{q}<1$. Let $v\in[0,1-2/n]$ be a fixed number. Then for all $v+2/n\le S\le T\le 1$ and all $n$, one has the bound
	\begin{align}\label{est.qdt3}
		\|\int_S^T[f(r,B_r)-f(r,B_{k_n(r)})]dr\|_{L_p(\Omega|\cff_v)}\le N(1/n)^{\frac12}\log(n)\|f\|_{\LL^q_p([S,T])}|T-S|^{\frac12-\frac1q-\frac d{2p}},
	\end{align}
	where $N=N(p,d)$ is a constant.
\end{proposition}
\begin{proof}
	We follow the argument used in \cref{prop.B1}, replacing the $L_p(\Omega)$-norm by the $L_p(\Omega|\cff_v)$-norm. The estimate \eqref{BMest} used therein (whose purpose is to deduce the analytic $L_p(\Rd)$-norm from the probabilistic $L_p(\Omega)$-norm) is replaced by the following estimate
	\begin{align}
	 	\|g(B_t)\|_{L_p(\Omega|\cff_v)}\le N(t-v)^{-\frac d{2p}}\|g\|_{L_p(\Rd)} \quad \forall g\in L_p(\Rd), t>v,
	\end{align}
	with the same constant $N$ as in \eqref{BMest}.
	This yields the following estimate, which corresponds to \eqref{Davieest},
	\begin{multline}\label{tmp.842}
		\Vert \int_S^T[f(r,B_r)-f(r,B_{k_n(r)})]dr \Vert_{L_p(\Omega|\cff_v)}
		\\\leq N(1/n)^{\frac12}\log(n) \Vert f\Vert_{\mathbb{L}_p^q([S,T])}\left[(S-v)^{-\frac d{2p}} |T-S|^{\frac{1}{2}-\frac1q}+(S-v)^{-\frac d{p}} |T-S|^{1-\frac2q}\right].
	\end{multline}
	Applying \cref{lem.Bellising}, we obtain \eqref{est.qdt3}.
\end{proof}
An advantage of the conditional norms over the usual ones is realized the next result.
\begin{proposition}\label{prop.supB}
	Let $f$ be a Borel function in  $\LL_p^q([0,1])\cap\LL_\infty^q([0,{1}])$ for some $p,q\in[2,\infty)$ satisfying $\frac{d}{p}+\frac{2}{q}<1$.
	
	We put $\beta_n(f)=\sup_{r\in D_n}\|f\|_{\LL^q_\infty([r,r+1/n])}$.
	Then for any $\bar p\in(0,p)$, there exists a constant $N=N(d,p,q,\bar p)$ such that
	\begin{multline*}
		\|\sup_{t\in[0,1]}|\int_0^t[f(r,B_r)-f(r,B_{k_n(r)})]dr|\|_{L_{\bar p}(\Omega)}
		\\\le N\left[(1/n)^{1-\frac1q}\beta_n(f)+ (1/n)^{\frac12}\log(n)\|f\|_{\LL^q_p([0,1])}\right].
	\end{multline*}
\end{proposition}
\begin{proof}
	Put $\caa_t=\int_0^t(f(r,B_r)-f(r,B_{k_n(r)}))\dif r$ which has continuous sample paths by  \eqref{tmp.86206}.
	In view of \cref{lem.lenglart}, it suffices to show that there exists a constant $N=N(d,p,q)$ such that
	\begin{align}\label{tmp.86212}
		\|\delta\caa_{s,t}\|_{L_p(\Omega|\cff_s)}\le N \left[(1/n)^{1-\frac1q}\beta_n(f)+ (1/n)^{\frac12}\log(n)\|f\|_{\LL^q_p([0,1])}\right]
	\end{align}
	for every $(s,t)\in \Delta$.
	
	Indeed, by assumption and H\"older inequality, we have
	\begin{align}\label{tmp.86206}
		|\delta \caa_{s,t}|
		\le2\int_s^t\|f_r\|_{L_\infty(\Rd)}dr
		\les \|f\|_{\LL^q_\infty([s,t])}(t-s)^{1-\frac1q}
	\end{align}
	for every $(s,t)\in \Delta$.
	For every $(s,t)\in \Delta$ satisfying $t-s\ge2/n$, we obtain from \eqref{tmp.86206} and \eqref{est.qdt3} that
	\begin{align*}
		\|\delta \caa_{s,t}\|_{L_p(\Omega|\cff_s)}
		&\le \|\delta \caa_{s,s+2/n}\|_{L_p(\Omega|\cff_s)}+\|\delta \caa_{s+2/n,t}\|_{L_p(\Omega|\cff_s)}
		\\&\les (1/n)^{1-\frac1q}\|f\|_{\LL^q_\infty([s,s+2/n])}+(1/n)^{\frac12}\log(n)\|f\|_{\LL^q_p([0,1])}.
	\end{align*}
	Note that $\|f\|_{\LL^q_\infty([s,s+2/n])}\les \beta_n(f)$.
	For every $(s,t)\in \Delta$ satisfying $t-s\le2/n$, \eqref{tmp.86206} trivially implies that
	\begin{align*}
		\|\delta \caa_{s,t}\|_{L_p(\Omega|\cff_s)}
		\les(1/n)^{1-\frac1q}\|f\|_{\LL^q_\infty([s,s+2/n])}
		\les(1/n)^{1-\frac1q}\beta_n(f).
	\end{align*}
	Hence, in both cases, we have obtained \eqref{tmp.86212}.
\end{proof}

Next, we turn to the functional $\int_s^t g(r,B_r)dr$.
\begin{lemma}\label{lem.gLp}
	Let $g$ be a function in $\LL^q_p([0,1])$ for some $p,q\in[1,\infty]$ satisfying $\frac d{p}+\frac2q<2$. Then for every $m\ge1$, there exists a constant $N=N(m,d,p,q)$ such that for every $(s,t)\in \Delta$,
	\begin{align*}
	 	\|\int_s^t g(r,B_r)dr\|_{L_m(\Omega)}\le N\|g\|_{\LL^q_p([s,t])}(t-s)^{1-\frac d{2p}-\frac1q}.
	\end{align*}
\end{lemma}
\begin{proof}
	We can assume that $g$ is nonnegative.
	Using standard estimate for the heat kernel and H\"older inequality, we have for every $(s,t)\in \Delta$ that
	\begin{align*}
		\E_s\int_s^t g(r,B_r)dr
		&=\int_s^t P_{s,r}g(r,B_s)dr\les\int_s^t(r-s)^{-\frac d{2p}}\|g_r\|_{L_p(\Rd)}dr
		\\&\les\|g\|_{\LL^q_p([s,t])}(t-s)^{1-\frac d{2p}-\frac 1q}.
	\end{align*}
	Applying \cref{lem.Khasminski}, we obtain the stated estimate for every integer $m\ge1$. This and the H\"older interpolation inequality implies the estimate for any real $m\ge1$.
\end{proof}

\begin{proposition}\label{prop.B2}
	Let $p\in[2,\infty)$ and $q\in[2,\infty]$. Let $\Gamma$ be a positive number, $w_0$ be a continuous control on $\Delta$ and $g$ be a function in $\LL^q_p([0,1])$ such that for any $(s,t)\in \Delta$,
	\begin{align*}
		\|g\|_{\LL^q_{-1,p}([s,t])}\le \Gamma w_0(s,t)^{\frac1q}
		\tand \|g\|_{\LL^q_p([s,t])}\le w_0(s,t)^{\frac1q}.
	\end{align*}
	(a) Then for every $0\le v<S\le T\le1$
	\begin{multline}\label{est.-1pmoment}
		\|\int_S^Tg(r,B_r)dr\|_{L_p(\Omega|\cff_v)}\le N w_0(S,T)^{\frac1q} \Gamma(1+|\log(\Gamma)|)
		\\\times\left[(S-v)^{-\frac d{2p}}(T-S)^{\frac12-\frac1q}+(S-v)^{-\frac d{p}}(T-S)^{1-\frac2q}\right].
	\end{multline}
	(b) If furthermore $p,q$ satisfy $\frac d{p}+\frac 2q<1$, then for any $\bar p\in(0,p)$, there exists a constant $N=N(d,p,q,\bar p)$ such that
	\begin{align*}
		\|\sup_{t\in[0,1]}|\int_0^t g(r,B_r)dr|\|_{L_{\bar p}(\Omega)}\le N \Gamma(1+|\log(\Gamma)|)w_0(0,1).
	\end{align*}
\end{proposition}
\begin{proof}
	(a) Let $v,S,T$ be fixed such that $0\le v<S<T\le1$.	
	We can assume without loss of generality that $w_0(S,T)=1$. 
	For each $(s,t)\in \Delta([S,T])$, let
	\begin{gather*}
		A_{s,t}=\E_s\int_s^t g(r,B_r)dr=\int_s^t P_{s,r}g(r,B_s)dr
		\tand
		J_{s,t}=\int_s^t g(r,B_r)dr-A_{s,t}.
	\end{gather*}
	In the above, we can interchange the conditional expectation and the integration due to Fubini theorem and \cref{lem.gLp}.
	Define the continuous control $w$ on $\Delta([S,T])$ by
	\begin{align*}
		w(s,t)=\left[(s-v)^{-\frac d{2p}}(t-s)^{\frac12-\frac1q}w_0(s,t)^{\frac1q}\right]^2+(s-v)^{-\frac d{2p}}(t-s)^{1-\frac1q}w_0(s,t)^{\frac1q}.
	\end{align*}
	
	Applying Minkowski inequality, \eqref{BMest} and H\"older inequality, we have
	\begin{align*}
		\|J_{s,t}\|_{L_p(\Omega|\cff_v)}\le 2\int_s^t\|g(r,B_r)\|_{L_p(\Omega|\cff_v)}dr
		&\les\int_s^t (r-v)^{-\frac d{2p}}\|g(r,\cdot)\|_{L_p(\Rd)}dr
		\\&\les (s-v)^{-\frac d{2p}} \|g\|_{\LL^q_p([s,t])}(t-s)^{1-\frac1q}\les w(s,t).
	\end{align*}
	Furthermore, $\E_s J_{s,t}=0$, showing that \eqref{con.sll1} is satisfied.

	On the other hand, we have
	\begin{align*}
		\|A_{s,t}\|_{L_p(\Omega|\cff_v)}
		&\le \int_s^t\|P_{s,r}g(r,B_s)\|_{L_p(\Omega|\cff_v)}dr
		\\&\les \int_s^t (s-v)^{-\frac d{2p}}\|P_{s,r}g(r,\cdot)\|_{L_p(\Rd)}dr
		\\&\les \int_s^t (s-v)^{-\frac d{2p}}(r-s)^{-\frac12} \|g(r,\cdot)\|_{L_{-1,p}(\Rd)}dr
		\\&\les  (s-v)^{-\frac d{2p}}(t-s)^{\frac12-\frac1q} \|g\|_{\LL^q_{-1,p}([s,t])}.
	\end{align*}
	Combining with our assumption on $g$ leads to $\|A_{s,t}\|_{L_p(\Omega)}\les \Gamma w(s,t)^{1/2}$. Since $\delta J_{s,u,t}=- \delta A_{s,u,t}$, this implies that $J$ satisfies \eqref{con.dJ}. The condition \eqref{con.EdJ} is trivial because $\E_s J_{s,t}=0$.
	Applying \cref{lem.Davie_iteration}, we have for every $(s,t)\in \Delta([S,T])$,
	\begin{align*}
		\|J_{s,t}\|_{L_p(\Omega|\cff_v)}
		\les \Gamma(1+|\log \Gamma|)w(s,t)^{\frac12}+\Gamma w(s,t)
	\end{align*}
	and by triangle inequality,
	\begin{align*}
		\|\int_s^t g(r,B_r)dr\|_{L_p(\Omega)}
		\les \Gamma(1+|\log \Gamma|)w(s,t)^{\frac12} +\Gamma w(s,t).
	\end{align*}
	Because $w_0(s,t)\le1$ and $t-s\le 1$, we have $w(s,t)\le2 (s-v)^{-\frac dp}(t-s)^{1-\frac2q}$. Hence, we deduce \eqref{est.-1pmoment} from the above estimate by taking $(s,t)=(S,T)$.	

	(b) Applying \cref{lem.Bellising} and part (a), we have
	\begin{align*}
		\|\int_s^t g(r,B_r)dr\|_{L_p(\Omega|\cff_v)}\les \Gamma(1+|\log(\Gamma)|)	w_0(0,1)	
	\end{align*}
	for every $v<s\le t\le1$.
	In view of \cref{lem.gLp} and Kolmogorov continuity theorem, it is easy to see that the process $t\mapsto\int_0^t g(r,B_r)dr$ has a continuous version.
	Hence, the above inequality holds for every $0\le v= s\le t\le1$. Applying \cref{lem.lenglart}, we obtain the result.
\end{proof}

\begin{proposition}\label{prop.B3}
	Let $\nu\in[0,1)$, $p\in[2,\infty)$ and $q\in[2,\infty]$, $\frac d{p}+\frac 2q+\nu<2$. Let  $g$ be a function in $\LL^q_p([0,1])\cap \LL^q_{-\nu,p}([0,1])$
	Then for any $\bar p\in(0,p)$, there exists a constant $N=N(\nu,d,p,q,\bar p)$ such that
	\begin{align*}
		\|\sup_{t\in[0,1]}|\int_0^t g(r,B_r)dr|\|_{L_{\bar p}(\Omega)}\le N \|g\|_{\LL^q_{-\nu,p}([0,1])}.
	\end{align*}
\end{proposition}
\begin{proof}
	In view of \cref{lem.lenglart}, it suffices to show that
	\begin{align}\label{tmp.8241036}
		\sup_{(s,t)\in \Delta}\|\int_s^t g(r,B_r)dr\|_{L_p(\Omega|\cff_s)}\les \|g\|_{\LL^q_{-\nu,p}([0,1])}.
	\end{align}
	The proof is similar to that of \cref{prop.B2}, however, the control can be chosen in a simpler way.
	Let $v\in[0,1]$ be fixed but arbitrary. For each $(s,t)\in \Delta$, $s>v$, define  $A_{s,t}$ and $J_{s,t}$ as in the proof of \cref{prop.B2}. As in the aforementioned proof,
	we have for every $v<s\le u\le t$, $\E_s \delta A_{s,u,t}=0$ and
	\begin{align*}
		\|J_{s,t}\|_{L_p(\Omega|\cff_v)}+ \|A_{s,t}\|_{L_p(\Omega|\cff_v)}\les (s-v)^{-\frac d{2p}}(t-s)^{1-\frac \nu2-\frac1q}\|g\|_{\LL^q_{-\nu,p}([s,t])}.
	\end{align*}
	Let $w$ be the control on $\Delta((v,1])$ defined by
	\begin{align*}
		w(s,t)=\left[(s-v)^{-\frac d{2p}}(t-s)^{1-\frac \nu2-\frac1q}\|g\|_{\LL^q_{-\nu,p}([s,t])}\right]^{1/(1- \nu/2)}.
	\end{align*}
	The previous estimate yields $\|J_{s,t}\|_{L_p(\Omega|\cff_v)}+\|A_{s,t}\|_{L_p(\Omega|\cff_v)}\les w(s,t)^{1- \nu/2}$. It is evident that $\E_sJ_{s,t}=0$. Noting that $1- \nu/2>1/2$ and applying \cref{lem.Davie_iteration}
	\begin{align*}
		\|J_{s,t}\|_{L_p(\Omega|\cff_v)}\les w(s,t)^{1-\frac \nu2}.
	\end{align*}
	By triangle inequality and the previous estimate for $\|A_{s,t}\|_{L_p(\Omega|\cff_v)}$, we have for every $(s,t)\in \Delta$, $s>v$,
	\begin{align*}
		\|\int_s^t g(r,B_r)dr\|_{L_p(\Omega|\cff_v)}\les w(s,t)^{1-\frac \nu2}
		\les (s-v)^{-\frac d{2p}}(t-s)^{1-\frac \nu2-\frac1q}\|g\|_{\LL^q_p([0,1])}.
	\end{align*}
	An application of \cref{lem.Bellising} gives
	\begin{align*}
		\|\int_s^t g(r,B_r)dr\|_{L_p(\Omega|\cff_v)}\les \|g\|_{\LL^q_p([0,1])}
	\end{align*}
	for every $(s,t)\in \Delta$, $s>v$.
	In view of \cref{lem.gLp} and Kolmogorov continuity theorem, the process $t\mapsto \int_0^t g(r,B_r)dr$ has a continuous version. For this version, we see that the previous estimate holds for every $(s,t)\in \Delta$ and $v=s$, which shows \eqref{tmp.8241036}.
\end{proof}

\section{Analysis of the discrete paths} 
\label{sec:regularizing_properties_of_the_discrete_paths}
	We extend the results of \cref{sec.BM} to functionals of the solution to the discrete scheme \eqref{eqn.EMscheme}.
	\begin{theorem}\label{prop.gf}
		Assume that Conditions \ref{con.A}\ref{con.Aholder} and \ref{con.B} hold.
		Let $X^n$ be the solution to \eqref{eqn.EMscheme} and let $f\in\mathbb{L}_p^q([0,1])\cap \mathbb{L}^q_\infty([0,1])$ and $g\in \LL^q_{1,p}([0,1])\cap\LL^\infty_\infty([0,1])$.
		Define $\beta_n(f)=\sup_{t\in D_n}\|f\|_{\LL^q_\infty([t,t+1/n])}$.
		Then for any $\bar p\in(0,p)$, there exists a constant $N=N(d,p,q,\bar p)$ such that
		\begin{align}\label{gfX}
			\Big\|\sup_{t\in[0,1]}&|\int_0^tg(r,X^n_r) [f(r,X^n_r)-f(r,X^n_{k_n(r)})]dr|\Big\|_{L_{\bar p}(\Omega)}
			\\&\le N\left[\|g\|_{\LL^\infty_\infty([0,1])}+\|g\|_{\LL^q_{1,p}([0,1])}\right]
			\nonumber\\&\quad\times\left[(1/n)^{1-\frac1q}\beta_n(f)+ (1/n)^{\frac \alpha2}\|f\|_{\LL^q_p([0,1])}+(1/n)^{\frac12}\log(n)\|f\|_{\LL^q_p([0,1])}\right].
			\nonumber
		\end{align}
	\end{theorem}
	The rest of the current section is devoted for the proof of \cref{prop.gf}.
	We follow the idea described in \cref{sec.BM}.
	First we derive some analytic estimates on the transition operators associated the discrete Euler--Maruyama scheme without drift.
	These estimates are similar to the ones in \cref{lem.heat1}. By means of the stochastic sewing techniques (\cref{lem.Davie_iteration}) and Girsanov theorem, these analytic estimates are utilized to obtain the desired moment bound.
	In what follows, we carry out this program in more detail. Conditions \ref{con.A}\ref{con.Aholder} and \ref{con.B} are enforced throughout the current \lcnamecref{sec:regularizing_properties_of_the_discrete_paths} unless indicated otherwise.
\subsection{Analytic estimates} 
\label{sub:estimates_of_densities_for_euler_s_approximations}
	For each $s\in D_n$ and $x\in\Rd$, let $\bar X^n(s,x)$ be the solution to the following Euler--Maruyama scheme 
	\begin{equation}
		\bar X^n_t=x
		+\int_s^t \sigma(r,\bar X^n_{k_n(r)})dB_r, \quad t\ge s.
	\end{equation}
	For each $t\ge s$ and bounded measurable function $f$, we define the function $Q^n_{s,t}f$ by
	\begin{align*}
		Q^n_{s,t}f(x)=\E f(\bar X^n_t(s,x)).
	\end{align*}
	Let
	\begin{align}\label{cTf}
		T^*_{s,t}f(y)=\E f\left(y+\int_s^t \sigma(r,y)dB_r\right),
	\end{align}
	and let the operator $T_{s,t}$ be conjugate to $T^*_{s,t}$ in $L_2$-sense, which can be computed explicitly
	\begin{align}\label{fml.Tf}
		T_{s,t}f(x)=\int_\Rd f(y)p_{\Sigma_{s,t}(y)}(y-x)dy, \quad\text{when}\quad s<t
	\end{align}
	and $T_{s,s}f(x)=f(x)$,
	where $\Sigma^{ij}_{s,t}=\frac12\int_s^ta_r^{ij}dr$, $a_r=\sigma_r \sigma^*_r$.
	Whenever $s<t$, the function $T_{s,t}f(x)$ is infinitely differentiable and satisfies
	\begin{equation*}
		\partial_s T_{s,t}f(x)=-\partial^2_{x_ix_j}T_{s,t}[a^{ij}_sf](x).
	\end{equation*}
	We also define for every $r<t$
	\begin{align*}
	 	\eta_r(x)=
	 	\int_{k_n(r)}^r \sigma(v,x)dB_v
	\end{align*}
	  and
	\begin{align*}
		H^n_{r,t}f(x)
		&=
		\E \left[a^{ij}_r(x)(\partial^2_{x_ix_j}T_{r,t}f)(x+\eta_r(x))-(\partial^2_{x_ix_j}T_{r,t}[a^{ij}_r f])(x+\eta_r(x))\right].
	\end{align*}
	The function $\eta$ also depends on $n$, however, we omit this dependence in the notation.
	By direct computations (see also \cite[p.153]{MR1392450} or \cite[p.11]{gyongy2021existence}), we have
	\begin{align*}
		H^n_{r,t}f(x)=\int_\Rd K^n_{r,t}(x,y)f(y)dy
	\end{align*}
	where	
	\begin{align*}
		&K^n_{r,t}(x,y)
		=
		[a^{ij}_r(x)-a^{ij}_r(y)] \partial^2_{ij} p_{\Sigma_{k_n(r),r}(x)+ \Sigma_{r,t}(y)} (y-x)
		\\&=
		[a^{ij}_r(x)-a^{ij}_r(y)]
		[(A_{r,t}(x,y)z)^i(A_{r,t}(x,y)z)^j-A^{ij}_{r,t}(x,y)]
		p_{\Sigma_{k_n(r),r}(x)+ \Sigma_{r,t}(y)} (z)\Big|_{z=y- x}
	\end{align*}
	and
	\begin{align*}
		A_{r,t}(x,y)=(\Sigma_{k_n(r),r}(x)+ \Sigma_{r,t}(y))^{-1}.
	\end{align*}
	The relation between $Q^n$ and $T,H^n$ is described in \cref{lem.discreteDuhamel}, which is a kind of Duhamel formula. It is, however, convenient to obtain first analytic estimates for $H^n$ and $T$.
	We make use of the following result, inspired by \cite[Lemma 4.1]{MR1392450}.\footnote{\cite{gyongy2021existence} recently corrected \cite[Lemma 4.1]{MR1392450}. Our assumptions and conclusions are different from both works.}

	\begin{lemma}\label{lem.GK}
		Let $\lambda,\ell,\varepsilon>0$, $\alpha\in[0,1]$ be fixed numbers, let $a_1$ be a symmetric $d\times d$ matrix
		and let $a(x),\bar a(x)$ be $d\times d$ matrix-valued functions.
		Assume that for each $x$, $a(x),\bar a(x)$ are symmetric, $\lambda^{-1} \ell I\le a(x)+a_1,\bar a(x)+a_1\le \lambda \ell I$, where $I$ is the $d\times d$ unit matrix, and $\sup_{y}\|\bar a(y)-a(y)\|\le \varepsilon$.
		Let $g(x)$ be a real function such that $|g(x)-g(y)|\le \lambda|x-y|^\alpha$ for all $x,y$. Let $\xi$ and $\eta$ be independent $d$-dimensional Gaussian vectors with zero means. Assume that $\xi$ has distribution $\mathcal N(0,I)$ and $\eta$ has distribution $\mathcal N(0,\sqrt{a_1})$. Define an operator $T^*$ by the formula $T^*f(y)=\E f(y+\sqrt{a(y)}\xi)$ and let $T$ be the conjugate for $T^*$ in $L_2$-sense. 
		Let $1\le i,j\le d$ be fixed and define
		\begin{align*}
			H(x)=\E \left[g(x)(\partial^2_{x_ix_j}Tf)(x+\eta)-(\partial^2_{x_ix_j}T[g f])(x+\eta)\right].
		\end{align*}
		Define $\bar T^*,\bar T$ and $\bar H(x)$ analogously with $\bar a$ replacing $a$.
		Then for any $p\in[1,\infty]$ and bounded Borel $f$,
		\begin{align}\label{est.1H}
			&\sup_{x\in\Rd}|H(x)|\le N \|f\|_{L_p(\Rd)}\ell^{\frac \alpha2-1-\frac d{2p}}
			,\quad \|H\|_{L_p(\Rd)}\le N \|f\|_{L_p(\Rd)}\ell^{\frac \alpha2-1},
			\\&\label{est.GK2}
			|\bar H(x)-H(x)|\le N \|f\|_{L_p(\Rd)} \varepsilon \ell^{\frac \alpha2-2-\frac d{2p}}
			\tand
			\|\bar H-H\|_{L_p(\Rd)}\le N\|f\|_{L_p(\Rd)}\varepsilon \ell^{\frac \alpha2-2},
		\end{align}
		where the constant $N$ depends only on $\lambda,d,p,\alpha$.
	\end{lemma}
	\begin{proof}
		By direct computations (see also \cite[p.153]{MR1392450} or \cite[p.11]{gyongy2021existence}), we have
		\begin{align*}
			H(x)=\int_\Rd[g(x)-g(y)]f(y)[(A(y)z)^i(A(y)z)^j-A^{ij}(y)]p_{a(y)+a_1}(z)\Big|_{z=y-x}dy,
		\end{align*}
		where $A(y)=(a(y)+a_1)^{-1}$.
		A similar formula for $\bar H(x)$ is valid with $\bar A(y)=(\bar a(y)+a_1)^{-1}$.
		By ellipticity of $a(y)+a_1$ and H\"older continuity of $g$, we have
		\begin{align*}
			|H(x)|
			\les\int_{\Rd}|x-y|^\alpha|f(y)|\left(\ell^{-2}|y-x|^2+\ell^{-1}\right)p_{2\lambda\ell}(y-x)dy.
		\end{align*}
		Setting $q=\frac p{p-1}$ and applying H\"older inequality, we get
		\begin{align*}
			|H(x)|
			\les\|f\|_{L_p(\Rd)}\left(\int_{\Rd}|y|^{q \alpha}(\ell^{-2}|y|^2+\ell^{-1})^qp_{2\lambda\ell}(y)^qdy \right)^{\frac1q}
			\les\|f\|_{L_p(\Rd)}\ell^{\frac \alpha2-1-\frac d{2p}}.
		\end{align*}
		  This shows the first estimate in \eqref{est.1H}. The second estimate in \eqref{est.1H} is obtained by a similar argument, using instead the Minkowski inequality.
		
		It is evident that $\|A(y)\|,\|\bar A(y)\|\les \ell^{-1}$ uniformly in $y$.
		For two invertible matrices $C,D$, we have $C^{-1}-D^{-1}=C^{-1}(D-C)D^{-1}$ so that $\|C^{-1}-D^{-1}\|\le \|C^{-1}\|\|D^{-1}\|\|D-C\|$.
		Thus $\|\bar A(y)-A(y)\|\le\|A(y)\|\|\bar A(y)\| \| \bar a(y)-a(y)\|\les \ell^{-2} \varepsilon$ and similarly
		\begin{align*}
			\|(\bar A(y)z)^i(\bar A(y)(y-x))^j-(A(y)z)^i(A(y)(y-x))^j\|\les |z|^2\ell^{-3} \varepsilon.
		\end{align*}
		Using our assumptions, it is straightforward to verify that $K^{-1}I\le (a(y)+a_1)(\bar a(y)+a_1)^{-1}\le K I$ and $\|I-(a(y)+a_1)(\bar a(y)+a_1)^{-1}\|\le K  \ell^{-1}\varepsilon$ for some constant $K>0$. Hence, from \eqref{est.psigma1}, we have
		\begin{align*}
			|p_{\bar a(y)+a_1}(z)-p_{a(y)+a_1}(z)|
			&\les \ell^{-1}\varepsilon\left(p_{(\bar a(y)+a_1)/2}(z)+p_{(a(y)+a_1)/2}(z)\right)
			\\&\les \ell^{-1}\varepsilon p_{\lambda \ell}(z),
		\end{align*}
		where we have used the fact that $\lambda^{-1}t\le a(y)+a_1,\bar a(y)+a_1\le \lambda t$.
		It follows that
		\begin{align*}
			|\bar H(x)-H(x)|
			\les \varepsilon \int_\Rd|f(y)||y-x|^\alpha\left(|y-x|^2\ell^{-3}+\ell^{-2}\right)p_{\lambda \ell}(y-x)dy.
		\end{align*}
		From here, we apply H\"older inequality and Minkowski inequality as previously to obtain \eqref{est.GK2}.
	\end{proof}
	\begin{lemma}\label{lem.HLp}
		Let $p_1,p_2\in[1,\infty]$ be such that $p_1\le p_2$ and let $f$ be a function in $L_{p_1}(\Rd)$. For every $r<t\le1$, we have
		\begin{align}\label{est.Tlp1}
			\|T_{r,t}f\|_{L_{p_2}(\Rd)}\le N(t-r)^{\frac d{2p_2}-\frac d{2p_1}}\|f\|_{L_{p_1}(\Rd)}
		\end{align}	
		and	
		\begin{align}\label{est.Hlp1}
			\|H^n_{r,t}f\|_{L_{p_2}(\Rd)}\le N(t-k_n(r))^{\frac \alpha2-1+\frac d{2p_2}-\frac d{2p_1}}\|f\|_{L_{p_1}(\Rd)},
		\end{align}
		where the constant $N$ depends only on $d,p_1,p_2,K_1,K_2$.
	\end{lemma}
	\begin{proof}
		By uniform ellipticity, there exists a  constant $\lambda>0$ such that for every $x,y\in\Rd$,
		\begin{align}\label{tmp.elliptic}
			\lambda^{-1}(t- k_n(r))\le \int_r^t a_\theta(y)d\theta+\int_{ k_n(r)}^r a_\theta(x)d\theta\le \lambda(t-  k_n(r)).
		\end{align}
		From here, we can derive \eqref{est.Tlp1} using standard Gaussian estimates.
		
		Applying \eqref{est.1H}, we have
		\begin{align*}
			&\|H^n_{r,t}f\|_{L_\infty(\Rd)}\les (t-k_n(r))^{\frac \alpha2-1-\frac d{2p_1}}\|f\|_{L_{p_1}(\Rd)}	
			\\\shortintertext{and}
			&\|H^n_{r,t}f\|_{L_{p_1}(\Rd)}\les (t-k_n(r))^{\frac \alpha2-1}\|f\|_{L_{p_1}(\Rd)}	.
		\end{align*}
		From the above estimates and the H\"older interpolation inequality
		\begin{align*}
			\|H\|_{L_{p_2}(\Rd)}\le\|H\|_{L_{p_1}(\Rd)}^{\frac{p_1}{p_2}}\|H\|_{L_{\infty}(\Rd)}^{1-\frac{p_1}{p_2}},
		\end{align*}
		we deduce \eqref{est.Hlp1}.		
	\end{proof}
	The estimate \eqref{est.Hlp1} shows in particular that whenever $r<t$, $H^n_{r,t}$ maps bounded measurable functions to bounded measurable functions. It is evident from their definitions that $Q^n_{r,t},T_{r,t}$ also map bounded measurable functions to bounded measurable functions.
	\begin{lemma}\label{lem.discreteDuhamel}
		Let $s\in D_n$ and $f$ be a bounded uniformly continuous function. Then, for every $t>s$ and $x\in\Rd$
		\begin{align}\label{id.duhamel}
			Q^n_{s,t}f(x)=T_{s,t}f(x)
			+\int_s^t Q^n_{s,k_n(r)}[H^n_{r,t}f](x)dr.
		\end{align}
	\end{lemma}
	\begin{proof}
		Let $\bar X^n=\bar X^n(s,x)$ and $\tau\in(s,t)$.
		Applying It\^o formula for $r\mapsto T_{r,t}f(\bar X^n_r)$, for any $t>s$, we obtain that
		\begin{multline*}
			\E T_{\tau,t} f(\bar X^n_\tau)
			\\=T_{s,t}f(\bar X^n_{s})
			+\E\int_{s}^\tau \left[a^{ij}_r(\bar X^n_{k_n(r)})(\partial^2_{x_ix_j}T_{r,t}f)(\bar X^n_r)-(\partial^2_{x_ix_j}T_{r,t}[a^{ij}_r f])(\bar X^n_r)\right]d r.
		\end{multline*}
		Writing $\bar X^n_r=\bar X^n_{ k_n(r)}+\eta_r(\bar X^n_{ k_n(r)})$, we  take conditional expectation given $\cff_{ k_n(r)}\supset\cff_s$. This yields
		\begin{align}\label{fml.Esf}
			\E T_{\tau,t} f(\bar X^n_\tau)
			&=T_{s,t}f(\bar X^n_s)
			+\int_s^\tau\E H^n_{r,t}f(\bar X^n_{ k_n(r)})d r.
		\end{align}
		We now take the limit $\tau\uparrow t$ in the above formula.
		By uniform continuity of $f$ and a.s. continuity of $\bar X^n$, $\lim_{\tau\uparrow t}T_{\tau,t}f(\bar X^n_\tau)=f(\bar X^n_t)$.
		From  \eqref{est.Hlp1}, we have
		\begin{align*}
			\int_\tau^t|\E H^n_{r,t}f(\bar X^n_{ k_n(r)})|d r
			&\les\int_\tau^t (t-k_n(r))^{\frac \alpha2-1}\|f\|_{L_\infty(\Rd)}dr
			\les\|f\|_{L_\infty(\Rd)}(t- \tau)^{\frac \alpha2},
		\end{align*}
		which allows one to apply the limit $\tau\uparrow t$ to the last  term in \eqref{fml.Esf}. Hence, we have
		\begin{align*}
			\E  f(\bar X^n_t)
			&=T_{s,t}f(\bar X^n_s)
			+\int_s^t\E H^n_{r,t}f(\bar X^n_{ k_n(r)})d r,
		\end{align*}
		which deduces to \eqref{id.duhamel}.
	\end{proof}
	\begin{theorem}\label{thm.QLp12}
		Assume that \cref{con.A}\ref{con.Aholder} holds.
		Let $p_1,p_2\in[1,\infty]$, $p_1\le  p_2$, $p_1<\infty$ and let $f$ be a function in $L_{p_1}(\Rd)$. There exists a constant $N=N(d,p_1,p_2,\alpha,K_1,K_2)$ such that for every $s\in D_n$ and $t\in(s,1]$, we have
		\begin{align}
			\|Q^n_{s,t}f\|_{L_{p_2}(\Rd)}\le N(t-s)^{\frac d{2p_2} -\frac d{2p_1}}\|f\|_{L_{p_1}(\Rd)}.
			\label{est.Q1}
		\end{align}
	\end{theorem}
	\begin{proof}
		We put $\rho=\frac d{2p_1}-\frac d{2p_2}$.

		\textit{Step 1.} We show some rough estimates for $\|Q^n_{s,t}f\|_{L_{p_2}(\Rd)}$ in terms of $\|f\|_{L_{p_1}(\Rd)}$. Assume first that $f$ is a bounded uniformly continuous function. From \eqref{id.duhamel} and \cref{lem.HLp}, we have for every $t\in[s,s+1/n]$
		\begin{align}
			\|Q^n_{s,t}f\|_{L_{p_2}(\Rd)}
			&\le \|T_{s,t}f\|_{L_{p_2}(\Rd)}
			+\int_s^{t}\|H^n_{r,t}f\|_{L_{p_2}(\Rd)}dr
			\nonumber\\&\les(t-s)^{-\rho}\|f\|_{L_{p_1}(\Rd)}
			+\int_s^t (t-k_n(r))^{\frac \alpha2-1- \rho}\|f\|_{L_{p_1}(\Rd)}dr
			\nonumber\\&\les (t-s)^{- \rho}\|f\|_{L_{p_1}(\Rd)},
			\nonumber
		\end{align}
		where the last inequality follows from the fact that $k_n(r)=s$ for $r\in[s,s+1/n)$.
		Since smooth functions are dense in $L_{p_1}(\Rd)$, it follows that that $\|Q^n_{s,t}\|_{L_{p_2(\Rd)}}\les (t-s)^{-\rho}\|f\|_{L_{p_1(\Rd)}}$ for any function $f$ in $L_{p_1}(\Rd)$.

		We proceed inductively.
		Let $j\ge1$ be an integer. Suppose that for every $t\in[s,s+j/n]$ and every function $f\in L_{p_1}(\Rd)$,
		\begin{align}\label{tmp.inductive}
			\|Q^n_{s,t}f\|_{L_{p_2}(\Rd)}\le C_j(t-s)^{-\rho}\|f\|_{L_{p_1}(\Rd)}
		\end{align}
		for some constant $C_j$, independent from $n,s,t,f$.

		Let $f$ be a bounded uniformly continuous function.
		Then for each $t\in(s+j/n,s+(j+1)/n]$, we obtain from \eqref{id.duhamel}, \cref{lem.HLp} and the inductive hypothesis that
		\begin{multline}\label{tmp.0909302}
			\|Q^n_{s,t}f\|_{L_{p_2}(\Rd)}
			\les \|T_{s,t}f\|_{L_{p_2}(\Rd)}
			+\int_s^{s+1/n}\|H^n_{r,t}f\|_{L_{p_2}(\Rd)}dr
			\\+C_j\int_{s+1/n}^{t}(k_n(r)-s)^{-\rho} \|H^n_{r,t}f\|_{L_{p_1}(\Rd)}dr.
		\end{multline}
		The first two terms are estimated as previously, $\|T_{s,t}f\|_{L_{p_2}(\Rd)}\les (t-s)^{-\rho}\|f\|_{L_{p_1}(\Rd)}$ and
		\begin{align*}
			\int_s^{s+1/n}\|H^n_{r,t}f\|_{L_{p_2}(\Rd)}dr\les (1/n)^{-\rho}\|f\|_{L_{p_1}(\Rd)}.
		\end{align*}
		Using \cref{lem.HLp}, we have
		\begin{align*}
			\|H^n_{r,t}f\|_{L_{p_1}(\Rd)})\les(t-k_n(r))^{\frac \alpha2-1}\|f\|_{L_{p_1}(\Rd)}.
		\end{align*}
		Using the above estimate and the fact that $k_n(r)-s\ge 1/n$ for any $r\ge s+1/n$ and \cref{lem.intkn}, we have
		\begin{align*}
			\int_{s+1/n}^{t}(k_n(r)-s)^{-\rho} \|H^n_{r,t}f\|_{L_{p_1}(\Rd)}dr
			&\les (1/n)^{-\rho}\int_{s+1/n}^t (t-k_n(r))^{\frac \alpha2-1}dr\|f\|_{L_{p_1}(\Rd)}
			\\&\les (1/n)^{-\rho}\|f\|_{L_{p_1}(\Rd)}.
		\end{align*}
		Observe furthermore that $(1/n)^{-\rho}\le(j+1)^\rho(t-s)^{-\rho}$.
		Putting these estimates in \eqref{tmp.0909302}, we see that \eqref{tmp.inductive} holds for $t\in(s+j/n,s+(j+1)/n]$ with a bounded uniformly continuous function $f$ and some constant $C_{j+1}$.
		By approximation, we can extend the inequality to all functions $f\in L_{p_1}(\Rd)$.

		To show that \eqref{tmp.inductive} actually holds uniformly in $j$, we proceed as follows.

		\textit{Step 2.} Assuming that $\rho=\frac d{2p_2}-\frac d{2p_1} <1$, we show that there exists a constant $C>0$, independent from $n$, such that
		\begin{align}\label{est.Qp12}
			\|Q^n_{s,t}f\|_{L_{p_2}(\Rd)}\le C(t-s)^{\frac d{2p_2}-\frac d{2p_1}}\|f\|_{L_{p_1}(\Rd)}
			\text{ for every $t>s$ and function  $f\in L_{p_1}(\Rd)$.}
		\end{align}

		In view of \eqref{tmp.inductive}, it suffices to consider $t\ge s+2/n$.
		Let $\lambda>1$ be a constant to be chosen later.
		For each $t\in[s+2/n,1]$, define
		\begin{align*}
			m_t=e^{-\lambda(t-s)} (t-s)^{\rho}\sup_{g\in L_{p_1}(\Rd)\ :\ \|g\|_{L_{p_1}(\Rd)}=1}\|Q^n_{s,t}g\|_{L_{p_2}(\Rd)}
		\end{align*}
		and $m^*_t=\sup_{r\in[s+2/n,t]} m_r$,
		which are finite by the previous step.
		In particular, for every $t\ge s+2/n$ and $g\in L_{p_1}(\Rd)$, we have (the case $\|g\|_{L_{p_1}(\Rd)}=0$ is treated by \eqref{tmp.inductive})
		\begin{align}\label{tmp.811541}
			\|Q^n_{s,t}g\|_{L_{p_2}(\Rd)}\le m^*_t e^{\lambda(t-s)}(t-s)^{-\rho}\|g\|_{L_{p_1}(\Rd)}.
		\end{align}
		
		Let $t\ge s+2/n$ and $f$ be a bounded uniformly continuous function, $\|f\|_{L_{p_1}(\Rd)}=1$.
		From \eqref{id.duhamel}, \eqref{tmp.inductive} and \eqref{tmp.811541}, we have that
		\begin{multline}\label{tmp.0909Qnf}
			\|Q^n_{s,t}f\|_{L_{p_2}(\Rd)}
			\les \|T_{s,t}f\|_{L_{p_2}(\Rd)}+\int_s^{s+2/n}\|H^n_{r,t}f\|_{L_{p_2}(\Rd)}dr
			\\+m^*_t\int_{s+2/n}^t e^{\lambda(k_n(r)-s)} (k_n(r)-s)^{-\rho}\|H^n_{r,t}f\|_{L_{p_1}(\Rd)}dr.
		\end{multline}
		From \cref{lem.HLp}, we have $\|T_{s,t}f\|_{L_{p_2}(\Rd)}\les (t-s)^{-\rho}$ and
		\begin{align*}
			\int_s^{s+2/n}\|H^n_{r,t}f\|_{L_{p_2}(\Rd)}dr
			\les\int_s^{s+2/n}(t-k_n(r))^{\frac \alpha2-1-\rho} dr
			&\les(1/n)(t-s-1/n)^{\frac \alpha2-1- \rho}
			\\&\les(t-s)^{-\rho},
		\end{align*}
		where we have used the fact that $1/n\le t-s\le1$ and $t-s-1/n\ge(t-s)/2$.
		Similarly, using \cref{lem.HLp}, we have
		\begin{align*}
			&\int_{s+2/n}^t e^{\lambda(k_n(r)-s)} (k_n(r)-s)^{-\rho}\|H^n_{r,t}f\|_{L_{p_1}(\Rd)}dr
			\\&\les\int_{s+2/n}^t e^{\lambda(k_n(r)-s)} (k_n(r)-s)^{-\rho}(t-k_n(r))^{\frac \alpha2-1} dr
			\\&\les\int_{s+1/n}^t e^{\lambda(r-s)}(r-s-1/n)^{-\rho}(t-r)^{\frac \alpha2-1} dr.
		\end{align*}
		To estimate the integrals on the right-hand sides above, we split them into two regions, putting $\bar s=s+1/n$,
		\begin{multline*}
			\int_{\bar s}^{(\bar s+t)/2} e^{-\lambda(t-r)}(r-\bar s)^{-\rho}(t-r)^{\frac \alpha2-1} dr
			\les e^{-\frac\lambda2(t-\bar s)}(t-\bar s)^{\frac \alpha2-1}\int_{\bar s}^{(\bar s+t)/2}(r-\bar s)^{-\rho}dr
			\\\les e^{-\frac\lambda2(t-\bar s)}(t-\bar s)^{\frac \alpha2- \rho}
			\les \lambda^{-\frac \alpha2}(t-\bar s)^{-\rho}	
		\end{multline*}
		and
		\begin{multline*}
			\int_{(\bar s+t)/2}^t e^{-\lambda(t-r)}(r-\bar s)^{-\rho}(t-r)^{\frac \alpha2-1} dr
			\le (t-\bar s)^{-\rho} \int_{(\bar s+t)/2}^t e^{-\lambda(t-r)}(t-r)^{\frac \alpha2-1}dr
			\\\le(t-\bar s)^{-\rho} \int_0^\infty e^{-\lambda u}u^{\frac \alpha2-1}du\les \lambda^{-\frac \alpha2}(t-\bar s)^{-\rho}.
		\end{multline*}
		In the above, all integrals are finite because $\alpha\in(0,1]$ and $\rho<1$.
		Observe furthermore that $(t-s)/(t-\bar s)\le 2$. Thus we have
		\begin{align*}
			\int_{s+2/n}^t e^{\lambda(k_n(r)-s)} (k_n(r)-s)^{-\rho}(t-k_n(r))^{\frac \alpha2-1} dr\les \lambda^{-\frac \alpha2}e^{\lambda(t-s)} (t-s)^{-\rho}.
		\end{align*}

		Putting the previous estimates altogether into \eqref{tmp.0909Qnf}, we have
		\begin{align*}
			\|Q^n_{s,t}f\|_{L_{p_2}(\Rd)}
			\les (t-s)^{-\rho}(1+m^*_t \lambda^{-\frac \alpha2}e^{\lambda(t-s)}).
		\end{align*}
		By approximations, the above estimate also holds for any function $f\in L_{p_1}(\Rd)$ with $\|f\|_{L_{p_1}(\Rd)}=1$.
		It follows that $m_t\les 1 +\lambda^{- \frac \alpha2}m^*_t$ for every $t\ge s+2/n$. By choosing $\lambda$ sufficiently large, we conclude that $m^*_1$ is bounded by a constant independent from $n$ and thus obtain \eqref{est.Qp12}.

		\textit{Step 3.} We remove the restriction $\rho<1$ in the previous step.

		Suppose that $\rho:=\frac d{2p_1}-\frac d{2p_2}<2$. Define $ p_3\in[p_1,p_2]$  by $\frac d{p_3}=\frac d{2p_1}+\frac d{2p_2}$ so that $\frac d{2p_1}-\frac d{2p_3}=\frac d{2p_3}-\frac d{2p_2}=\frac \rho2<1$. Let $t\ge s+4/n$ and $u=(s+t)/2$. Then by Markov property of the Euler--Maruyama scheme, we have $Q^n_{s,t}=Q^n_{k_n(u),t}Q^n_{s,k_n(u)}$. It is easy to see that $t>k_n(u)>s$ so that by \eqref{est.Qp12}, we have for every $f\in L_{p_1}(\Rd)$ that
		\begin{align*}
			\|Q^n_{s,t}f\|_{L_{p_2}(\Rd)}
			&\les(t-k_n(u))^{-\frac \rho2}\|Q^n_{s,k_n(u)}f\|_{L_{p_3}(\Rd)}
			\\&\les (t-k_n(u))^{-\frac \rho2}(k_n(u)-s)^{-\frac \rho2}\|f\|_{L_{p_1}(\Rd)}.
		\end{align*}
		It is straightforward to see that $t-k_n(u)\ge (t-s)/2$ and $k_n(u)-s\ge(t-s)/4$. Hence, from the above estimate, we deduce that $\|Q^n_{s,t}f\|_{L_{p_2}(\Rd)}\les (t-s)^{-\rho}\|f\|_{L_{p_1}(\Rd)}$ for any $t\ge s+4/n$. Combining with \eqref{tmp.inductive} from step 1., we see that \eqref{est.Qp12} holds for any $p_1,p_2\in[1,\infty]$ satisfying $\frac d{2p_1}-\frac d{2p_2}<2$.
		We iterate the argument. After $\lfloor\log_2(d/2)\rfloor+1$ iterations, we see that \eqref{est.Qp12} holds whenever $\rho\le d/2$, which is trivially satisfied for any $p_1,p_2\in[1,\infty]$. Hence, we have shown \eqref{est.Q1}.
	\end{proof}
	\begin{remark}
		\cref{thm.QLp12} complements previous works.
		It is shown in \cite{MR2735377} that for each $s,t\in D_n$, $s<t$, the operator $Q^n_{s,t}$ has a kernel density which has Gaussian upper bounds. From here, one can deduce estimate \eqref{est.Q1} for discrete times $s,t\in D_n$.
		Gaussian upper bounds for the kernel density of $Q^n_{s,t}$, which are valid for all times $t>s$, $s\in D_n$, are also established in \cite{bao2020convergence,jourdain2021convergence} under Lipschitz regularity of $a$. Some related estimates are also obtained in \cite{MR1392450,gyongy2021existence} under the additional condition $\alpha>d/p$. We are able to remove this condition herein mainly due to \eqref{est.Hlp1}, which was known previously with the factor $(t-r)$ on its right-hand side.
	\end{remark}
	\begin{lemma}
		Let $p_1\in[1,\infty]$ and let $r<k_n(t)$.
		Then for every $f\in L_{p_1}(\Rd)$,
		\begin{align}\label{est.Tlp2}
			\|T_{r,t}f-T_{r,k_n(t)}f\|_{L_{p_1}(\Rd)}\le N(1/n)(t-r)^{-1}\|f\|_{L_{p_1}(\Rd)}
		\end{align}
		and
		\begin{align}\label{est.Hlp2}
			\|H^n_{r,t}f-H^n_{r,k_n(t)}f\|_{L_{p_1}(\Rd)}\le N(1/n)(t-k_n(r))^{\frac \alpha2-2}\|f\|_{L_{p_1}(\Rd)},
		\end{align}	
		where the constant $N$ depends only on $d,p_1,K_1,K_2$.
	\end{lemma}
	\begin{proof}
		We make use of \eqref{tmp.elliptic}. Observe furthermore that $1/2\le(k_n(t)-k_n(r))/(t-k_n(r))\le1$.
		It is then straightforward to verify the hypothesis of \cref{lem.pSig} for $\Sigma=\Sigma_{r,t}(y)$ and $\bar\Sigma=\Sigma_{r,k_n(t)}(y)$. In addition
		\begin{align*}
		 	\|I- \Sigma\bar \Sigma^{-1}\|\les (t-k_n(t))(t-k_n(r))^{-1}\les (1/n)(t-k_n(r))^{-1}.
		\end{align*}
		From the definition of $T$, we apply \eqref{est.psigma1} to get that
		\begin{align*}
			&|T_{r,t}f(x)-T_{r,k_n(t)}f(x)|
			\les(1/n) (t-r)^{-1}\int_\Rd p_{c(t-r)}(y) |f(y-x)|dy
		\end{align*}
		for some universal constant $c$. From here, we apply Minkowski inequality to obtain \eqref{est.Tlp2}.
		\eqref{est.Hlp2} is obtained analogously, using  \eqref{est.GK2}.
	\end{proof}
	
	\begin{corollary}\label{cor.Qfsut}
		Assuming \cref{con.A}\ref{con.Aholder}.
		Let $f$ be a function in $L_{p_1}(\Rd)$, $p_1\in[1,\infty)$. Then for any $s\in D_n$ and $t\in(s+2/n,1]$,
		\begin{align*}
			\|Q^n_{s,t}f-Q^n_{s,k_n(t)}f\|_{L_{p_1}(\Rd)}\le N  (1/n)^{\frac \alpha2} \|f\|_{L_{p_1}(\Rd)}(t-s)^{-\frac  \alpha2},
		\end{align*}
		where the constant $N$ depends only on $d,p_1,K_1,K_2$.
	\end{corollary}
	\begin{proof}
		By approximation, we can assume that $f$ is bounded and uniformly continuous.
		We put $u=k_n(t)$.
		From \eqref{id.duhamel}, $Q^n_{s,t}f-Q^n_{s,u}f=I_1+I_2+I_3$,	where
		\begin{align*}
			&I_1=T_{s,t}f-T_{s,u}f,
			\\&I_2=\int_{s}^uQ^n_{s,k_n(r)}[H^n_{r,t}f-H^n_{r,u}f]dr,
			\quad I_3=\int_u^t Q^n_{s,k_n(r)}[H^n_{r,t}f ]dr.
		\end{align*}
		We estimate each $I_1,I_2,I_3$ below.
		Note that $s< t-1/n$ implies $s<k_n(t)$.
		Applying \eqref{est.Tlp2}, we have $\|I_1\|_{L_p(\Rd)}\les (1/n)(t-s)^{-1} \|f\|_{L_p(\Rd)}$.

		From \cref{thm.QLp12}, we have $\|Q^n_{s,k_n(r)}f\|_{L_p(\Rd)}\les\|f\|_{L_p(\Rd)}$
		for every $r>s$. It follows that
		\begin{align*}
			\|I_2\|_{L_p(\Rd)}\les\int_s^u\|H^n_{r,t}f-H^n_{r,u}f\|_{L_p(\Rd)}dr.
		\end{align*}
 		Applying \eqref{est.Hlp2} and \cref{lem.intkn} (noting that $u-s\ge1/n$), we have
		\begin{align*}
			\|I_2\|_{L_p(\Rd)}\les\int_s^u (1/n)(u-k_n(r))^{\frac \alpha2-2}dr\|f\|_{L_p(\Rd)}
			\les (1/n)^{\frac \alpha2}\|f\|_{L_p(\Rd)}.
		\end{align*}
		Applying \cref{thm.QLp12} and \cref{lem.HLp}, we have
		\begin{align*}
			\|I_3\|_{L_p(\Rd)}\les\int_u^t\|H^n_{r,t}f\|_{L_p(\Rd)}dr
			\les\int_u^t(t-k_n(r))^{\frac \alpha2-1}\|f\|_{L_p(\Rd)}dr
			\les (1/n)^{\frac \alpha2}\|f\|_{L_p(\Rd)}.
		\end{align*}
		Combining the previous estimates, we obtain the result.
	\end{proof}
	\begin{corollary}
		Let $g$ be a function in $L_{1,p}(\Rd)$ for some $p\in[1,\infty)$
		\begin{align}\label{est.Qgg}
			\|Q^n_{s,t}g-g\|_{L_p(\Rd)}\les\|g\|_{L_{1,p}(\Rd)}(t-s)^{\frac \alpha2}.
		\end{align}
	\end{corollary}
	\begin{proof}
		By approximation, we can assume $g$ is continuously differentiable and has bounded derivatives.
		From \eqref{id.duhamel}, we have $Q^n_{s,t}g-g=I_1+I_2$, where $I_1=T_{s,t}g-g$ and $I_2=\int_s^t Q^n_{s,k_n(r)}[H^n_{r,t}g]dr$.
		From \eqref{fml.Tf}, we have
		\begin{align*}
			I_1
			&=\int_{\Rd} [p_{\Sigma_{s,t}(y)}(y-x)g(y)-p_{\Sigma_{s,t}(x)}(y-x)g(x)]dy
			\\&=\int_{\Rd} [p_{\Sigma_{s,t}(y)}(y-x)-p_{\Sigma_{s,t}(x)}(y-x)]g(y)dy+\int_{\Rd} [g(y)-g(x)]p_{\Sigma_{s,t}(x)}(y-x)dy
			\\&=I_{11}+I_{12}.
		\end{align*}
		Using \cref{con.A}\ref{con.Aholder}, it is straightforward to verify that $\lambda^{-1}I\le \Sigma_{s,t}(y)\Sigma_{s,t}(x)^{-1}\le \lambda I$ and $\|I- \Sigma_{s,t}(x)\Sigma_{s,t}(y)^{-1}\|\le \lambda |x-y|^\alpha$ for some finite constant $\lambda$. Then, we apply \cref{lem.pSig} to get
		\begin{align*}
			|I_{11}|\les\int_{\Rd}|x-y|^\alpha p_{N(t-s)}(x-y)|g(y)|dy.
		\end{align*}
		Using Minkowski inequality, we obtain from the above that $\|I_{11}\|_{L_p(\Rd)}\les (t-s)^{\frac \alpha2}\|g\|_{L_p(\Rd)}$.
		From the Hardy--Littlewood maximal inequality, there is a non-negative function $h\in L_p(\Rd)$ such that $\|h\|_{L_p(\Rd)}\les\|\nabla g\|_{L_p(\Rd)}$ and
		\begin{align*}
			|g(y)-g(x)|\le|x-y|(h(x)+h(y)),\quad \text{a.e.}\quad x,y\in\Rd.
		\end{align*}
		Using ellipticity and the above estimate, we have that
		\begin{align*}
			|I_{12}|
			&\le\int_{\Rd}|x-y|(h(x)+h(y))p_{\Sigma_{s,t}(x)}(x-y)dy
			\\&\les h(x)(t-s)^{\frac12}+ \int_{\Rd}|x-y|h(y)p_{N(t-s)}(x-y)dy.
		\end{align*}
		Applying Minkowski inequality, we obtain that $\|I_{12}\|_{L_p(\Rd)}\les (t-s)^{1/2}\|h\|_{L_p(\Rd)}\les (t-s)^{1/2}\|\nabla g\|_{L_p(\Rd)}$.
		Applying \cref{thm.QLp12} and \cref{lem.HLp}, we have
		\begin{align*}
			\|I_2\|_{L_p(\Rd)}
			\les\int_s^t\|H^n_{r,t}g\|_{L_p(\Rd)}dr
			\les\int_s^t(t-r)^{\frac \alpha2-1}\|g\|_{L_p(\Rd)}dr
			\les (t-s)^{\frac \alpha2}\|g\|_{L_p(\Rd)}.
		\end{align*}
		Combining the estimates for $I_{11},I_{12}$ and $I_2$, we obtain the result.
	\end{proof}

\subsection{Moment estimates} 
\label{sub:regularizing_properties_of_the_discrete_paths}
	We consider the Euler--Maruyama scheme
	\begin{align}\label{eqn.EMnob}
		\bar X^n_t=x+\int_0^t \sigma(s,\bar X^n_{k_n(s)})dB_s,
	\end{align}
	where $x_0$ is a $\cff_0$-random variable.
	By Markov property, for every $s\in D_n$ and bounded measurable $f$, we have $\E[f(\bar X^n_t)|\cff_s]=Q^n_{s,t}f(\bar X^n_s)$.
	
	\begin{proposition}\label{prop.momentg}
		Let $\bar X^n$ be the solution to \eqref{eqn.EMnob}.
		
		(i) Let $h$ be a measurable function such that $\|h\|_{L_\rho(\Rd)}$ is finite for some $\rho\in(0,\infty]$.  Then for every $r,v\in[0,1]$, $r-v\ge2/n$,
		\begin{align}\label{est.1fLp}
			\|h(\bar X^n_r)\|_{L_ \rho(\Omega|\cff_v)}\le N \|h\|_{L_ \rho(\Rd)}\left(r-v\right)^{-\frac d{2 \rho}}.
		\end{align}	

		(ii) Let $f$ be a function in $L_p(\Rd)$, $g$ be a function in $L_{1,p}(\Rd)\cap L_\infty(\Rd)$ for some $p\in[1,\infty)$.
		Then, for every $r,s,v\in[0,1]$ such that $r-v\ge2/n$, $r> k_n(s)+3/n$ and $s-v\ge2/n$,
		\begin{equation}\label{est.ffLp}
			\|\E_s(f(\bar X^n_r)-f(\bar X^n_{k_n(r)}))\|_{L_p(\Omega|\cff_v)}\le N(1/n)^{\frac \alpha2}\|f\|_{L_p(\Rd)}\left(s-v\right)^{-\frac d{2p}}\left(r-s-\frac2n\right)^{-\frac \alpha2}
		\end{equation}
		and
		\begin{multline}\label{est.fgLp}
			\|\E_s(g(\bar X^n_r)f(\bar X^n_r)-g(\bar X^n_r)f(\bar X^n_{k_n(r)}))\|_{L_p(\Omega|\cff_v)}
			\le N (1/n)^{\frac \alpha2}\|f\|_{L_p(\Rd)}\left(s-v\right)^{-\frac d{2p}}
			\\\times\left[\|g\|_{L_\infty(\Rd)}\left(r-s-\frac 2n\right)^{-\frac \alpha2}+ \|g\|_{L_{1,p}(\Rd)} \left(r-s-\frac2n\right)^{-\frac d{2p}}\right].
		\end{multline}		
	\end{proposition}
	\begin{proof}
		(i)
		Put $\bar v=k_n(v)+1/n$.
		In the case when $\rho<\infty$, applying \cref{thm.QLp12} (with the choice $p_1=1$ and $p_2=\infty$), we have
		\begin{align*}
			\E[|h(\bar X^n_r)|^\rho|\cff_{\bar v}]=Q^n_{\bar v,r}[|h|^\rho](\bar X^n_{\bar v})\les \left(r-\bar v\right)^{-\frac d2}\||h|^\rho \|_{L_1(\Rd)}.
		\end{align*}
		Noting that $r-\bar v\ge (r-v)/2$, we obtain \eqref{est.1fLp} for any $\rho\in(0,\infty)$ from the above. When $\rho=\infty$, \eqref{est.1fLp} is trivial.

		(ii) Put $\bar s=k_n(s)+1/n$.
		Applying \eqref{est.1fLp} and \cref{cor.Qfsut} (noting that $\bar s-v\ge2/n$ and $r -\bar s>2/n$), we have
		\begin{align*}
			\|\E_{\bar s}(f(\bar X^n_r)-f(\bar X^n_{k_n(r)}))\|_{L_p(\Omega|\cff_v)}
			&=\|Q^n_{\bar s,r}f(\bar X^n_{\bar s})-Q^n_{\bar s,k_n(r)}f(\bar X^n_{\bar s})\|_{L_p(\Omega|\cff_v)}
			\\&\les\left(\bar s- v\right)^{-\frac d{2p}}\|Q^n_{\bar s,r}f-Q^n_{\bar s,k_n(r)}f\|_{L_p(\Rd)}
			\\&\les (1/n)^{\frac \alpha2}\|f\|_{L_p(\Rd)}\left(\bar s-v\right)^{-\frac d{2p}}(k_n(r)-\bar s)^{-\frac \alpha2}.
		\end{align*}
		We observe that $\bar s- v\ge s-v$, $k_n(r)-\bar s\ge r-s-2/n$ and
		\begin{align*}
			\|\E_s(f(\bar X^n_r)-f(\bar X^n_{k_n(r)}))\|_{L_p(\Omega|\cff_v)}\le \|\E_{\bar s}(f(\bar X^n_r)-f(\bar X^n_{k_n(r)}))\|_{L_p(\Omega|\cff_v)}.
		\end{align*}
		From here, we obtain \eqref{est.ffLp} by combining the previous estimates.

		Lastly, we show \eqref{est.fgLp}. We write
		\begin{align*}
			&g(\bar X^n_r)f(\bar X^n_r)-g(\bar X^n_r)f(\bar X^n_{k_n(r)})
			\\&=[(gf)(\bar X^n_r)-(gf)(\bar X^n_{k_n(r)})]
			+[(g(\bar X^n_{k_n(r)})-g(\bar X^n_r))f(\bar X^n_{k_n(r)})].
		\end{align*}
		We observe that $\|fg\|_{L_p(\Rd)}\le\|f\|_{L_p(\Rd)}\|g\|_{L_\infty(\Rd)}$ and apply \eqref{est.ffLp} to see that
		\begin{align*}
			\left\|\E_s\left((fg)(\bar X^n_r)-(fg)(\bar X^n_{k_n(r)})\right)\right\|_{L_p(\Omega|\cff_v)}
		\end{align*}
		is smaller than the right-hand side of \eqref{est.fgLp}. It suffices to estimate the $L_p(\Omega|\cff_v)$-norm of
		\begin{align*}
			A:=\E_{\bar s}[(g(\bar X^n_{k_n(r)})-g(\bar X^n_r))f(\bar X^n_{k_n(r)})].
		\end{align*}
		By conditioning on $\cff_{k_n(r)}$, we have
		\begin{align*}
			A=\E_{\bar s}[h(\bar X^n_{k_n(r)})]=Q^n_{\bar s,k_n(r))}h(\bar X^n_{\bar s}),
			\quad\text{where}\quad h=(g-Q^n_{k_n(r),r}g)f.
		\end{align*}
		Applying \eqref{est.1fLp},
		\begin{align*}
			\|A\|_{L_p(\Omega|\cff_v)}\les(\bar s-v)^{-\frac d{2p}} \|Q^n_{\bar s,k_n(r)} h\|_{L_p(\Rd)}.
		\end{align*}
		We continue by applying \cref{thm.QLp12} (with $p_1=p/2$ and $p_2=p$),
		\begin{align*}
			\|A\|_{L_p(\Omega|\cff_v)}\les(\bar s-v)^{-\frac d{2p}}(k_n(r)-\bar s)^{-\frac d{2p}}\|h\|_{L_{p/2}(\Rd)}.
		\end{align*}
		By H\"older inequality
		\begin{align*}
			\|h\|_{L_{p/2}(\Rd)}\le \|Q_{k_n(r),r}g-g\|_{L_p(\Rd)}\|f\|_{L_p(\Rd)}
		\end{align*}
		and by \eqref{est.Qgg},
		\begin{align*}
			\|Q_{k_n(r),r}g-g\|_{L_p(\Rd)}\les(r-k_n(r))^{\frac \alpha2}\|g\|_{L_{1,p}(\Rd)}.
		\end{align*}
		Combining the previous estimates, we see that $\|A\|_{L_p(\Omega|\cff_v)}$ is also smaller than the right-hand side of \eqref{est.fgLp}, finishing the proof.
	\end{proof}
	\begin{remark}\label{rmk.vinDn}
		Concerning \cref{prop.momentg}(i), if $v\in D_n$, then there exists a constant $N=N(d,\rho,\alpha,K_1,K_2)$ such that for every $r>v$, we have
		\begin{align}\label{est.1fLprmk}
			\|h(\bar X^n_r)\|_{L_\rho(\Omega|\cff_v)}\le N \|h\|_{L_\rho(\Rd)}\left(r-v\right)^{-\frac d{2\rho}}.
		\end{align}	
		Indeed, the inequality is trivial when $\rho=\infty$. When $\rho<\infty$ and when $v\in D_n$, \cref{thm.QLp12} is applied directly, which yields
		\begin{align*}
			\E[|h(\bar X^n_r)|^\rho|\cff_{ v}]=Q^n_{ v,r}[|h|^\rho](\bar X^n_{ v})\les \left(r- v\right)^{-\frac d2}\||h|^\rho \|_{L_1(\Rd)}.
		\end{align*}
		From here, we deduce \eqref{est.1fLprmk}.		
	\end{remark}

	\begin{proposition}\label{prop.fXX2}
		Let $\bar X^n$ be the solution to \eqref{eqn.EMnob}.
		Let $f\in\mathbb{L}_p^q([0,1])$ and $g\in\LL^q_{1,p}([0,1])\cap\LL^\infty_\infty([0,1])$, with $p,q\in[2,\infty)$ satisfying $\frac{d}{p}+\frac{2}{q}<1$. Let $v\in[0,1-4/n]$ be a fixed number, $n\ge4$ is an integer. Then for every $v+\frac 4n\le S\le T\le1$, one has the bound
		\begin{multline}\label{est.fgXX}
			\Vert \int_S^Tg(r,\bar X^n_r)[ f(r,\bar X^n_r)-f(r,\bar X^n_{k_n(r)})]\dif r\Vert_{L_p(\Omega|\cff_v)}
			\\\leq N[(1/n)^{\frac \alpha2} +(1/n)^{\frac12}\log(n)](\|g\|_{\LL^\infty_\infty([S,T])}+ \|g\|_{\LL^q_{1,p}([S,T])} )\Vert f\Vert_{\mathbb{L}_p^q([S,T])},
		\end{multline}
		where $N=N(d,p,q,K_1,K_2)$ is a constant.
	\end{proposition}
	\begin{proof}
		Let $S,T$ be such that $v+4/n\le S\le T\le1$.
		By linearity, we can assume that $\|f\|_{\LL^q_p([S,T])}=\|g\|_{\LL^\infty_\infty([S,T])}+ \|g\|_{\LL^q_{1,p}([S,T])}= 1$.

		For each $(s,t)\in \Delta_2([S,T])$, define
		\begin{align*}
		A_{s,t}:=\E_s\int_s^tg(r,\bar X^n_r) (f(r,\bar X^n_r)-f(r,\bar X^n_{k_n(r)}))\dif r.
		\end{align*}
		We treat two cases $t\leq k_n(s)+\frac{4}{n}$ and $t>k_n(s)+\frac{4}{n}$ separately as following.

		\textit{Case 1.} For $t\in(s,k_n(s)+\frac{4}{n}]$, by triangle inequality and \eqref{est.1fLp} (note that $k_n(r)-v\ge k_n(s)-v\ge s-v-1/n\ge2/n$) we have
		\begin{align*}
		\|A_{s,t}\|_{L_p(\Omega|\cff_v)}
		&\leq\|g\|_{\LL^\infty_\infty([S,T])} \int_s^t\Vert f(r,\bar X^n_r)\Vert_{L_p(\Omega|\cff_v)}+\Vert f(r,\bar X^n_{k_n(r)})\Vert_{L_p(\Omega|\cff_v)} \dif r
		\\&\lesssim\int_s^t \left(k_n(r)-v\right)^{-\frac{d}{2p}}\Vert f(r,\cdot)\Vert_{L_p(\mathbb{R}^d)}\dif r.
		\end{align*}
		Note that $k_n(r)-v\ge k_n(s)-v\ge (s-v)/2$, applying H\"older inequality and the fact that $t-s\le4/n$, we have
		\begin{align*}
			\int_s^t \left(k_n(r)-v\right)^{-\frac{d}{2p}}\Vert f(r,\cdot)\Vert_{L_p(\mathbb{R}^d)}\dif r
			&\les (s-v)^{-\frac d{2p}}\|f\|_{\LL^q_p([s,t])}(t-s)^{1-\frac1q}
			\\&\les (1/n)^{\frac12} (s-v)^{-\frac d{2p}}\|f\|_{\LL^q_p([s,t])}(t-s)^{\frac12-\frac1q}.
		\end{align*}
		This gives
		\begin{align}\label{X.est1}
			\|A_{s,t}\|_{L_p(\Omega|\cff_v)}\les (1/n)^{\frac12} (s-v)^{-\frac d{2p}}\|f\|_{\LL^q_p([s,t])}(t-s)^{\frac12-\frac1q}.
		\end{align}
		
		\textit{Case 2.} When $t\in(k_n(s)+\frac{4}{n},1]$, by triangle inequality,
		\begin{align*}
		\|A_{s,t}\|_{L_p(\Omega|\cff_v)}&\le\|A_{s,k_n(s)+4/n}\|_{L_p(\Omega|\cff_v)}
		\\&\quad+\int_{k_n(s)+\frac{4}{n}}^t\Vert \E_s[g(r,\bar X^n_r)(f(r,\bar X^n_r)-f(r,\bar X^n_{k_n(r)}))]\Vert_{L_p(\Omega|\cff_v)} \dif r
		\\&=:I_1+I_2.
		\end{align*}
		For $I_1$, from \eqref{X.est1} we know that
		\begin{align*}
		I_1\lesssim (1/n)^{\frac12} (s-v)^{-\frac d{2p}}\|f\|_{\LL^q_p([s,t])}\left(k_n(s)+\frac4n-s\right)^{\frac12-\frac1q}.
		\end{align*}
		Because  $k_n(s)+\frac4n-s\le t-s$,
		we get
		\begin{align*}
			I_1\lesssim (1/n)^{\frac12} (s-v)^{-\frac d{2p}}\|f\|_{\LL^q_p([s,t])}(t-s)^{\frac12-\frac1q}.
		\end{align*}
		Applying \eqref{est.fgLp} and H\"older inequality, we have for $I_2$
		\begin{multline*}
			I_2\lesssim (1/n)^{\frac \alpha2}(s-v)^{-\frac{d}{2p}} \|g\|_{\LL^\infty_\infty([S,T])}\int_{k_n(s)+\frac{4}{n}}^t \left(r-s-\frac2n\right)^{-\frac{\alpha}{2}} \Vert f(r,\cdot)\Vert_{L_p(\mathbb{R}^d)}\dif r
			\\\quad+(1/n)^{\frac \alpha2}(s-v)^{-\frac{d}{2p}}\int_{k_n(s)+\frac{4}{n}}^t \left(r-s-\frac2n\right)^{-\frac d{2p}}\|g(r,\cdot)\|_{L_{1,p}(\Rd)} \Vert f(r,\cdot)\Vert_{L_p(\mathbb{R}^d)}\dif r
			\\\lesssim (1/n)^{\frac \alpha2} (s-v)^{-\frac{d}{2p}}\left[\|f\|_{\LL^q_p([s,t])}(t-s)^{1-\frac \alpha2-\frac1q}+  \|g\|_{\LL^q_{1,p}([s,t])} \|f\|_{\LL^q_p([s,t])}(t-s)^{1-\frac d{2p}-\frac 2q}\right].
		\end{multline*}

		Combining these two cases together we obtain that for $v+4/n\le s\le t\le 1$,
		\begin{multline}\label{X.A}
		\Vert A_{s,t}\Vert_{L_p(\Omega|\cff_v)}
		\lesssim (1/n)^{\frac12}(s-v)^{-\frac d{2p}}\|f\|_{\LL^q_p([s,t])}(t-s)^{\frac12-\frac1q}
		\\+ (1/n)^{\frac \alpha2}(s-v)^{-\frac{d}{2p}}\left[\|f\|_{\LL^q_p([s,t])}(t-s)^{1-\frac{\alpha}{2}-\frac{1}{q}}+\|g\|_{\LL^q_{1,p}([s,t])} \|f\|_{\LL^q_p([s,t])}(t-s)^{1-\frac d{2p}-\frac 2q}\right].
		\end{multline}
		Furthermore, for $u\in(s,t)$, we have $\E_s \delta A_{s,u,t}=0$.
		Let $w$ be the continuous control on $\Delta([v+4/n,1])$ defined by
		\begin{align*}
			w(s,t)&=\left[(s-v)^{-\frac d{2p}}\|f\|_{\LL^q_p([s,t])}(t-s)^{\frac12-\frac 1q}\right]^2
			\\&\quad+\left[(s-v)^{-\frac d{2p}}\|f\|_{\LL^q_p([s,t])}(t-s)^{1-\frac \alpha2-\frac 1q}\right]^{1/(1-\frac \alpha2)}
			\\&\quad+\left[(s-v)^{-\frac d{2p}}\|g\|_{\LL^q_{1,p}([s,t])} \|f\|_{\LL^q_p([s,t])}(t-s)^{1-\frac d{2p}-\frac 2q}\right]^{1/(1-\frac d{2p})}
			\\&\quad+(s-v)^{-\frac d{2p}}\|f\|_{\LL^q_p([s,t])}(t-s)^{1-\frac1q}.
		\end{align*}
		Denote
		\begin{align*}
		 \mathcal{A}_{t}:=\int_0^t(f(r,\bar X^n_r)-f(r,\bar X^n_{k_n(r)}))\dif r,\quad J_{s,t}:=\delta\mathcal{A}_{s,t}-A_{s,t}.
		\end{align*}
		Using similar estimates leading to \eqref{X.est1}, we have
		\begin{align*}
			\|J_{s,t}\|_{L_p(\Omega|\cff_v)}\les (s-v)^{-\frac d{2p}}\|f\|_{\LL^q_p([s,t])}(t-s)^{1-\frac1q}\les w(s,t).
		\end{align*}
		Furthermore, $\delta J_{s,u,t}=-\delta A_{s,u,t}$ and we derive from \eqref{X.A} that
		\begin{align*}
			\|\delta J_{s,u,t}\|_{L_p(\Omega|\cff_v)}\les(1/n)^{\frac12}w(s,t)^{\frac12}+  (1/n)^{\frac \alpha2}\left[w(s,t)^{1-\frac \alpha2}+w(s,t)^{1-\frac d{2p}}\right].
		\end{align*}
		It is obvious that $\E_s J_{s,t}=0$ and hence $\E_s \delta J_{s,u,t}=0$.
		Applying \cref{lem.Davie_iteration}, we have
		\begin{multline*}
			\|J_{s,t}\|_{L_p(\Omega|\cff_v)}
			\\\les [(1/n)^{\frac \alpha2} +(1/n)^{\frac12}\log(n)]\left[w(s,t)^{\frac12}+w(s,t)^{1-\frac \alpha2} +w(s,t)^{1-\frac d{2p}}+ w(s,t)\right].
		\end{multline*}
		By triangle inequality and \eqref{X.A}, this implies that
		\begin{multline*}
			\|\delta\caa_{s,t}\|_{L_p(\Omega|\cff_v)}
			\\\les [(1/n)^{\frac \alpha2} +(1/n)^{\frac12}\log(n)]\left[w(s,t)^{\frac12}+w(s,t)^{1-\frac \alpha2} +w(s,t)^{1-\frac d{2p}}+ w(s,t)\right].
		\end{multline*}
		Because $\|f\|_{\LL^q_p([s,t])}\le\|f\|_{\LL^q_p([S,T])}=1$, we have
		\begin{align}
		 	w(s,t)&\le \left[(s-v)^{-\frac d{2p}}(t-s)^{\frac12-\frac1q}\right]^2
		 	\nonumber\\&\quad+\left[(s-v)^{-\frac d{2p}}(t-s)^{1-\frac \alpha2-\frac1q}\right]^{1/(1-\frac \alpha2)}
		 	+\left[(s-v)^{-\frac d{2p}}(t-s)^{1-\frac d{2p}-\frac2q}\right]^{1/(1-\frac d{2p})}
		 	\nonumber\\&\quad+(s-v)^{-\frac d{2p}}(t-s)^{1-\frac1q}
		 	\label{tmp.89253}
		\end{align}
		and hence
		\begin{align*}
			w(s,t)^{\frac12}+w(s,t)^{1-\frac \alpha2} +w(s,t)^{1-\frac d{2p}}+w(s,t)
			\les \sum_{i=1}^{16}(s-v)^{-\eta_i}(t-s)^{\tau_i},
		\end{align*}
		where for each $i=1,\ldots,16$; $\eta_i,\tau_i\in[0,1]$ are some constants such that $\tau_i- \eta_i>0$. The constants $\eta_i,\tau_i$'s can be  calculated explicitly by applying the powers $1/2$, $1- \alpha/2$, $1-d/(2p)$ and $1$ to the singularity exponents and H\"older exponents in the right-hand side of \eqref{tmp.89253}, however, their exact values are non-essential. 
		That $\tau_i- \eta_i$ is positive  for each $i$ because the sums of the H\"older exponents and the corresponding singular exponents of each factor on the right-hand side of \eqref{tmp.89253} are positive.
		Hence, we deduce from the above estimate that
		\begin{multline*}
			\Vert \int_S^Tg(r,\bar X^n_r) (f(r,\bar X^n_r)-f(r,\bar X^n_{k_n(r)}))\dif r\Vert_{L_p(\Omega|\cff_v)}
			\\\leq N[(1/n)^{\frac \alpha2} +(1/n)^{\frac12}\log(n)] \sum_{i=1}^{16}(S-v)^{-\eta_i}(T-S)^{\tau_i},
		\end{multline*}
		which holds for every $v+4/n\le S\le T\le1$. We then apply \cref{lem.Bellising} to obtain \eqref{est.fgXX}.
	\end{proof}
	\begin{proposition}\label{prop.gfnob}
		Let $\bar X^n$ be the solution to \eqref{eqn.EMnob}.
		Let $f\in\mathbb{L}_p^q([0,1])\cap \mathbb{L}^q_\infty([0,1])$ and $g\in \LL^q_{1,p}([0,1])\cap\LL^\infty_\infty([0,1])$, with $p,q\in[2,\infty)$ satisfying $\frac{d}{p}+\frac{2}{q}<1$.
		
		As in \cref{prop.supB}, we put $\beta_n(f)=\sup_{r\in D_n}\|f\|_{\LL^q_\infty([r,r+1/n])}$.
		Then for any $\bar p\in(0,p)$, there exists a constant $N=N(d,p,q,\bar p)$ such that
		\begin{multline*}
			\|\sup_{t\in[0,1]}|\int_0^tg(r,\bar X^n_r) [f(r,\bar X^n_r)-f(r,\bar X^n_{k_n(r)})]dr|\|_{L_{\bar p}(\Omega)}
			\le N\left[\|g\|_{\LL^\infty_\infty([0,1])}+\|g\|_{\LL^q_{1,p}([0,1])}\right]
			\\\times\left[(1/n)^{1-\frac1q}\beta_n(f)+ (1/n)^{\frac \alpha2}\|f\|_{\LL^q_p([0,1])}+(1/n)^{\frac12}\log(n)\|f\|_{\LL^q_p([0,1])}\right].
		\end{multline*}
	\end{proposition}
	\begin{proof}
		This result is a consequence of \cref{prop.fXX2} and \cref{lem.lenglart}.
		The proof is analogous to that of \cref{prop.supB}, hence,  omitted.
	\end{proof}
	\begin{lemma}\label{Girsanov}
		Let $\bar X^n$ be the solution to \eqref{eqn.EMnob} and $f$ be a function in $\LL^{q_1}_{p_1}([0,1])$ for some $p_1,q_1\in[1,\infty]$ satisfying $\frac d{p_1}+\frac2{q_1}<2$.
		Then
		\begin{align}\label{est.expnok}
			\E\exp\left(\int_0^1 f(r,\bar X^n_r)dr\right)\le 2\exp\left(N\|f\|_{\LL^{q_1}_{p_1}([0,1])}^{1/(1-d/(2p_1))} \right),
		\end{align}
		where $N$ depends only on $d,\alpha,p_1,q_1, K_1,K_2$.

		Assume additionally that there are continuous control $w_0$ on $\Delta$ and positive constants $M,\gamma_0$ such that
		\begin{align}\label{con.fGirs}
			(1/n)^{1-\frac1{q_1}}\|f\|_{\LL^{q_1}_\infty([s,t])}\le w_0(s,t)^{ \gamma_0} \quad\forall \ 0\le t-s\le1/n
		\end{align}
		and
		\begin{align*}
			\|f\|_{\LL^{q_1}_{p_1}([0,1])}+w_0(0,1)\le M.
		\end{align*}
		Then there exists a finite constant $\bar N$ which depends only on $M,\gamma_0,d,\alpha,p_1,q_1, K_1,K_2$ such that
		\begin{align}\label{est.expk}
			\E\exp\left(\int_0^1 f(r,\bar X^n_{k_n(r)})dr\right)\le \bar N.
		\end{align}
	\end{lemma}
	
	\begin{proof}
		We can assume without loss of generality that $f$ is nonnegative.
		We rely on \cref{lem.Khasminski}.
		For each $(s,t)\in \Delta$, define $\bar s=k_n(s)+1/n$. For each $r\in(s,\bar s)$, we write $\bar X^n_r=\bar X^n_s+\int_s^r \sigma(\theta,\bar X^n_{k_n(s)})dB_\theta$ so that
		\begin{align*}
			\E_s\int_s^{\bar s \wedge t}f(r,\bar X^n_r)dr= \int_s^{\bar s\wedge t}P_{\Sigma_{s,r}(\bar X^n_{k_n(s)})}f_r(\bar X^n_s)dr.
		\end{align*}
		Using ellipticity of $\Sigma_{s,r}$ and Gaussian estimates, we see that $\sup_y\|P_{\Sigma_{s,r}(y)}f_r\|_{L_\infty(\Rd)}\les (r-s)^{-d/(2p_1)}\|f_r\|_{L_{p_1}(\Rd)}$. Hence, using this estimate and H\"older inequality, we have
		\begin{align*}
			\E_s\int_s^{\bar s \wedge t}f(r,\bar X^n_r)dr 
			\les\int_s^t (r-s)^{-\frac d{2p_1}}\|f_r\|_{L_{p_1}(\Rd)}dr\les\|f\|_{\LL^{q_1}_{p_1}([s,t])}(t-s)^{1-\frac d{2p_1}-\frac 1{q_1}}.
		\end{align*}
		On the interval $(\bar s\wedge t,t)$, we use \eqref{est.1fLprmk} and H\"older inequality to see that
		\begin{align*}
			\E_s\int_{\bar s\wedge t}^tf(r,\bar X^n_r)dr
			\les\int_{\bar s\wedge t}^t(r-\bar s)^{-\frac d{2p_1}}\|f_r\|_{L_{p_1}(\Rd)}dr
			\les \|f\|_{\LL^{q_1}_{p_1}([s,t])}(t-s)^{1-\frac d{2p_1}-\frac 1{q_1}}.
		\end{align*}
		It follows that
		\begin{align*}
			\E_s\int_s^t f(r,\bar X^n_r)dr\les\|f\|_{\LL^{q_1}_{p_1}([s,t])}(t-s)^{1-\frac d{2p_1}-\frac 1{q_1}}.
		\end{align*}
		Observe that $w$ defined by $w(s,t)^{1-\frac d{2p_1}}=\|f\|_{\LL^{q_1}_{p_1}([s,t])}(t-s)^{1-\frac d{2p_1}-\frac 1{q_1}}$ is a continuous control on $\Delta$. Applying \cref{lem.Khasminski}, we obtain \eqref{est.expnok}.
		
		The second part is obtained in an analogous way.
		For each $(s,t)\in \Delta$, define $\tilde s=k_n(s)+2/n$, using H\"older inequality, \eqref{con.fGirs} and \eqref{est.1fLp}, we have
		\begin{align*}
			\E_s\int_s^{\tilde s\wedge t}f(r,\bar X^n_{k_n(r)})dr\le \int_s^{\tilde s\wedge t}\|f_r\|_{L_\infty(\Rd)}dr
			\les \|f\|_{\LL^{q_1}_\infty([s,\tilde s\wedge t])}(1/n)^{1-\frac1{q_1}}
			\les w_0(s,t)^{ \gamma_0}
		\end{align*}
		and
		\begin{align*}
			\E_s\int_{\tilde s \wedge t}^tf(r,\bar X^n_{k_n(r)})dr
			&=\int_{\tilde s \wedge t}^t\E_{ s}f(r,\bar X^n_{k_n(r)})dr
			\les\int_{\tilde s\wedge t}^t (k_n(r)- s)^{-\frac d{2p_1}}\|f_r\|_{L_{p_1}(\Rd)}dr
			\\&\les\int_{\tilde s\wedge t}^t (r- s)^{-\frac d{2p_1}}\|f_r\|_{L_{p_1}(\Rd)}dr
			\les\|f\|_{\LL^{q_1}_{p_1}([s,t])}(t-s)^{1-\frac d{2p_1}-\frac1{q_1}}.
		\end{align*}
		Hence, $\E_s\int_s^t f(r,\bar X^n_{k_n(r)})dr\les w_0(s,t)^{\gamma_0}+w(s,t)^{1-\frac d{2p_1}}$, where $w$ is the control defined previously.
		Applying \cref{lem.Khasminski} and \cref{rmk.khas12}, we obtain \eqref{est.expk}.
	\end{proof}
	\begin{remark}
		From \cref{rmk.khas12}, one can compute $\bar N$ explicitly, however \eqref{est.expk} is sufficient for our considerations.
	\end{remark}

	\begin{proof}[\bf Proof of \cref{prop.gf}]
	For any  continuous process $Z$, we define
	\begin{align*}
	   \chh(Z)=\sup_{t\in[0,1]}\Big|\int_0^tg(r,Z_r) [f(r,Z_r)-f(r,Z_{k_n(r)})]dr\Big|.
	\end{align*}
	 Let $\bar X^n$	be the solution to \eqref{eqn.EMnob} and
	\begin{align*}
	    \rho:=&\exp\left(-\int_0^1(\sigma^{-1}b^n)(r,\bar X_{k_n(r)}^n)dB_r-\frac{1}{2}\int_0^1\Big|(\sigma^{-1}b^n)(r,\bar X_{k_n(r)}^n)\Big|^2dr\right).
	\end{align*}
	Using the fact that $\sigma^{-1}b^n\in \LL^q_\infty([0,1])$, we see that $\rho$ is a probability density.
	 It follows from Girsanov theorem  (see e.g. \cite[IV Corollary of Theorem 4.2]{IW})
and H\"older inequality for $\frac{1}{\gamma'}+\frac{1}{\gamma}=1$ with $\gamma>1$ close enough to $1$ such that $\gamma\bar p<p$
	\begin{align}\label{hX}
	    \mE \chh(X^n)^{\bar p}=\mE (\rho\chh(\bar X^n)^{\bar p})\leq [ \mE\chh(\bar  X^n)^{\gamma\bar p}]^{1/\gamma}[\mE\rho^{\gamma'}] ^{1/\gamma'}.
	\end{align}
	From \cref{prop.gfnob} 
	we immediately get that
	\begin{multline}\label{hbar}
	(\E\chh(\bar X^n)^{\gamma\bar p})^{\frac1{\gamma\bar p}}=
	 \Vert \chh(\bar X^n) \Vert_{L_{\gamma\bar p}(\Omega)}
	 \leq N\left[\|g\|_{\LL^\infty_\infty([0,1])}+\|g\|_{\LL^q_{1,p}([0,1])}\right]
		\\\times\left[(1/n)^{1-\frac1q}\beta_n(f)+ (1/n)^{\frac \alpha2}\|f\|_{\LL^q_p([0,1])}+(1/n)^{\frac12}\log(n)\|f\|_{\LL^q_p([0,1])}\right].
	\end{multline}
	Using 
	Cauchy--Schwarz inequality, we have
	\begin{align*}
	    \mE\rho^{\gamma'}
	    &=\E\exp\left(-\gamma'\int_0^1(\sigma^{-1}b^n)(r,\bar X_{k_n(r)}^n)dB_r-\frac{\gamma'}{2}\int_0^1\Big|(\sigma^{-1}b^n)(r,\bar X_{k_n(r)}^n)\Big|^2dr\right)
	    \\&\le \left[\E\exp\left(-2\gamma'\int_0^1(\sigma^{-1}b^n)(r,\bar X_{k_n(r)}^n)dB_r-2 {\gamma'}^2\int_0^1\Big|(\sigma^{-1}b^n)(r,\bar X_{k_n(r)}^n)\Big|^2dr\right)\right]^{\frac12}
	    \\&\quad\times\left[\E\exp\left((2 {\gamma'}^2-\gamma')\int_0^1\Big|(\sigma^{-1}b^n)(r,\bar X_{k_n(r)}^n)\Big|^2dr\right)\right]^{\frac12}.
	\end{align*}
	In the right-hand side above, the first factor is identical to $1$ by martingale properties. For the second factor, we recall \cref{con.B} and the uniform ellipticity of $\sigma$, which imply that the function $f:=|\sigma^{-1}b^n|^2$ belongs to $\LL^{q/2}_{p/2}([0,1])\cap \LL^{q/2}_{\infty}([0,1])$ and satisfies
	\begin{align*}
		(1/n)^{1-\frac2q}\|f\|_{\LL^{q/2}_\infty([s,t])}\les \left[(1/n)^{\frac12-\frac1q}\|b^n\|_{\LL^q_\infty([s,t])}\right]^2\les \mu(s,t)^{2 \theta}, \quad\forall \ 0\le t-s\le1/n.
	\end{align*}
	Applying \cref{Girsanov}, we see that $\E\exp\left((2 {\gamma'}^2- \gamma')\int_0^1|(\sigma^{-1}b^n)(r,\bar X_{k_n(r)}^n)|^2dr\right)$ is bounded uniformly by a finite constant.
	Hence, we have shown that $ \E \rho^{\gamma'}$ is bounded uniformly in $n$.
	Combining with \eqref{hX} and \eqref{hbar}, we obtain \eqref{gfX}.
	\end{proof}
	
\section{Analysis of the continuum paths} 
\label{sec:regularizing_properties_of_the_continuum_paths}
	To show \cref{thm.alpha}, we need the following result, extending the results of \cref{sec.BM} to functionals of solutions to \eqref{sde00}.
	\begin{theorem}\label{thm.prealpha}
		Let $X$ be the solution to \eqref{sde00}.

		(i) Assuming Conditions \ref{con.A}\ref{con.Aholder} and \ref{con.B}. Let $h$ be a function in $\LL^{q_1}_{p_1}([0,1])$ for some $p_1,q_1\in[1,\infty]$ satisfying $\frac d{p_1}+\frac2{q_1}<2$. Then for  every $m\ge1$, there exists a constant $N=N(d,p,q,m)$ such that
		\begin{align*}
			\|\int_0^1 h(r,X_r)dr\|_{L_m(\Omega)}\le N\|h\|_{\LL^{q_1}_{p_1}([0,1])}.
		\end{align*}

		(ii) Assuming  \crefrange{con.A}{con.B} with $q_0=\infty$ and $\frac1p+\frac1{p_0}<1$. Let $g$ be a function in $\LL^{q_2}_{p_2}([0,1])$ and let $\nu\in[0,1)$ such that $\frac{d}{p_2}+\frac{2}{q_2}+\nu<2$. Then for any $\bar p\in(0,p_2)$, there exists a constant $N=N(\nu,d,p,q,p_2,q_2,\bar p)$ such that
		\begin{align*}
			\|\sup_{t\in[0,1]}|\int_0^t g(r,X_r)dr|\|_{L_{\bar p}(\Omega)}\le
			N \|g\|_{\LL^{q_2}_{-\nu,p_2}([0,1])}.
		\end{align*}

		(iii) Assuming  \crefrange{con.A}{con.B} with $q_0=\infty$ and $\frac1p+\frac1{p_0}<1$. Let $g$ be a function in $\LL^{q_2}_{p_2}([0,1])$ with $\frac{d}{p_2}+\frac{2}{q_2}<1$, $\Gamma$ be a nonnegative number and $w_0$ be a continuous control on $\Delta$. We assume that for every $(s,t)\in \Delta$,
		\begin{align*}
		 	\|g\|_{\LL^{q_2}_{-1,p_2}([s,t])}\le \Gamma w_1(s,t)^{\frac1{q_2}}
		 	\tand
		 	\|g\|_{\LL^{q_2}_{p_2}([s,t])}\le w_1(s,t)^{\frac1{q_2}}.
		\end{align*}
		Then for any $\bar p\in(0,p_2)$, there exists a constant $N=N(\nu,d,p_2,q_2,p,q,\bar p)$ such that
		\begin{align*}
			\|\sup_{t\in[0,1]}|\int_0^t g(r,X_r)dr|\|_{L_{\bar p}(\Omega)}\le
			N \Gamma(1+|\log(\Gamma)|)w_1(0,1)^{\frac1{q_2}}.
		\end{align*}
	\end{theorem}
	Similar to the methods in \cref{sec.BM,sec:regularizing_properties_of_the_discrete_paths}, first we derive some analytic estimates on the transition operators associated to solutions of \eqref{sde00} without drift.
	Using these estimates, one can apply stochastic sewing techniques (\cref{lem.Davie_iteration}) and Girsanov theorem to obtain the desired moment bounds. 
	
	We begin with  moment bounds on the solutions to the driftless SDEs.
	Let $\bar X$ be a solution to SDE
	\begin{align}\label{eqn.SDEnob}
		d\bar X_t=\sigma(t,\bar X_t)dB_t, \quad \bar X_0=x\in\Rd.
	\end{align}
	Under \cref{con.A}, it is well-known (see \cite{MR2190038}) that the probability law of $\bar X$ is unique and Markov. In fact, solutions to equation \eqref{eqn.SDEnob} are strongly unique under \cref{con.A}. This follows from the arguments in the proof of \cref{thm.main} in the following section. However, only the law of $\bar X$ is relevant to our considerations herein.
	Let $Q_{s,t}$ be the transition operator associated to $\bar X$. In particular, we have $\E (f(\bar X_t)|\cff_s)=Q_{s,t}f(\bar X_s)$ for any bounded measurable function $f$.
	
	\begin{lemma}
		Assuming \cref{con.A}\ref{con.Aholder}.
		Let $p_1,p_2\in[1,\infty]$, $p_1\le p_2$. There exists a constant $N=N(\alpha,d,p_1,p_2,K_1,K_2)$ such that for every $f\in L_{p_1}(\Rd)$ and $s\le t$,
		\begin{align}\label{cQ}
			\|Q_{s,t}f\|_{L_{p_2}(\Rd)}\le N(t-s)^{\frac d{2p_2}-\frac d{2p_1}}\|f\|_{L_{p_1}(\Rd)}.
		\end{align}
	\end{lemma}
	\begin{proof}
		Let $\bar X^n$ be the solution to the Euler--Maruyama scheme \eqref{eqn.EMnob}.
		It suffices to show that the laws of $\bar X^n$ converge to the law of $\bar X$ for \eqref{cQ} is then derived from \cref{thm.QLp12}.
		Let $\bP^n$ be the probability law of $\bar X^n$ on $C([0,1])$. Here $C([0,1])$ is the space of continuous functions $\omega:[0,1]\to\Rd$ equipped with the topology of uniform convergence, the Borel $\sigma$-algebra and the filtration $t\mapsto\cgg_t=\sigma\{\omega_s:s\in[0,t]\}$.
		Let $\phi$ be a smooth function with bounded derivatives.
		By It\^o formula, we see that
		\begin{align*}
			M^n_t(\omega)= \phi(\omega_t)- \phi(x)-\frac12\int_0^t a^{ij}(r,\omega_{k_n(r)}) \partial^2_{ij} \phi(\omega_{r})dr
		\end{align*}
		is a martingale under $\bP^n(d \omega)$. Define
		\begin{align*}
			M_t(\omega)= \phi(\omega_t)- \phi(x)-\frac12\int_0^t a^{ij}(r,\omega_r) \partial^2_{ij}\phi(\omega_{r})dr.
		\end{align*}
		
		It is easy to see that $\|\bar X^n_t-\bar X^n_s\|_{L_p(\Omega)}\les (t-s)^{1/2}$
		for any $p\ge2$ and $s\le t$. This implies that the probability laws $\{\bP^n\}_n$ are tight. Let $\bP$ be a probability measure such that $\bP^n$ converges to $\bP$ through a subsequence, which we still denote by $\bP^n$. Let $s\le t$ be fixed and $G\in\cgg_s$. We have
		\begin{align*}
			\int \delta M_{s,t}\1_G d\bP
			&=\int \delta M_{s,t}\1_G (d\bP-d\bP^n)
			+
			\int (\delta M_{s,t}-\delta M^n_{s,t})\1_G d\bP^n
			+\int \delta M^n_{s,t}\1_G d\bP^n
			\\&=:I_1+I_2+I_3.
		\end{align*}
		It is evident that $\lim_n I_1=0$. Using H\"older continuity of $a$   
		we have
		\begin{align*}
			|I_2|\les\int \left[\int_s^t|\omega_r- \omega_{k_n(r)}|^\alpha dr\right] d\bP^n(\omega)
			&\les\int_s^t\E|\bar X^n_r-\bar X^n_{k_n(r)}|^\alpha dr
			\\&\les\int_s^t|r-{k_n(r)}|^{\alpha/2} dr.
		\end{align*}
		This implies that $\lim_nI_2=0$. Because $M^n$ is a martingale under $\bP^n$, $I_3=0$. It follows that $\int \delta M_{s,t}\1_G d\bP=0$, and hence $M$ is a martingale under $\bP$. In other words, $\bP$ is a solution to the martingale problem associated to equation \eqref{eqn.SDEnob}, which is unique (\cite{MR2190038}). We have shown that $\{\bP^n\}_n$ has exactly one accumulating point, which is the law of \eqref{eqn.SDEnob}. This also means that $\bar X^n$ converges weakly to $\bar X$.		
	\end{proof}
	\begin{lemma}\label{lem.LrhobarX}
		Assuming \cref{con.A}\ref{con.Aholder}.
		Let $h$ be a measurable function on $\Rd$ such that $\|h\|_{L_\rho(\Rd)}$ is finite for some $\rho\in(0,\infty]$. Then there exists a finite constant $N=N(\alpha,d,\rho,K_1,K_2)$ such that for every $r,v\in[0,1]$, $r>v$,
		\begin{align}\label{est.Qf}
			\|h(\bar X_r)\|_{L_\rho(\Omega|\cff_v)}\le N(r-v)^{-\frac d{2 \rho}}\|h\|_{L_\rho(\Rd)}.
		\end{align}
	\end{lemma}
	\begin{proof}
		This is a direct consequence of \eqref{cQ}.  The argument is similar to that of \cref{prop.momentg}(i), hence, is omitted.
	\end{proof}

	The next result is a special case of \cref{thm.prealpha} when $b=0$, which is an analogue of \cref{prop.B2,prop.B3}.
	\begin{proposition}\label{prop.Xnob}
		Let $p\in(1,\infty)$, $q\in(2,\infty)$. 
		Let $\bar X$ be a solution to \eqref{eqn.SDEnob}.

		(i) Assuming \cref{con.A}\ref{con.Aholder}. Let $h$ be a function in $\LL^{q_1}_{p_1}([0,1])$ for some $p_1,q_1\in[1,\infty]$ satisfying $\frac d{p_1}+\frac2{q_1}<2$. Then for  every $m\ge1$, there exists a constant $N=N(d,p_1,q_1,m)$ such that
		\begin{align*}
			\|\int_0^1 h(r,\bar X_r)dr\|_{L_m(\Omega)}\le N\|h\|_{\LL^{q_1}_{p_1}([0,1])}.
		\end{align*}

		(ii) Assuming  \cref{con.A} with $q_0=\infty$ and $\frac1p+\frac1{p_0}<1$.
		Let $g$ be a function in $\LL^{q}_{p}([0,1])$ and let $\nu\in[0,1)$ such that $\frac{d}{p}+\frac{2}{q}+\nu<2$. Then for any $\bar p\in(0,p)$, there exists a constant $N=N(\nu,d,p,q,\bar p)$ such that
		\begin{align*}
			\|\sup_{t\in[0,1]}|\int_0^t g(r,\bar X_r)dr|\|_{L_{\bar p}(\Omega)}\le
			N \|g\|_{\LL^{q}_{-\nu,p}([0,1])}.
		\end{align*}

		(iii) Assuming  \cref{con.A} with $q_0=\infty$, $\frac1p+\frac1{p_0}<1$ and $\frac dp+\frac2q<1$. Let $g$ be a function in $\LL^q_p([0,1])$, $\Gamma$ be a nonnegative number and $w_1$ be a continuous control on $\Delta$. We assume that for every $(s,t)\in \Delta$,
		\begin{align*}
		 	\|g\|_{\LL^q_{-1,p}([s,t])}\le \Gamma w_1(s,t)^{\frac1q}
		 	\tand
		 	\|g\|_{\LL^q_p([s,t])}\le w_1(s,t)^{\frac1q}.
		\end{align*}
		Then for any $\bar p\in(0,p)$, there exists a constant $N=N(\nu,d,p,q,\bar p)$ such that
		\begin{align*}
			\|\sup_{t\in[0,1]}|\int_0^t g(r,\bar X_r)dr|\|_{L_{\bar p}(\Omega)}\le
			N \Gamma(1+|\log(\Gamma)|)w_1(0,1)^{\frac1q}.
		\end{align*}	
	\end{proposition}
	\begin{proof}
		(i) Using Minkowski inequality, \cref{lem.LrhobarX} and H\"older inequality, we have
		\begin{multline*}
			\|\int_s^th(r,\bar X_r)dr\|_{L_{p_1}(\Omega|\cff_s)}\le\int_s^t\|h(r,\bar X_r)\|_{L_{p_1}(\Omega|\cff_s)}dr
			\\\les\int_s^t (r-s)^{-\frac d{2p_1}}\|h_r\|_{L_{p_1}(\Rd)}dr
			\les(t-s)^{1-\frac d{2p_1}-\frac 1{q_1}}\|h\|_{\LL^{q_1}_{p_1}([s,t])}.
		\end{multline*}
		This implies that
		\begin{align}\label{tmp.2809Esh}
			\E_s\int_s^t|h(r,\bar X_r)|dr\les (t-s)^{1-\frac d{2p_1}-\frac 1{q_1}}\|h\|_{\LL^{q_1}_{p_1}([s,t])}.
		\end{align}
		From here, we apply \cref{lem.Khasminski} to obtain part (i).

		The proofs of parts (ii,iii) are similar to those of \cref{prop.B2,prop.B3}. The statistical estimates for Brownian motion are replaced by those obtained in \cref{thm.pde-1}.
		The condition $\frac dp+\frac 2q+\nu<2$ appears when applying \cref{lem.Bellising}.
		We only provide proof of (iii) while the proof of (ii) is left to the readers.

		(iii) For each $(s,t)\in \Delta$, put
		\begin{align*}
			A_{s,t}=\E_s\int_s^t g(r,\bar X_r)dr
			\tand J_{s,t}=\int_s^tg(r,\bar X_r)dr-A_{s,t}.
		\end{align*}
		Define the control $w$ by
		\begin{align*}
			w(s,t)=\left[(s-v)^{-\frac d{2p}}(t-s)^{\frac12-\frac1q}w_1(s,t)^{\frac1q}\right]^2+(s-v)^{-\frac d{2p}}(t-s)^{1-\frac1q}w_1(s,t)^{\frac1q}.
		\end{align*}
		
		Let $u^t\in\LL^q_{2,p}([0,t])$ be the solution (\cite[Theorem 2.1]{LX}) to
		\begin{align*}
		 	(\partial_s+\frac12 a^{ij}\partial^2_{ij})u+g=0, \quad u(t,\cdot)=0.
		\end{align*}
		Applying It\^o formula for non-degenerate diffusions (see \cite[Lemma 4.1]{XXZZ}), we see that $A_{s,t}=u^t_s(\bar X_s)$.
		Applying \eqref{est.Qf} and \cref{thm.pde-1}, we have
		\begin{align*}
			\|A_{s,t}\|_{L_p(\Omega|\cff_v)}
			&\les (s-v)^{-\frac d{2p}}\|u^t_s\|_{L_p(\Rd)}
			\les  (s-v)^{-\frac d{2p}}(t-s)^{\frac12-\frac1q}\|g\|_{\LL^q_{-1,p}([s,t])}.
		\end{align*}
		By our assumption, the previous estimate implies that $\|\delta J_{s,u,t}\|_{L_p(\Omega|\cff_v)}\les \Gamma w(s,t)^{1/2}$ for every $v<s\le u\le t\le1$. It is evident that $\E_sJ_{s,t}=0$ and hence $\E_s \delta J_{s,u,t}=0$. This verifies the conditions \eqref{con.dJ} and \eqref{con.EdJ} of \cref{lem.Davie_iteration}.

		On the other hand, using Minkowski inequality, \eqref{est.Qf} and H\"older inequality, we have
		\begin{align*}
			\|J_{s,t}\|_{L_p(\Omega|\cff_v)}&\le2\int_s^t\|g(r,\bar X_r)\|_{L_p(\Omega|\cff_v)}dr
			\les\int_s^t (r-v)^{-\frac d{2p}}\|g_r\|_{L_p(\Rd)}dr
			\\&\les (r-v)^{-\frac d{2p}}(t-s)^{1-\frac1q}\|g\|_{\LL^q_p([s,t])}
			\les w(s,t),
		\end{align*}
		verifying condition \eqref{con.sll1}.
		An application of \cref{lem.Davie_iteration} yields that $\|J_{s,t}\|_{L_p(\Omega|\cff_v)}\les \Gamma(1+|\log(\Gamma)|)w(s,t)^{1/2}+\Gamma w(s,t)$ for every $v<s\le u\le t\le1$. From here, the argument follows analogously as in the proof of \cref{prop.B2}, using \cref{lem.Bellising} to remove the singularity near $v$ then using \cref{lem.lenglart} to obtain the desired estimate for the supremum.
	\end{proof}
	\begin{proof}[\bf Proof of \cref{thm.prealpha}]
		Let $\bar X$ be a solution to \eqref{eqn.SDEnob}.
		Similar to the proof of \cref{prop.gf}, we use Girsanov transformation to deduce the moment estimates for $X$ from those obtained in \cref{prop.Xnob}.
		Indeed, define the measure $\bar \mP:=\rho\mP$ where
		\begin{align*}
		    \rho:=\exp\left(-\int_0^1(\sigma^{-1}b)(r,\bar X_r)dB_r-\frac{1}{2}\int_0^1\Big|(\sigma^{-1}b)(r,\bar X_r)\Big|^2dr\right).
		\end{align*}
		From \cite[Lemma 4.1]{XXZZ} and Novikov criterion, we see that $\E \rho^{r}<\infty$ for every $r\in\R$.
		By Girsanov theorem, $\bar\mP$ is a probability measure and the law of $X$ under $\bar\mP$ is the same as the law of $\bar X$ under $\mP$. From here, we deduce \cref{thm.prealpha} from \cref{prop.Xnob}, using similar computations as in the proof of \cref{prop.gf}.
	\end{proof}
	\begin{remark}\label{rmk.weakuniq}
		Observe that pathwise uniqueness is not necessary and only weak uniqueness of \eqref{sde00} is used in the above proof.
		In addition, reasoning as in \cite[Lemmas 3.2, 3.3 and Remark 3.5]{MR2117951}, one can derive weak uniqueness for \eqref{sde00}  from \cref{prop.Xnob}(i). Consequently, \cref{thm.prealpha} holds for any adapted solution to \eqref{sde00}.
	\end{remark}

\section{Proof of the main results}
\label{sec:proof}
We present in the current section the proofs of \cref{thm.main,thm.alpha}.
We state a maximal regularity result for parabolic equations, which is a direct consequence of \cite[Theorem 3.2]{XXZZ}.
\begin{lemma}\label{pde}
	Assume \crefrange{con.A}{con.B}.
	Let $f\in{{\mathbb{L}}}_p^{q}([0,1])$ and $M>0$ be such that
	\begin{align*}
		\|f\|_{\LL^q_p([0,1])}+\|b\|_{\LL^q_p([0,1])}\le M.
	\end{align*}
	Then there exists $\lambda_0=\lambda_0(M,a)\geq 1$ such that that for all $\lambda\ge\lambda_0$, there is a unique solution $u$ in $\mathbb{L}_{2,p}^{q}([0,1])$ to the equation
	\begin{align}\label{pdelambda}
		\partial_tu+\frac{1}{2}a^{ij}\partial^2_{ij} u+b\cdot\nabla u+f=\lambda u, \quad u(1,\cdot)=0.
	\end{align}
	Furthermore, for any $\gamma\in[0,2)$, $p_1\in[p,\infty)$, $q_1\in[q,\infty)$ with $\frac{d}{p}+\frac{2}{q}<2-\gamma+\frac{d}{p_1}+\frac{2}{q_1}$, there is a constant $C=C(M,\gamma,p_1,q_1)>0$ such that for any $\lambda\geq \lambda_0$,
	\begin{align*}
	\lambda^{\frac{1}{2}(2-\gamma+\frac{d}{p_1}+\frac{2}{q_1}-\frac{d}{p}-\frac{2}{q})}\Vert u\Vert_{\mathbb{L}_{\gamma, p_1}^{q_1}([0,1])}+\Vert \partial_tu\Vert_{\mathbb{L}_p^q([0,1])}+\Vert u\Vert_{\mathbb{L}_{2, p}^{q}([0,1])}\leq C\Vert f\Vert_{\mathbb{L}_{p}^{q}([0,1])}.
	\end{align*}
\end{lemma}
\begin{lemma}\label{lem.Xtkt}
	Let $X^n$ be the solution to \eqref{eqn.EMscheme}. Then for every $m\ge2$,
	\begin{align*}
		\sup_{t\in[0,1]}\|X^n_{t}-X^n_{k_n(t)}\|_{L_m(\Omega)}\les (1/n)^{\frac 12}.
	\end{align*}
\end{lemma}
\begin{proof}
	We have
	\begin{align*}
		X^n_t-X^n_{k_n(t)}=\int_{k_n(t)}^t b^n(r,X^n_{k_n(t)})dr+\int_{k_n(t)}^t \sigma(r,X^n_{k_n(t)})dB_r.
	\end{align*}
	Using BDG inequality and H\"older, we have
	\begin{align*}
		\|X^n_t-X^n_{k_n(t)}\|_{L_m(\Omega)}\les (t-k_n(t))^{1-\frac1q}\|b^n\|_{\LL^q_\infty([k_n(t),t])}+(t-k_n(t))^{1/2}.
	\end{align*}
	Using $t-k_n(t)\le 1/n$ and \cref{con.B},
	\begin{align*}
		(t-k_n(t))^{1-\frac1q}\|b^n\|_{\LL^q_\infty([k_n(t),t])}\les (1/n)^{\frac12}.
	\end{align*}
	Combining these estimates, we obtain the result.
\end{proof}

Recall that $U$ is the solution to the equation \eqref{eqn.uKol}.
Under \crefrange{con.A}{con.B}, there exists $\lambda_0>0$ such that for every $\lambda\ge \lambda_0$, $\nabla U$ is bounded H\"older continuous on $[0,1]\times\Rd$,
\begin{align}\label{tmp.2709regu}
	\sup_n\left[\|\p_tU\|_{\mL^q_{ p}([0,1])}+\|U\|_{\mathbb{L}_{2, p}^{q}([0,1])}\right]<\infty
	\tand
	\sup_n\sup_{(t,x)\in[0,1]\times\Rd}|\nabla U(t,x)|=o_\lambda(1),
\end{align}
where $o_\lambda(1)$ denotes any constant such that $\lim_{\lambda\to\infty}o_\lambda(1)=0$.
Indeed, existence and uniqueness and the first estimate in \eqref{tmp.2709regu} follow from \cref{pde}. 
H\"older continuity of $\nabla U$  and the second estimate in \eqref{tmp.2709regu} follow by applying \cite[Lemma 10.2]{MR2117951} and the estimate in \cref{pde}.
Let $\mathcal{M}$ be the Hardy--Littlewood maximal operator defined as  
\begin{align*}
	\cmm f(x):=\sup_{0<r<\infty}\frac{1}{|\cB_r|}\int_{\cB_r}f(x+y)dy, \quad \cB_r:=\{x\in\mathbb{R}^d:|x|<r\},\quad r>0.
\end{align*}
It is well-known that $\cmm$ is bounded on $L_p(\Rd)$. Thus we have
\begin{align}
\label{max}
\Vert\mathcal{M} f\Vert_{\mathbb{L}_p^q([0,1])}\lesssim \Vert f\Vert_{\mathbb{L}_p^q([0,1])}.    
\end{align}
Define
\begin{multline}\label{def.Ant}
	A_t^n:= t+\int_0^t \left[ \cmm|\nabla^2 U|(s,X_s)+\cmm|\nabla^2 U|(s,X^n_s)\right]^2 \dif s
	\\+\int_0^t \left[ \cmm|\nabla \sigma|(s,X_s)+\cmm|\nabla\sigma|(s,X^n_s)\right]^2\dif s.
\end{multline}
\begin{proposition}\label{prop.weightedmoment}
	For every ${\bar p}\in(1,p)$, there exists a finite positive constant $c_{\bar p}$ such that
	\begin{align*}
		\left\|e^{-c_{\bar p}|A^n_1|^{\max({\bar p}/2,1)}}\sup_{t\in[0,1]}|X_t-X^n_t|\right\|_{L_{\bar p}(\Omega)}
		&\les\|x_0-x^n_0\|_{L_{\bar p}(\Omega)}+ \varpi_n({\bar p})
		\\&\quad+(1/n)^{\frac \alpha2}+(1/n)^{\frac12}\log(n).
	\end{align*}
	
\end{proposition}
\begin{proof}
	Applying It\^o's formula (\cite[Lemma 4.1]{XXZZ}) for $U(t, X_t)$, we obtain that
	\begin{multline}\label{eq-bX}
	\int_0^t b^n(r,X_r) \dif r = U(0,X_0)-U(t, X_t) + \lambda \int_0^t U(r, X_r) \dif r
	\\+\int_0^t\nabla U(r,X_r)[b(r,X_r)-b^n(r,X_r)]dr  + \int_0^t (\nabla U \cdot\sigma)(r, X_r) \dif B_r,
	\end{multline}
	and similarly,
	\begin{align}\label{eq-bXn}
	\int_0^t b^n(r,X^n_r) \dif r &= U(0,X^n_0)-U(t, X^n_t) + \lambda \int_0^t U(r, X^n_r) \dif r
	\nonumber\\&\quad+\int_0^t \nabla U(r, X_r^n) [b^n(r, X_{k_n(r)}^n)-b^n(r,X^n_r)]  \dif r
	\nonumber\\&\quad+\int_0^t\nabla^2 U(r,X_r)[a(r,X^n_{k_n(r)})-a(r,X^n_r)]dr
	\nonumber\\&\quad+\int_0^t\nabla U(r,X^n_r)\sigma(r,X^n_{k_n(r)})dB_r.
	\end{align}
	From equations \eqref{sde00} and \eqref{eqn.EMscheme}, we have
	\begin{align*}
		X_t-X^n_t
		&=x_0-x_0^n
		+\int_0^t[b^n(r,X_r)-b^n(r,X^n_r)]dr
		\\&\quad+\int_0^t[b(r,X_r)-b^n(r,X_r)]dr+\int_0^t[b^n(r,X^n_r)-b^n(r,X^n_{k_n(r)})]dr
		\\&\quad+\int_0^t[\sigma(r,X_r)-\sigma(r,X^n_r)]dB_r
		+\int_0^t[\sigma(r,X^n_r)-\sigma(r,X^n_{k_n(r)})]dB_r.
	\end{align*}
	We plug \eqref{eq-bX} and \eqref{eq-bXn} into the previous identity, raise to ${\bar p}$-th power to find that
	\begin{align*}
		\xi_t:=\sup_{s\in[0,t]}|X_s-X^n_s|^{\bar p}
		\les |x_0-x_0^n|^{\bar p}+V^0_t+\sum_{i=1}^3V^i_t+\sum_{i=1}^3I^i_t,
	\end{align*}
	where
	\begin{align*}
		&\begin{aligned}
			V^0_t&=|U(0,x_0)-U(0,x_0^n)|^{\bar p}
		+ \sup_{s\in[0,t]}|U(s,X_s)-U(s,X^n_s)|^{\bar p}
		\\&\quad+\lambda^{\bar p}\Big|\int_0^t|U(r,X_r)-U(r,X^n_r)|dr\Big|^{\bar p}
		\end{aligned}
		,\\&V^1_t=\sup_{s\in[0,t]}\Big|\int_0^s[I+\nabla U(r,X_r)][b(r,X_r)-b^n(r,X_r)]dr\Big|^{\bar p}
		,\\&V^2_t=\sup_{s\in[0,t]}\Big|\int_0^s[I+\nabla U(r,X^n_r)][b^n(r,X^n_r)-b^n(r,X^n_{k_n(r)})]dr\Big|^{\bar p}
		,\\&V^3_t=\sup_{s\in[0,t]}\Big|\int_0^s\nabla^2 U(r,X_r)[a(r,X^n_{k_n(r)})-a(r,X^n_r)]dr \Big|^{\bar p}
		,\\&I^1_t=\sup_{s\in[0,t]}\Big|\int_0^s[I+\nabla U(r,X^n_r)][\sigma(r,X_r)-\sigma(r,X^n_r)]dB_r\Big|^{\bar p}
		,\\&I^2_t=\sup_{s\in[0,t]}\Big|\int_0^s[I+\nabla U(r,X^n_r)][\sigma(r,X^n_r)-\sigma(r,X^n_{k_n(r)})]dB_r\Big|^{\bar p}
		,\\&I^3_t=\sup_{s\in[0,t]}\Big|\int_0^s[\nabla U(r,X_r)-\nabla U(r,X^n_r)]\cdot \sigma(r,X_r)dB_r \Big|^{\bar p}.
	\end{align*}
	Using \eqref{tmp.2709regu} and Cauchy--Schwarz inequality
	\begin{align*}
		V^0_t\les|x_0-x_0^n|^{\bar p}+o_\lambda(1)\sup_{s\in[0,t]}|X_s-X_s^n|^{\bar p}+\left(\int_0^t \xi_r^{2/{\bar p}} dr\right)^{\frac {\bar p}2}.
	\end{align*}

	To estimate $I^1,I^2,I^3$, we will utilize a special case of the pathwise Burkholder--Davis--Gundy (BDG) inequality of \cite[Theorem 5]{Pietro}. Namely, there exists a constant $C=C({\bar p},d)$ such that for any c\'adl\'ag martingale $\bar M$, there exists a local martingale $\widetilde M$ such that with probability one,
	\begin{align*}
		\sup_{s\in[0,t]}|\bar M_s|^{\bar p}\le C[\bar M]_t^{\frac{\bar p}{2}}+\widetilde M_t, \quad \forall t.
	\end{align*}
	In the above, $[\bar M]$ is the quadratic variation of $\bar M$.
	We now estimate $I^1$.
	By property of maximal function (see \cite[Proposition C.1]{hinz2020variability}) and continuity of $\sigma$, we have for every $r\in[0,1]$ and every $x,y\in\mR^d$
	\begin{align*}
		|\sigma(r,x)- \sigma(r,y)|\les|x-y|(\cmm|\nabla\sigma(r,x)|+\cmm|\nabla\sigma(r,y)|).
	\end{align*}
	Using this, \eqref{tmp.2709regu} and the pathwise BDG inequality, we can find a local martingale $M^1$ such that
	\begin{align*}
		I^1_t&\les \left(\int_0^t|I+\nabla U(r,X^n_r)|^2|X_r-X^n_r|^2\left(\cmm|\nabla \sigma|(r,X_r)+\cmm|\nabla \sigma|(r,X^n_r)\right)^2dr\right)^{\frac {\bar p}2}+M^1_t
		\\&\les\left(\int_0^t \xi_r^{2/{\bar p}}dA^n_r \right)^{\frac {\bar p}2}+M^1_t.
	\end{align*}
	Similarly,
	\begin{align*}
		I^3_t\les \left(\int_0^t \xi_r^{2/{\bar p}}dA^n_r \right)^{\frac {\bar p}2}+M^3_t
	\end{align*}
	for some local martingale $M^3$.
	Using \cref{lem.Xtkt}, \eqref{tmp.2709regu}, H\"older continuity of $\sigma$ and the pathwise BDG inequality, we have
	\begin{align*}
	 	I^2_t\les\left(\int_0^t|X^n_r-X^n_{k_n(r)}|^{2  \alpha}dr\right)^{\frac {\bar p}2}+M^2_t\les (1/n)^{{\bar p} \frac \alpha2}+M^2_t.
	\end{align*}
	It follows that
	\begin{align*}
		\xi_t\les o_\lambda(1)\xi_t+ \left(\int_0^t \xi^{2/{\bar p}}dA^n\right)^{\frac {\bar p}2}+V_t+M_t
	\end{align*}
	where $V=|x_0-x^n_0|^{\bar p}+V^1+V^2+V^3$ and $M=M^1+M^2+M^3$.
	By choosing $\lambda$ sufficiently large, this deduces to
	\begin{align*}
	 	\xi_t\les  \left(\int_0^t \xi^{2/{\bar p}}dA^n\right)^{\frac {\bar p}2}+V_t+M_t
	\end{align*}
	Applying stochastic Gr\"onwall lemma, \cref{lem.SGr\"onwall}, we have
	\begin{align}\label{tmp.2709eA}
		\E e^{-c_{\bar p} |A^n_1|^{\max({\bar p}/2,1)}} \xi_1\le\E V_1
	\end{align}
	for some finite positive constant $c_{\bar p}$.
	In view of \cref{def.rate}, it is evident that
	\begin{align*}
		\E V^1_1\les \varpi_n({\bar p})^{\bar p}.
	\end{align*}
	Using \cref{prop.gf}, \eqref{tmp.2709regu} and \cref{con.B}, we have
	\begin{align*}
		\E V^2_1
		\les \left[(1/n)^{1-\frac1q}\beta_n(b^n)+(1/n)^{\frac \alpha2}+(1/n)^{\frac12}\log(n)\right]^{\bar p}
		\les \left[(1/n)^{\frac \alpha2}+(1/n)^{\frac12}\log(n)\right]^{\bar p}.
	\end{align*}
	Using \cref{con.A} and Cauchy--Schwarz inequality, we have
	\begin{align*}
		\E V^3_1&\les
		\E\Big|\int_0^1|\nabla^2 U(r,X_r)||X^n_r-X^n_{k_n(r)}|^\alpha dr \Big|^{\bar p}
		\\&\le \E\left(\int_0^1|\nabla^2 U(r,X_r)|^2dr\right)^{\frac {\bar p}2}\left(\int_0^1|X^n_r-X^n_{k_n(r)}|^{2 \alpha} dr \right)^{\frac {\bar p}2}
		\\&\le \left[\E\left(\int_0^1|\nabla^2 U(r,X_r)|^2dr\right)^{{\bar p}}\right]^{1/2} \left[\E\left(\int_0^1|X^n_r-X^n_{k_n(r)}|^{2 \alpha} dr \right)^{{\bar p}}\right]^{1/2}.
	\end{align*}
	In view of  \eqref{tmp.2709regu}, \cref{thm.prealpha}(i) and \cref{lem.Xtkt}
	\begin{align*}
		\E V^3_1\les (1/n)^{ {\bar p}\frac \alpha2}.
	\end{align*}
	The previous estimates for $\E V^i_1$'s and \eqref{tmp.2709eA} yield the result.
\end{proof}
\begin{lemma}\label{lem.expA}
	Let $\rho\in(0,\frac{p\wedge p_0} d)$ and $\kappa>0$ be some fixed constants. Then $\sup_n\E e^{\kappa |A^n_1|^\rho}$ is finite.
\end{lemma}
\begin{proof}
	We observe that
	\begin{align*}
		\E_s \delta A^n_{s,t}\les (t-s)+ (t-s)^{1-\frac dp-\frac2q}\|\cmm|\nabla^2 u|\|_{\LL^q_p([s,t])}^2+(t-s)^{1-\frac d{p_0}-\frac 2{q_0}}\|\cmm|\nabla \sigma|\|_{\LL^{q_0}_{p_0}([s,t])}^2.
	\end{align*}
	Indeed, the estimates for functionals of $X^n$ follow from \cref{prop.momentg} and Girsanov theorem.
	The estimates for the functionals of $X$ follows from \eqref{tmp.2809Esh}, \cref{lem.Khasminski} and Girsanov theorem; or alternatively can be derived from those of $X^n$ and the weak convergence of $X^n$ to $X$. 
	In view of \cref{rmk.khas12}, this implies that for any $\lambda\geq0$
	\begin{align*}
		\E e^{\lambda A_1^n}\les e^{c\lambda^a}, \quad \frac1a=1-\frac d{p\wedge p_0},
	\end{align*}
	where $c$ is some universal positive constant.
	For simplicity, we write $A$ for $A^n_1$ below. For every $x>0$, using Chebyshev inequality, we have
	\begin{align*}
		\mP(A>x)=\mP(e^{\lambda A}>e^{\lambda x})\le e^{-\lambda x}\E e^{\lambda A}
		\les e^{-\lambda x+c \lambda^a}.
	\end{align*}
	One can optimize in $\lambda$ to obtain that for every $x$ bounded away from $0$,
	\begin{align*}
		\mP(A>x)\les e^{-c x^{a'}}, \quad\frac1a+\frac1{a'}=1,
	\end{align*}
	where $c$ is another positive constant. In view of layer cake representation
	\begin{align*}
		\E e^{\kappa A^\rho}=\kappa \rho\int_0^\infty e^{\kappa x^\rho}x^{\rho-1}\mP(A>x)dx,
	\end{align*}
	we see that $\E e^{\kappa A^\rho}$ is finite if $\rho<a'$, completing the proof because $a'=(p\wedge p_0)/d$.
\end{proof}

\begin{proof}[\bf Proof of \cref{thm.main}]
	For $\bar p\in(0,2\frac{p\wedge p_0}d)$, we obtain from \cref{lem.expA} that the quantity $\sup_n\E e^{\kappa |A^n_1|^{\max(\gamma/2,1)}}$ is finite for any constant positive $\kappa$. From here, we obtain \eqref{maine} from \cref{prop.weightedmoment} and H\"older inequality.
\end{proof}
\begin{proof}[\bf Proof of \cref{thm.alpha}]
	We put $g=\nabla U$. From \cref{pde}, we have
	\begin{align*}
		\sup_n (\|g\|_{\LL^{q_2}_{\nu,p_2}([0,1])}+ \|g\|_{\LL^\infty_\infty([0,1])}+\|g\|_{\LL^q_{1,p}([0,1])})<\infty
	\end{align*}
	for all $p_2\in[p,\infty),q_2\in[q,\infty)$ and $\nu\in[0,1)$ with 
	\begin{align}\label{con.p1}
		\frac dp+\frac2q+\nu-1<\frac d{p_2}+\frac2{q_2}.
	\end{align}

	\textit{Part (i).}
	Let $p_1,q_1$ be as in \cref{thm.alpha}(i).
	From \cref{thm.prealpha}(i), we have
	\begin{align*}
		\varpi_n(m)\les \|(1+g)(b-b^n)\|_{\LL^{q_1}_{p_1}}\les (1+\|g\|_{\LL^\infty_\infty})\|b-b^n\|_{\LL^{q_1}_{p_1}}
	\end{align*}
	which shows \eqref{est.ap}.
	
	\textit{Part (ii).}
	Define $q_3$ by $\frac1{q_3}=\frac1{q_2}+\frac1q$.
	For each $\nu\in[0,1]$ and $p_3\in(1,\infty)$  satisfying
	\begin{align}\label{con.p3a}
		\frac1{p}\le\frac1{p_3}\le\frac1p+\frac1{p_2}<\frac1{p_3}+\frac \nu d,
	\end{align}
	an application of \cref{lem.multiplication}(ii) and H\"older inequality shows that the pointwise multiplication is a continuous bilinear map
	\begin{align*}
		\LL^{q}_{-\nu,p}([0,1])\times\LL^{q_2}_{\nu,p_2}([0,1])\to \LL^{q_3}_{-\nu,p_3}([0,1]).
	\end{align*}
	If $p_3$ can be chosen such that
	\begin{align}\label{con,p3b}
		\frac d{p_3}+\frac2{q_2}+\frac2q<2- \nu, 	
	\end{align}
	then \cref{thm.prealpha}(ii) can be applied, which gives for every $\bar p\in(0,p_3)$,
	\begin{align*}
		\|\sup_{t\in[0,t]}|\int_0^tg(b-b^n)(r,X_r)dr|\|_{L_{\bar p}(\Omega)}\lesssim \|g(b-b^n)\|_{\LL^{q_3}_{-\nu,p_3}}\les \|b-b^n\|_{\LL^{q}_{-\nu,p}},
	\end{align*}
	To obtain the last step, we apply the multiplication result above to see that
	\begin{align*}
		\|g(b-b^n)\|_{\LL^{q_3}_{-\nu,p_3}}\lesssim \|g\|_{\LL^{q_2}_{\nu,p_2}} \|b-b^n\|_{\LL^{q}_{-\nu,p}}\lesssim \|b-b^n\|_{\LL^{q}_{-\nu,p}}.
	\end{align*}
	On the other hand, applying \cref{thm.prealpha}(ii), we have
	\begin{align*}
		\|\sup_{t\in[0,1]}|\int_0^t(b-b^n)(r,X_r)dr|\|_{L_{\bar p}(\Omega)}\les \|b-b^n\|_{\LL^q_{-\nu,p}}.
	\end{align*}
	These estimates verify \eqref{est.anu}.
	
	Next, we verify that it is possible to choose $p_2,q_2,p_3$ satisfying all the above conditions.
	Given $p_2,q_2,p,q,\nu$, there exists $p_3\in(1,\infty)$ satisfying \eqref{con.p3a} and \eqref{con,p3b} iff 
	\begin{equation*}
		\left\{
		\begin{aligned}
			&\frac dp<2- \nu-\frac2{q_2}-\frac2q
			\\& \frac dp+\frac d{p_2}- \nu<2- \nu-\frac2{q_2}-\frac2q
		\end{aligned}
		\right.
	\end{equation*}
	which is deduced to 
	\begin{equation}\label{con.p2b}
		\left\{
		\begin{aligned}
			&\frac2{q_2}<2- \nu-\frac dp-\frac2q
			\\& \frac d{p_2}+\frac2{q_2}<2-\frac dp-\frac2q.
		\end{aligned}
		\right.
	\end{equation}
	Given $p,q,\nu$, the existence $p_2\in[p,\infty),q_2\in[q,\infty)$ satisfying \eqref{con.p1} and \eqref{con.p2b} is equivalent to the problem (P): given real numbers $a>0,b>0$ and $c_1<c_2$, find $x\in(0,a), y\in(0,b) $ such that $c_1<x+y<c_2$. Here, we have put $x=d/p_2$, $y=2/q_2$, $a=d/p$, $b=\min(2/q,2- \nu- \zeta)$, $c_1=\zeta+\nu-1$, $c_2=2-\zeta$ and $\zeta=d/p+2/q$.
	With some plotting aid, it is seen that this problem has a solution $(x,y)$ iff $c_1<a+b$ and $c_2>0$. This deduces to the condition \eqref{con.nu}.

	Since \eqref{est.anu} is valid for all $\bar p\in(0,p_3)$, it remains to identify the largest possible value for $p_3$, denoted by $p_3^*$. From \eqref{con.p3a} and \eqref{con,p3b}, we see that
	\begin{align*}
		\frac 1{p_3^*}=\max\left(\frac1p,\frac1p+\frac1{p_2}-\frac \nu d\right).
	\end{align*}
	We observe that the problem (P) with the additional constraint $d/p_2\le \nu$ has a solution $(x,y)$ iff \eqref{con.nu} holds. In other words, we can choose $p_2$ so that $d/p_2\le \nu$ and hence $p_3^*=p$. This shows that \eqref{est.anu} holds for every $\bar p\in(0,p)$.

	\textit{Part (iii).}
	When $\nu=1$, we ought to take $p_2=p$, $q_2=2$. We can also choose $p_3=p$ and $q_3=q/2$. Condition \eqref{con.p3a} is trivially satisfied while condition \eqref{con,p3b} is verified by  \eqref{con.nu1}.
	Define the control $w_1$ by
	\begin{align*}
		w_1(s,t)=\left(\|g\|_{\LL^q_{-1,p}([s,t])}^{q/2}+\|g\|_{\LL^q_{\infty}([s,t])}^{q/2}\right)w_0(s,t)^{1/2}
	\end{align*}
	Then by the multiplication result above, H\"older inequality and \eqref{con.bbn-1}, we have
	\begin{gather*}
		\|g(b-b^n)\|_{\LL^{q/2}_{-1,p}([s,t])}\les \|g\|_{\LL^q_{-1,p}([s,t])}\|b-b^n\|_{\LL^q_{-1,p}([s,t])}
		\les \Gamma w_1(s,t)^{\frac2q},
		\\\|g(b-b^n)\|_{\LL^{q/2}_{p}([s,t])}\les \|g\|_{\LL^q_{\infty}([s,t])}\|b-b^n\|_{\LL^q_{p}([s,t])}\les w_1(s,t)^{\frac 2q}.
	\end{gather*}
	We then apply \cref{thm.prealpha}(iii) to get 
	\begin{align*}
		\|\sup_{t\in[0,t]}|\int_0^tg(b-b^n)(r,X_r)dr|\|_{L_{\bar p}(\Omega)}
		&\les \Gamma(1+|\log \Gamma|)w_1(0,1)^{\frac 2q}
		\\&\les \Gamma(1+|\log \Gamma|)w_0(0,1)^{\frac 1q}.
	\end{align*}
	On the other hand, applying \cref{thm.prealpha}(iii) under condition \eqref{con.bbn-1} yields
	\begin{align*}
		\|\sup_{t\in[0,1]}|\int_0^t (b-b^n)(r,X_r)dr|\|_{L_{\bar p}(\Omega)}\les \Gamma(1+|\log \Gamma|)w_0(s,t)^{\frac1q}.
	\end{align*}
	Combining the previous two estimates, we obtain \eqref{est.a-1}.
\end{proof}

\section{Application: stochastic transport equations}
\label{sec.application}

Let $(W_t)$ be a standard $d$-dimensional Brownian motion on  a filtered probability space $(\Omega, \mathcal{F}, (\mathcal{F}_t)_{t\in[0,1]},\mP)$ and let $b:[0,1]\times\Rd\to\Rd$ be a Borel measurable function satisfying \eqref{pq}. In this section, we propose a numerical scheme for the following (forward) stochastic linear transport equation
\begin{align}
    \label{STE}
    \partial_tu+b\cdot \nabla u+\nabla u\circ d W_t=0,
\quad u(0,x)=\rho(x),
\end{align}
where   $\rho\in \cap_{r\geq1}L_{1,r}(\mR^d)$ and $\nabla u\circ d W_t$  is interpreted in Stratonovich sense. As in \cite{FF13}, we say $u$ is a {\em weakly differentiable solution} to \eqref{STE} if 
\begin{itemize}
    \item $u:\Omega\times[0,1]\times\mR^d\rightarrow\mR$ is measurable, $\int_{\mR^d}u(t,x)\psi(x)dx$ is progressively measurable for each $\psi\in C_0^\infty(\mR^d)$;
    \item $\mP(u(t,\cdot)\in\cap_{r\geq1} L_{1,r}^{loc}(\mR^d))=1$ for $t\in[0,1]$ and both $u$ and $\nabla u$ are in $C([0,1];\cap_{r\geq1}L^r(\mR^d\times\Omega))$; 
    \item for any $\psi\in C_0^\infty(\mR^d)$ and $t\in[0,1]$ with probability one the following holds
    \begin{align*}
       \int_{\mR^d}u(t,x)\psi(x)dx+& \int_0^t\int_{\mR^d}b(s,x)\cdot \nabla u(s,x)\psi(x)dxds\\ &=\int_{\mR^d}\rho(x)\psi(x)dx+\sum_{i=1}^d\int_0^t\Big(\int_{\mR^d}u(s,x)\partial_{x_i}\psi(x)dx\Big)d W_s^i
      \\&\quad\quad\quad\quad\quad\quad\quad\quad\quad+\frac{1}{2}\int_0^t\int_{\mR^d}u(s,x)\Delta\psi(x)dxds.
    \end{align*}
\end{itemize}
It is known from \cite[Theorems 10, 11]{FF13} that a weakly differentiable solution $u$ to \eqref{STE} exists uniquely and has the representation  $u(\tau,x)=\rho(\phi_0^\tau(x))$,   where $\phi_0^\tau(x)$ is the inverse of the stochastic flow of homeomorphisms generated by the solution $(X_\tau(x))_{\tau\in[0,1]}$ to the SDE 
\begin{align}\label{eq:SDE-ad}
    dX_\tau(x)=b(\tau,X_\tau(x))d\tau+d W_\tau,\quad \tau\in[0,1],\quad X_0(x)=x.
\end{align}
For each fixed $\tau\in[0,1]$, consider the  backward-in-time SDE
\begin{align}
     \label{backward-sde}
   	X_{\tau,s}(x)&=x-\int_s^\tau b(r,X_{\tau,r}(x))dr+W_s-W_\tau,\quad 0\leq s\leq \tau, \quad X_{\tau,\tau}(x)=x.
\end{align}
The inverse flow $\phi^\tau_0(x)$ is directly related to the solution to the previous SDE through the relation $\phi_0^{\tau}(x)=X_{\tau,0}(x)$. Indeed, when \(b\) is a smooth bounded function with bounded derivatives, this relation is classical, see \cite[Theorem 3.7.1]{MR3929750}. When $b$ belongs to $\LL_p^q$, one can approximate it by smooth functions. Because both $\phi_0^{\tau}(x), X_{\tau,0}(x)$ are stable under this approximation (see \cite[Lemma 3]{FF13} and \cite[Theorem 1.2]{MR4546633} respectively), the relation holds true in this case. Alternatively, another argument for this fact is recently provided in \cite{ALL} utilizing \emph{path-by-path uniqueness} of \eqref{eq:SDE-ad}.

To devise a numerical scheme for \eqref{STE}, it convenes to introduce
$ X_{s}^\tau(x):=X_{\tau,\tau-s}(x)$. By a change of variables, we find that
\begin{align}\label{change-variable}
    X_s^\tau(x)=x-\int_0^s b^\tau(r, X^\tau_r(x))dr+ W^\tau_s, \quad 0\leq s\leq \tau,
\end{align}
where $b^\tau(r,x):=b(\tau-r,x)$ and $W^\tau_r:=W_\tau-W_{\tau-r}$ for $r\in[0,\tau]$. Observe that $(W_t^\tau)_{t\in[0,\tau]}$ is a $(\mathcal{F}_t^\tau)$-Brownian motion with  $\mathcal{F}_t^\tau:=\sigma(W_{\tau-r_1}-W_{\tau-r_2},0\leq r_1\leq r_2\leq t)$ for $t\in[0,\tau]$.
Hence, we have the representation $u(\tau,x)=\rho(X^\tau_\tau(x))$. This naturally suggests the numerical scheme 
\begin{align*}
	u^n(\tau,x)=\rho(X^{\tau,n}_\tau(x)), 
\end{align*}
  where for each $\tau\in(0,1]$, $ (X_s^{\tau,n})_{s\in[0,\tau]}$ is the tamed Euler--Maruyama approximation for \eqref{change-variable}, namely
\begin{align}
    \label{BE-EM}
  X_{s }^{\tau, n}(x)&=x-\int_0^s b^{\tau, n}(r,X_{k_n(r)}^{\tau,n}(x))dr+ W^\tau_s,\quad 0\leq s\leq \tau.
\end{align}
Here, $b^{\tau, n}(r,x):=b^n(\tau-r,x)$ for $r\in[0,\tau]$ and $b^n$ is an approximation for $b$ satisfying  \cref{con.B}.

\begin{theorem}\label{thm.STE}
 Suppose that \cref{con.B} holds.
 Let $\nu\in[0,1)$ satisfy \eqref{con.nu}  and $p_1,q_1\in[1,\infty]$ satisfy $\frac d{p_1}+\frac 2{q_1}<2$. Then  for any  $l \in(1,p\wedge\frac{2p}{d})$, any $\bar r\in(1,\infty)$ satisfying  $\frac1{\bar r}<\frac1l-\frac1p(1\vee\frac d2)$, 
 we have
   \begin{align}
       \label{TE-app}
     \sup_{(\tau,x)\in[0,1]\times\mR^d} &\tau^{\frac{d}{2\bar r}}\Vert  u^n(\tau,x)-u(\tau,x)\Vert_{L_{l}(\Omega)}\nonumber\\&\leq N\Vert \nabla \rho\Vert_{L_{\bar r}(\mR^d)}((1/n)^{\frac12}\log n+\min(\Vert b-b^n\Vert_{\LL_{p_1}^{q_1}([0,1])},\Vert b-b^n\Vert_{\LL_{-\nu,p}^q([0,1])})),
   \end{align}
   where  $N$ depends on $K_4,p,d,l$, $p_1$, $q_1$ and $\bar r$.
\end{theorem}
\begin{proof}
 We put $b(r,x)=0$ and $b^n(r,x)=0$ whenever $r\in\R\setminus[0,1]$ so that $b^\tau(r,x)$ and $b^{\tau,n}(r,x)$ are well-defined functions on $[0,1]\times\Rd$. Let $(\hat W_t)_{t\in[0,1]}$ be a standard $(\hat\cff_t)$-Brownian motion which is independent from $(W_t)_{t\in[0,1]}$ and define
 \begin{align}
  \label{btau}
  W_r^\tau:=\begin{cases}
 W_\tau-W_{\tau-r}& \text{if } r\in[0,\tau],\\
 W_\tau+\hat W_r-\hat W_\tau&  \text{if } r\in(\tau,1],
 \end{cases}
 \end{align}
 which is a standard Brownian motion with respect to the filtration $(\mathcal{G}^\tau_t)_{t\in[0,1]}:=(\mathcal{F}_{t\wedge \tau}^\tau\vee \hat {\mathcal{F}}_{t\vee\tau})_{t\in[0,1]}$.
 Equations \eqref{change-variable} and \eqref{BE-EM} (for each fixed $\tau\in[0,1]$) are extended uniquely over the whole time interval $[0,1]$.

By property of maximal function  and continuity of $\rho$ (see \cite[Proposition C.1]{hinz2020variability}) we have for every $x,y\in\mathbb{R}^d$,
\[
	|\rho(x)-\rho(y)|\lesssim|x-y|(\cmm|\nabla \rho|(x)+\cmm|\nabla \rho|(y)).
\]
It follows that
\begin{align}\label{tmp.1453}
   | u(\tau,x)-u_n(\tau,x)|\lesssim|X_{\tau}^\tau(x)-X_{\tau}^{\tau,n}(x)|(\mathcal{M}|\nabla \rho|( X_{\tau}^\tau(x))+\mathcal{M}|\nabla \rho|(X_{\tau}^{\tau,n}(x))).
\end{align}
Next, we estimate the terms on the right-hand side of the previous inequality.

 In order to apply \cref{thm.main} to obtain estimates for $X^{\tau(x)}-X^{\tau,n}(x)$, we verify that $ b^\tau$ and $b^{\tau,n}$ fulfill \cref{con.B} for each $\tau$.
 Indeed,  it is evident that $b^\tau\in \LL_p^q([0,1])$, $b^{\tau,n}\in\LL^q_p([0,1])\cap\LL^q_\infty([0,1])$ and that $\|b^{\tau,n}\|_{\LL^q_p([0,1])}\le \|b^n\|_{\LL^q_p([0,1])}$ which is bounded uniformly in $n$. In addition, define $\mu^{\tau,n}(s,t)=\mu^n((\tau-t)\vee0,(\tau-s)\vee0)$, which is a continuous control on the simplex $\Delta([0,1])$. We have $\|b^{\tau,n}\|_{\LL^q_\infty([s,t])}=\|b^n\|_{\LL^q_\infty([(\tau-t)\vee0,(\tau-s)\vee0])}$ and hence by \eqref{con.locGK}
 \begin{gather*}
 	(1/n)^{\frac12-\frac1q} \|b^{\tau,n}\|_{\LL^q_\infty([s,t])}\le \mu^{\tau,n}(s,t)^\theta.
 \end{gather*}
 Similarly, $\|b^{\tau,n}\|_{\LL^q_p([0,1])}=\|b^n\|_{\LL^q_p([0,\tau])}\le \|b^n\|_{\LL^q_p([0,1])}$ and $\mu^{\tau,n}(0,1)=\mu^n(0,\tau)\le \mu^n(0,1)$ so that
 \begin{align*}
	\sup_{n\ge1}\left(\Vert b^{\tau,n}\Vert_{\LL_p^q([0,1])}+\mu^{\tau,n}(0,1)\right)\leq K_4,
 \end{align*}
 where $K_4$ is the constant in \cref{con.B}.
 
Hence \cref{thm.alpha}(i-ii) yields that for $\bar p\in[1,p)$
\begin{align*}
    \varpi_n(\bar p)
    \leq N\min(\Vert b-b^n\Vert_{\LL_{p_1}^{q_1}([0,1])},\Vert b-b^n\Vert_{\LL_{-\nu,p}^q([0,1])}).
\end{align*}
\cref{thm.main} yields for any $\gamma\in(0,1)$ and $\bar p\in[1,p\wedge\frac{2p}{d})$,
\begin{align}\label{X-Xn}
	\sup_{(t,x)\in[0,1]\times\mR^d}&\|X^\tau_t(x)-X^{\tau,n}_t(x)\|_{L_{\gamma\bar p(\Omega)}}
	\nonumber\\&\le N\left(n^{-1/2}\log n+\min(\Vert b-b^n\Vert_{\LL_{p_1}^{q_1}([0,1])},\Vert b-b^n\Vert_{\LL_{-\nu,p}^q([0,1])})\right). 
\end{align}
The constant $N$ depends on $K_4,\nu,p,q,p_1,q_1,\bar p$.

Similar to the arguments used in the proofs of \cref{thm.prealpha,prop.gf}, using Girsanov theorem, one can deduce the estimates for $X^{\tau,n},X^\tau$ to estimates for Brownian motion, which are obtained in \cref{lem.heat1}. 
Hence,  for any $\bar r>r>1$, we have that
\begin{align*}
	 \|\mathcal{M}|\nabla \rho|(X_\tau^\tau(x))\|_{L_r(\Omega)}\les \|\mathcal{M}|\nabla \rho|(B_\tau)\|_{L_{\bar r}(\Omega)}\les \tau^{-\frac{d}{2\bar r}}\Vert \mathcal{M}|\nabla \rho|\Vert_{L_{\bar r}(\mathbb{R}^d)}
    \lesssim \tau^{-\frac{d}{2\bar r}}\Vert \nabla \rho\Vert_{L_{\bar r}(\mathbb{R}^d)},
\end{align*}
and similarly,
\begin{align*}
   \Vert \mathcal{M}|\nabla \rho|(X_{\tau}^{\tau,n}(x))\Vert_{L_r(\Omega)}
   \lesssim \tau^{-\frac{d}{2\bar r}}\Vert \nabla \rho\Vert_{L_{\bar r}(\mathbb{R}^d)}.
\end{align*}

Given $l,\bar r$ as in the statement, we can choose  $r\in(1,\bar r)$,  $\bar p\in[1,p)$, $\gamma\in(0,1)$  such that  $\frac{1}{r}+\frac{1}{\gamma\bar p}=\frac{1}{l}$. From \eqref{tmp.1453}, applying H\"older inequality, we have 
\begin{multline*}
    \Vert u_n(\tau,x)-u(\tau,x)\Vert_{L_{l}(\Omega)}
    \\\les \|X_{\tau}^\tau(x)-X_{\tau}^{\tau,n}(x)\|_{L_{\gamma\bar p}(\Omega)}(\|\mathcal{M}|\nabla \rho|( X_{\tau}^\tau(x))\|_{L_r(\Omega)} +\|\mathcal{M}|\nabla \rho|(X_{\tau}^{\tau,n}(x))\|_{L_r(\Omega)} ).
\end{multline*}
Combining with the estimates obtained previously, we obtain that
\begin{multline*}
	 \Vert u_n(\tau,x)-u(\tau,x)\Vert_{L_{l}(\Omega)}
	 \\\lesssim \tau^{-\frac{d}{2\bar r}}\Vert \nabla \rho\Vert_{L_{\bar r}(\mR^d)} \left(n^{-1/2}\log n+\min(\Vert b-b^n\Vert_{\LL_{p_1}^{q_1}([0,1])},\Vert b-b^n\Vert_{\LL_{-\nu,p}^q([0,1])})\right),
\end{multline*}
which implies \eqref{TE-app}.
\end{proof}

\begin{appendix}
\label{sec:parabolic_equations_with_variable_coefficients}
\section*{Parabolic equations with distributional forcing} 
In this section we show the well-posedness and regularity estimates for the solutions for a class parabolic equations with distributional forcing, which are used in \cref{sec:regularizing_properties_of_the_continuum_paths} and have their own interests. Although such equations have been considered extensively in literatures, for instance in \cite{Kry}, \cite{Kim}, \cite{Zhangzhao} and \cite{XXZZ}, the available results therein are valid under different hypotheses and are not applicable to our situations. 


	For each $r\in[1,\infty]$, we denote its H\"older conjugate by  $r'$, i.e. $\frac1r+\frac1{r'}=1$.
	For each Banach space $\mathcal E$, we denote its topological dual by $\mathcal E^*$, and the dual paring between $\mathcal E$ and $\mathcal E^*$ by $\wei{\cdot,\cdot}_{\mathcal E^\ast,\mathcal E}$.
	We consider the parabolic partial differential equations (PDEs)
	\begin{align}\label{eqn.u}
		&(\partial_s+a^{ij}\partial^2_{ij})u=f, \quad u(1,\cdot)=0
		\\\shortintertext{and}
		\label{eqn.v}&\partial_tv- \partial^2_{ij}(a^{ij}v)+g=0, \quad v(0,\cdot)=0
	\end{align}
	under the following assumptions:\footnote{Parabolic PDEs with distributional forcing have been considered by Kim \cite{Kim}. However, his result is applicable to \eqref{eqn.u} and \eqref{eqn.v} only when $a$ is continuously differentiable in the spatial variables with bounded derivatives and $p\leq q$  hence \cref{con.A1} is excluded.  }
	
	\begin{customcon}{$\mathfrak{A}^\prime$}\label{con.A1} \phantom{em}
  		\begin{enumerate}[label=$\mathbf{\arabic*}$.]
  			\item $a$ is a $d\times d$-symmetric matrix-valued measurable function on $[0,1]\times\Rd$. There exists a constant $k_1\in[1,\infty)$ such that for every $s\in[0,1]$ and $x\in\Rd$
		  	\begin{align}\label{ellptic-cona}
		  		k_1^{-1}I\le a(s,x)\le k_1I.
		  	\end{align}
			Furthermore, $a(s,\cdot)$ is weakly differentiable for a.e. $s\in[0,1]$ and $k_3:=\|\nabla a\|_{\LL_{p_0}^\infty([0,1])}$ is finite for some $p_0\in(d,\infty)$.
			\item $f\in\LL^q_{-1,p}([0,1])$ and $g\in\LL^{q'}_{-1,p'}([0,1])$ for some $p,q\in(1,\infty)$ satisfying $\frac1p+\frac1{p_0}< 1$. 
  		\end{enumerate}
	\end{customcon}
	\begin{definition}\label{def-sol}
		A measurable function $u:[0,1]\times\Rd\to\R$ is a solution to \eqref{eqn.u} if $u\in\LL^q_{1,p}([0,1])$, $\partial_su\in\LL^{q}_{-1,p}([0,1])$, $u(1,\cdot)=0$ and equation \eqref{eqn.u} holds in $\LL^{q'}_{-1,p'}([0,1])$, i.e. for every $\phi\in \LL^{q'}_{1,p'}([0,1])$ 
		\begin{align}\label{eqn.uweak}
			\int_0^1\wei{	(\partial_s+a^{ij}\partial^2_{ij})u_t,\phi_t}_{L_{-1,p}(\mR^d)\times L_{1,p'}(\mR^d)}dt=\int_0^1\wei{f_t,\phi_t}_{L_{-1,p}(\mR^d)\times L_{1,p'}(\mR^d)}dt.
		\end{align}

		Likewise, a measurable function $v:[0,1]\times\Rd$ is a solution to \eqref{eqn.v} if $v\in\LL^{q'}_{1,p'}([0,1])$, $\partial_tv\in\LL^{q'}_{-1,p'}([0,1])$, $v(0,\cdot)=0$  and equation \eqref{eqn.v} holds in $\LL^{q}_{-1,p}([0,1])$, i.e. for every $\phi\in \LL^{q}_{1,p}([0,1])$ 
		\begin{align}\label{eqn.vweak}
		\int_0^1\wei{\partial_tv_s- \partial^2_{ij}(a^{ij}v_s),\phi_s}_{L_{-1,p'}(\mR^d)\times L_{1,p}(\mR^d)}ds+\int_0^1\wei{g_s,\phi_s}_{L_{-1,p'}(\mR^d)\times L_{1,p}(\mR^d)}ds=0.
		\end{align}
	\end{definition}
	In the above definitions, we have implicitly understood that $a^{ij}\partial^2_{ij}u$ and $\partial^2_{ij}(a^{ij}v)$ are well-defined distributions in $\LL^{q'}_{-1,p'}([0,1])$ and $\LL^q_{-1,p}([0,1])$  respectively.
	To see this, we need the following multiplication result:
	\begin{lemma}\label{lem.multiplication}
		Let $p,p_1,p_2$ be real numbers in $(1,\infty)$ and let $\nu\in(0,1]$.

		(i) Assume that $p_1,p_2\ge p$ and that $\frac1p\le\frac1{p_1}+\frac1{p_2}<\frac1p+\frac \nu d$. Then the pointwise multiplication is a continuous bilinear map
		\begin{align*}
			L_{\nu,p_1}(\Rd)\times L_{\nu,p_2}(\Rd)\to L_{\nu,p}(\Rd).
		\end{align*}

		(ii) Assume that $p_1\ge p$, that $p_2\ge p_1'$ and that $\frac1p\le\frac1{p_1}+\frac1{p_2}<\frac1p+\frac \nu d$. Then the pointwise multiplication is a continuous bilinear map
		\begin{align*}
			L_{-\nu,p_1}(\Rd)\times L_{\nu,p_2}(\Rd)\to L_{-\nu,p}(\Rd).
		\end{align*}

		(iii) Let $g$ be a bounded measurable function such that $\nabla g\in L_{p_0}(\Rd)$ for some $p_0\in(d,\infty)$. Let $f$ be in $L_{1,p'}(\Rd)$, $h$ be in $L_{-1,p}(\Rd)$ and assume that $\frac1p+\frac1{p_0}<1$.
		Then $fg$ belongs to $L_{1,p'}(\Rd)$, $gh$ belongs to $L_{-1,p}(\Rd)$ and
		\begin{align}
			\label{est.g1}
			&\|fg\|_{L_{1,p'}(\Rd)}\les(\|g\|_{L_\infty(\Rd)}+\|\nabla g\|_{L_{p_0}(\Rd)})\|f\|_{L_{1,p'}(\Rd)},
			\\&\|gh\|_{L_{-1,p}(\Rd)}\les(\|g\|_{L_\infty(\Rd)}+\|\nabla g\|_{L_{p_0}(\Rd)})\|h\|_{L_{-1,p}(\Rd)}.
			\label{est.g2}
		\end{align}
	\end{lemma}
	\begin{proof}
		(i-ii) are consequences of \cite[Lemma 2.2]{Zhangzhao}. Concerning (iii), define $p_3$ by $\frac1{p'}=\frac1{p_0}+\frac1{p_3}$. Then by H\"older inequality
		\begin{align*}
			\|\nabla gf\|_{L_{p'}(\Rd)}\le\|\nabla g\|_{L_{p_0}(\Rd)}\|f\|_{L_{p_3}(\Rd)}.
		\end{align*}
		The embedding $L_{1,p'}(\Rd)\hookrightarrow L_{p_3}(\Rd)$ is valid if $\frac1{p'}-\frac1d\le\frac1{p_3}\le\frac1{p'}$,
		which is justified by our assumption.
		It follows that $\|\nabla gf\|_{L_{p'}(\Rd)}\les\|\nabla g\|_{L_{p_0}(\Rd)}\|f\|_{L_{1,p'}(\Rd)}$. It is evident that
		\begin{align*}
			\| gf\|_{L_{p'}(\Rd)}+\| g\nabla f\|_{L_{p'}(\Rd)}\les\| g\|_{L_{\infty}(\Rd)}\|f\|_{L_{1,p'}(\Rd)}.
		\end{align*}
		From here, we obtain \eqref{est.g1}. To show \eqref{est.g2}, we note that by duality and \eqref{est.g1},
		\begin{align*}
			\|fgh\|_{L_1(\Rd)}
			&\les\|h\|_{L_{-1,p}(\Rd)} \|fg\|_{L_{1,p'}(\Rd)}
		\\&	\les\|h\|_{L_{-1,p}(\Rd)}(\|g\|_{L_\infty(\Rd)}+\|\nabla g\|_{L_{p_0}(\Rd)}) \|f\|_{L_{1,p'}(\Rd)}.
		\end{align*}
		This implies \eqref{est.g2} by duality.
	\end{proof}

	\begin{proposition}\label{prop.aD2u}
		For every $u\in\LL^{q}_{1,p}([0,1])$ and $v\in\LL^{q'}_{1,p'}([0,1])$, under \cref{con.A1} we have
		\begin{align*}
			a^{ij}\partial^2_{ij}u\in\LL^{q}_{-1,p}([0,1])
			\tand \partial^2_{ij}(a^{ij}v)\in\LL^{q'}_{-1,p'}([0,1]).
		\end{align*}
	\end{proposition}
	\begin{proof}
		Using \cref{lem.multiplication}(iii), we see that
		\begin{align*}
			\|a^{ij} \partial^2_{ij}u\|_{L_{-1,p}(\Rd)}
			&\les(\|a^{ij}\|_{L_\infty(\Rd)}+\|\nabla a^{ij}\|_{L_{p_0}(\Rd)})\|\partial^2_{ij}u\|_{L_{-1,p}(\Rd)}
			\\&\les(\|a^{ij}\|_{L_\infty(\Rd)}+\|\nabla a^{ij}\|_{L_{p_0}(\Rd)})\|u\|_{L_{1,p}(\Rd)}
		\end{align*}
		and		
		\begin{align*}
			\|\partial^2_{ij}(a^{ij}v)\|_{L_{-1,p'}(\Rd)}
			\les\|a^{ij}v\|_{L_{1,p'}(\Rd)}
			\les(\|a^{ij}\|_{L_\infty(\Rd)}+\|\nabla a^{ij}\|_{L_{p_0}(\Rd)})\|v\|_{L_{1,p'}(\Rd)}.
		\end{align*}
		These estimates imply the result.
	\end{proof}

	\begin{theorem}\label{thm.uv}
		Under \cref{con.A1}, there exist a unique solution $u$ to \eqref{eqn.u} and a unique solution $v$ to \eqref{eqn.v}. Furthermore, we have
		\begin{align}\label{est.upq}
			&\|u\|_{\LL^q_{1,p}([0,1])}+\|\partial_s u\|_{\LL^q_{-1,p}([0,1])}\le N\|f\|_{\LL^q_{-1,p}([0,1])},
			\\&\|v\|_{\LL^{q'}_{1,p'}([0,1])}+\|\partial_tv\|_{\LL^{q'}_{-1,p'}([0,1])}\le N\|g\|_{\LL^{q'}_{-1,p'}([0,1])},
			\label{est.vpq}
		\end{align}
		  where $N$ is a finite positive constant depending on $d,p,q,p_0, k_1, k_3$.
	\end{theorem}
	Before giving the proof of the above theorem, we show several auxiliary results.
	\begin{lemma}[{\cite[Lemma 4.1]{Zhangzhao}}]\label{lem.localize}
	 	Let $\zeta$ be a nonzero smooth function with compact support. Define $\zeta^z(x)=\zeta(x-z)$. For any $\nu\in\R$ and $p\in(1,\infty)$, there exists a constant $C\ge1$ depending only on $\nu,p,\zeta$ such that for any $f\in L_{\nu,p}(\mR^d)$,
	 	\begin{align*}
	 		C^{-1}\|f\|_{L_{\nu,p}(\Rd)}\le\left(\int_\Rd\|f \zeta^z\|_{L_{\nu,p}(\Rd)}^pdz\right)^{1/p}\le C \|f\|_{L_{\nu,p}(\Rd)}.	
	 	\end{align*}	
	\end{lemma}
	\begin{lemma}[{\cite[Lemma 2.5]{Kim}}]\label{lem.kry}
		For $k=1,\ldots,n$, let $a^k:\R_+\to\Rd\times\Rd$ be a measurable function satisfying \eqref{ellptic-cona}.
		For fixed $\nu\in\R$, $p\in(1,\infty)$, let $u^k\in\LL^q_{\nu,p}([0,1])$ solve the following PDE
		\begin{align*}
		 	(\partial_s+a_k^{ij}\partial^2_{ij})u^k=f^k, \quad u(1,\cdot)=0.
		\end{align*}
		Then
		\begin{align*}
			\int_0^1\prod_{k=1}^n\|\nabla^2 u^k(t)\|_{L_{\nu,p}(\Rd)}^pdt\le N\sum_{k=1}^n\int_0^1\|f^k\|_{L_{\nu,p}(\Rd)}^p\prod_{j\neq k}\|\nabla^2 u^j(t)\|_{L_{\nu,p}(\Rd)}^pdt.
		\end{align*}
	\end{lemma}

	We will make use the following: 

	\noindent\textbf{Convention.}  For a parameter $\rho>0$, we write $o_\rho$ and $C_\rho$ for any constants whose exact values depend on $\rho$ and may change from one instance to another, but it is always enforced that $\lim_{\rho\to0}o_\rho=0$. In particular, the inequality $A\les D+o_\rho F+C_\rho E$ means that $A\le cD+co_\rho F+cC_\rho E$ for some constant $c$ independent from $\rho$.
	
	\begin{lemma}\label{lem.asmall}
		Assuming \cref{con.A1}. Let  $\psi_1$ be a smooth function supported in the ball $\cB_1:=\{x\in\mR^d:|x|\leq 1\}$. For each $\rho>0$ and $z\in\Rd$, define $\cB_\rho^z:=\{x\in\mR^d:|x-z|\leq \rho\}$, $\psi_\rho^z(x):=\psi_1(\frac{x-z}{\rho})$ and $a(z)(t,x)=\frac1{|\cB^z_\rho|}\int_{\cB^z_\rho}a(t,y)dy$.
		Then   we have
   \begin{align}
      \label{est:rho-froz}\sup_{(t,z)\in[0,1]\times\mR^d}\|(a-a(z)) \psi_\rho^z\|_{L_{1,p_0}(\Rd)}\lesssim o_\rho
  \end{align}

		Here $\psi_\rho^z(x):=\psi_1(\frac{x-z}{\rho})$, $x,z\in\mR^d$, $\rho>0$.
	\end{lemma}	
	\begin{proof} Observe that
		\begin{align*}
		I_1:=	 \sup_{z\in\mR^d}\|(a-a(z) )\psi_\rho^z\|_{L_{p_0}(\Rd)}\les\|\psi_1\|_{L_\infty(\Rd)} \|a\|_{L_\infty(\Rd)}\|\1_{\cB_{2 \rho}}\|_{L_{p_0}(\Rd)}
		\end{align*}
		and
		\begin{align*}
		I_2:=	\sup_{z\in\mR^d}\|\nabla a \psi_\rho^z\|_{L_{p_0}(\Rd)}\les\|\psi_1\|_{L_\infty(\Rd)}\|\nabla a\cdot\1_{\cB_{2 \rho}}\|_{L_{p_0}(\Rd)}.		
		\end{align*}
		Using Poincar\'e inequality (\cite[Theorem 3.17]{Guisti})
		\begin{align*}
		I_3:=	\sup_{z\in\mR^d}\|(a-a(z)) \nabla\psi_\rho^z\|_{L_{p_0}(\Rd)}
			\les \|\nabla \psi_1\|_{L_\infty(\Rd)} \|\nabla a\cdot\1_{\cB_{2 \rho}}\|_{L_{p_0}(\Rd)}.
		\end{align*}
 \cref{con.A1} implies that $I_1+I_2+I_3\le o_\rho$. Furthermore, because 
		\begin{align*}
		  \sup_{z\in\mR^d}\|(a-a(z)) \psi_\rho^z\|_{L_{1,p_0}(\Rd)}\les  (I_1+I_2+I_3),
		\end{align*}
		we obtain the desired result.
	\end{proof}

 \begin{lemma}\label{lem.u}
	Assuming  \cref{con.A1} and additionally that $q\geq p$.
	Then there exists a unique solution $u$  to \eqref{eqn.u} which satisfies \eqref{est.upq}.
	\end{lemma}
	\begin{proof}
		By Marcinkiewicz interpolation theorem, it suffices to show the result when $q=\bar np$ for any integer $\bar n\ge1$.
		Let $\bar n\ge1$ be a fixed integer.
		By the method of continuity (e.g. see \cite[Theorem 5.2]{MR1814364}), it suffices to show that there exists positive  $N=N(d,p,q,p_0,k_1,k_3)$ such that whenever $u$ is a solution to \eqref{eqn.u},
		\begin{align}\label{est.upnp}
			\|u\|_{\LL^{\bar np}_{1,p}([0,1])}\le N\|f\|_{\LL^{\bar np}_{-1,p}([0,1])}.
		\end{align}
		Note that if $u$ is a solution to \eqref{eqn.u}, then using \cref{prop.aD2u}, the above estimate implies that
		\begin{align*}
			\|\partial_tu\|_{\LL^{\bar np}_{-1,p}([0,1])}\les \|f\|_{\LL^{\bar np}_{-1,p}([0,1])}.	
		\end{align*}
		
		Let $\rho>0$ be a fixed constant and $\phi$ be a nonnegative smooth function such that $\phi$ is supported in the ball $\cB_\rho:=\{x\in\Rd:|x|\le \rho\}$ and $\|\phi\|_{L_p(\Rd)}=1$. For each $z\in\Rd$, define $a(z)$ as in \cref{lem.asmall} and
		\[
			\phi^z(x):=\phi(x-z),  \quad u^z(s,x):=u(s,x)\phi^z(x), \quad f^z(s,x):=f(s,x)\phi^z(x).
		\]
		Then $u^z$ satisfies  the relation
		\begin{align}
			&\partial_tu^z+a^{ij}(z)\partial^2_{ij}u^z=F^z, \quad u^z(1,\cdot)=0,
			\label{eqn.uz}
		\\&F^z:=f \phi^z+2a^{ij} \partial_i u \partial_j \phi^z +a^{ij} u \partial^2_{ij}\phi^z+(a^{ij}(z)-a^{ij})\partial^2_{ij}u^z.	 \nonumber
		\end{align}
		The proof is now divided into several steps.

		\textit{Step 1.} We show that for each $t\in[0,1]$,
		\begin{align}
			\left(\int_\Rd \|F^z_t\|_{L_{-1,p}(\Rd)}^pdz\right)^{1/p}\les \|f_t\|_{L_{-1,p}(\Rd)}+C_\rho\|u_t\|_{L_{p}(\Rd)}
   +o_\rho
   \|u_t\|_{L_{1,p}(\Rd)}.\label{tmp.Fz}
		\end{align}
   		Applying \cref{lem.localize} (with $\zeta=\phi,\partial_j \phi,\partial^2_{ij} \phi$ respectively) and \cref{lem.multiplication},  noting that $\|\nabla\phi\|_{\infty}+\|\nabla^2\phi\|_{\infty}\les C_\rho$,  we obtain that
		\begin{multline*}
			\left(\int_\Rd\|f \phi^z+2a^{ij} \partial_i u \partial_j \phi^z +a^{ij} u \partial^2_{ij}\phi^z\|_{L_{-1,p}(\Rd)}^pdz\right)^{1/p}
			\\\les \|f\|_{L_{-1,p}(\Rd)}+\|\nabla\phi\|_{\infty}\|a^{ij}\partial_i u\|_{L_{-1,p}(\Rd)}+|\nabla^2\phi\|_{\infty}\|a^{ij}u\|_{L_{-1,p}(\Rd)}
			\\\les\|f\|_{L_{-1,p}(\Rd)}+ C_\rho
			\|u\|_{L_p(\Rd)}.
		\end{multline*}
		Let $\psi$ be a smooth function on $\Rd$ such that $\psi(x)=1$ if $|x|\le \rho$ and $\psi(x)=0$ if $|x|\ge 2 \rho$. Define $\psi^z(x)=\psi(x-z)$.
		Applying \cref{lem.multiplication} (whose hypothesis is justified because $p_0>p'$ and $p_0>d$), we have
		\begin{align*}
			 \|(a^{ij}(z)-a^{ij})\partial^2_{ij}u^z\|_{L_{-1,p_0}(\Rd)}\les\|(a^{ij}(z)-a^{ij})\psi^z\|_{L_{1,p_0}(\Rd)}\|\partial^2_{ij}u^z\|_{L_{-1,p}(\Rd)}.
		\end{align*}
		By \cref{lem.asmall}, we see that
		\begin{align*}
			\sup_{z\in\Rd}\|(a^{ij}(z)-a^{ij})\psi^z\|_{L_{1,p_0}(\Rd)}\le 
   {o_\rho}.
		\end{align*}
		Hence,
		\begin{align*}
			\|(a^{ij}(z)-a^{ij})\partial^2_{ij}u^z\|_{L_{-1,p_0}(\Rd)}\le  o_\rho
   \|u^z\|_{L_{1,p}(\Rd)}.
		\end{align*}
		It is easy to see that $\int_\Rd\|h \phi^z\|_{L_p(\Rd)}^pdz=\|h\|_{L_p(\Rd)}$ for any $h\in L_p(\Rd)$. Hence, by Minkowski inequality and \cref{lem.localize}, we have 
		\begin{align}
			&\left(\int_\Rd\|u^z\|_{L_{1,p}}^pdz \right)^{1/p}
			\le \left(\int_\Rd\|\nabla u^z\|_{L_{p}}^pdz \right)^{1/p}+\left(\int_\Rd\|u^z\|_{L_{p}}^pdz \right)^{1/p}
			\nonumber\\&\le \left(\int_\Rd\|\nabla u \phi^z\|_{L_{p}}^pdz \right)^{1/p}+\left(\int_\Rd\|u \nabla \phi^z\|_{L_{p}}^pdz \right)^{1/p}+\left(\int_\Rd\|u^z\|_{L_{p}}^pdz \right)^{1/p}
			\nonumber\\&\lesssim \|\nabla u\|_{L_p(\Rd)}+C_\rho
			\|u\|_{L_p(\Rd)}.\label{tmp.uz1}
		\end{align}
		This shows that
		\begin{align*}
			\left(\int_\Rd \|(a^{ij}(z)-a^{ij})\partial^2_{ij}u^z\|_{L_{-1,p}(\Rd)}^pdz\right)^{1/p}\les  o_\rho
	\|u\|_{L_{1,p}(\Rd)}+C_\rho\|u\|_{L_p(\Rd)}.
		\end{align*}
		Hence, we have \eqref{tmp.Fz}.

		\textit{Step 2.} We show that for every integer $1\le n\le\bar n$ and every $s\in[0,1]$,
		\begin{align}\label{tmp.u2}
			\| u\|_{\LL_{1,p}^{np}([s,1])}\les \|f\|_{\LL_{-1,p}^{np}([s,1])}+\|u\|_{\LL_{-1,p}^{np}([s,1])}.
		\end{align}
		Since
		\begin{align}\label{tmp.equivnorm}
		 	\|u\|_{\LL_{1,p}^{np}([s,1])}\approx\|\nabla^2u \|_{\LL_{-1,p}^{np}([s,1])}+\|u\|_{\LL_{-1,p}^{np}([s,1])},
		\end{align}
		  it suffices to estimate $\|\nabla^2u \|_{\LL_{-1,p}^{np}([s,1])}$.
		From \cref{lem.localize}, we have
		\begin{align}\label{tmp.nton-1}
			\|\nabla^2 u\|_{\LL^{np}_{-1,p}([s,1])}^{np}
			&\les\int_s^1\left(\int_\Rd\|\nabla^2 u^z_t\|_{L_{-1,p}(\Rd)}^pdz\right)^n dt+\int_s^1\|u_t\|_{L_p(\Rd)}^{np}dt.
		\end{align}	
From Tonelli's theorem, \cref{lem.kry} and \eqref{tmp.Fz},
		\begin{multline*}
			\int_s^1\left(\int_\Rd\|\nabla^2 u^z_t\|_{L_{-1,p}(\Rd)}^pdz\right)^n dt
			\\\les\int_s^1 \left(\int_\Rd\|\nabla^2 u^z_t\|_{L_{-1,p}(\Rd)}^pdz\right)^{n-1}\left(\|f_t\|_{L_{-1,p}(\Rd)}+
   \|u_t\|_{L_{p}(\Rd)}+o_\rho
   \|u_t\|_{L_{1,p}(\Rd)}\right)^pdt.
		\end{multline*}
		Applying H\"older inequality, we have
		\begin{multline*}
			\int_s^1\left(\int_\Rd\|\nabla^2 u^z_t\|_{L_{-1,p}(\Rd)}^pdz\right)^n dt
			\les\left[\int_s^1 \left(\int_\Rd\|\nabla^2 u^z_t\|_{L_{-1,p}(\Rd)}^pdz\right)^{n}dt\right]^{1-\frac1n}
			\\\times
			\left[\int_s^1\left(\|f\|_{L_{-1,p}(\Rd)}
			+C_\rho \|u\|_{L_p(\Rd)}
			+o_\rho\|u\|_{L_{1,p}(\Rd)}\right)^{np}dt\right]^{\frac1n},
		\end{multline*}
		which yields that
		\begin{align*}
			\int_s^1&\left(\int_\Rd\|\nabla^2 u^z_t\|_{L_{-1,p}(\Rd)}^pdz\right)^n dt
			\\&\les\int_s^1\left(\|f\|_{L_{-1,p}(\Rd)}^{np}
			+C_\rho \|u\|_{L_p(\Rd)}^{np}
			+o_\rho  \|u\|_{L_{1,p}(\Rd)}^{np}\right)dt.
		\end{align*}
		Putting this into \eqref{tmp.nton-1}, we obtain that 
	\begin{align*}
		\|\nabla^2 u\|_{\LL_{-1,p}^{np}([s,1])}^{np}\les \|f\|_{\LL_{-1,p}^{np}([s,1])}^{np}
		+C_\rho\|u\|_{\LL_{p}^{np}([s,1])}^{np}
		+ o_\rho  \|u\|_{\LL_{1,p}^{np}([s,1])}^{np}.
	\end{align*}
	Using interpolation inequality 
	\begin{align}\label{tmp.interpolation}
		\|u\|_{\LL_{p}^{np}([s,1])}\les 
 C_\rho\|u\|_{\LL_{-1,p}^{np}([s,1])}+o_\rho
   \|u\|_{\LL_{1,p}^{np}([s,1])},
	\end{align}
	we get
	\begin{align*}
		\|\nabla^2 u\|_{\LL_{-1,p}^{np}([s,1])}\les \|f\|_{\LL_{-1,p}^{np}([s,1])}+
		C_\rho\|u\|_{\LL_{-1,p}^{np}([s,1])}
+o_\rho
   \|u\|_{\LL_{1,p}^{np}([s,1])}.
	\end{align*}
	In view of \eqref{tmp.equivnorm}, we have
	\begin{align*}
		\|u\|_{\LL_{1,p}^{np}([s,1])}
		&\les\|f\|_{\LL_{-1,p}^{np}([s,1])}+C_\rho\|u\|_{\LL_{-1,p}^{np}([s,1])}+o_\rho
   \|u\|_{\LL_{1,p}^{np}([s,1])}.
	\end{align*}
	By choosing $\rho$ sufficiently small, we derive  derive \eqref{tmp.u2}  from the above estimate.
	
	\textit{Step 3.} We show that
	\begin{align}\label{tmp.u3}
		\|u\|_{\LL^\infty_{-1,p}([0,1])}\les\|f\|_{\LL^p_{-1,p}([0,1])}.
	\end{align}
	From \eqref{eqn.uz}, we have
	\begin{align*}
		u^z_s=\int_s^1P_{\Sigma_{s,t}(z)}F^z_tdt, \quad\text{where}\quad \Sigma_{s,t}(z)=2\int_s^t a(r,z)dr.
	\end{align*}
Then   Minkowski inequality and \cite[Theorem 5.30]{MR3086433} yield 
	\begin{align*}
		\|u_s^z\|_{L_{-1,p}(\Rd)}\les\int_s^1\|F^z_t\|_{L_{-1,p}(\Rd)}dt.
	\end{align*}
	Applying H\"older inequality, \cref{lem.localize}, \eqref{tmp.Fz} and the interpolation inequality \eqref{tmp.interpolation}, we obtain from the above that
	\begin{align*}
		\|u_s\|_{L_{-1,p}(\Rd)}^p
		&\les\int_s^1\int_\Rd\|F^z_t\|_{L_{-1,p}(\Rd)}^pdzdt
		\\&\les\int_s^1\left[\|f_t\|^p_{L_{-1,p}(\Rd)}+\|u_t\|^p_{L_{-1,p}(\Rd)}+\|u_t\|^p_{L_{1,p}(\Rd)}\right]dt.
	\end{align*}
	Applying \eqref{tmp.u2} (with $n=1$), we have
	\begin{align*}
		\|u_s\|^p_{L_{-1,p}(\Rd)}\les\|f\|^p_{\LL^p_{1,p}([0,1])}+\int_s^1\|u_t\|_{L_{-1,p}(\Rd)}^pdt,
	\end{align*}
	which implies \eqref{tmp.u3} by Gr\"onwall inequality.

	\textit{Step 4.}
	Using \eqref{tmp.u3} in \eqref{tmp.u2} yields
	\begin{align*}
		\|u\|_{\LL^{np}_{1,p}([0,1])}\les\|f\|_{\LL^{np}_{1,p}([0,1])}+\|f\|_{\LL^p_{1,p}([0,1])}
		\les\|f\|_{\LL^{np}_{1,p}([0,1])},
	\end{align*}
	which implies \eqref{est.upq}.
	\end{proof}
	\begin{lemma}\label{lem.v}
	Assuming \cref{con.A1} and additionally that $q'\geq p'$. Then there exists a unique solution $v$  to \eqref{eqn.v}  which satisfies \eqref{est.vpq}.
	\end{lemma}
	\begin{proof}
		The proof is similar to that of \cref{lem.u}.
		The main differences are the computations in step 1 of the proof of \cref{lem.u}, which we will explain below.
		Let $v$ be a solution to \eqref{eqn.v}.
		Define $\phi^z,a(z)$ as in the proof of \cref{lem.u}. In addition, define $g^z(t,z)=g(t,x)\phi^z(x)$ and $v^z(t,x)=v(t,x)\phi^z(x)$. Then $v^z$ satisfies the parabolic equation
		\begin{align*}
			&\partial_t v^z- a^{ij}(z)\partial^2_{ij}v^z+G^z=0, \quad v^z(0,\cdot)=0,
		\end{align*}
		where
		\begin{align*}
			G^z&=g^z+a^{ij}(z)\partial^2_{ij}v^z- \partial^2_{ij}(a^{ij}v)\phi^z
			\\&=g^z- \partial^2_{ij}\left((a^{ij}-a^{ij}(z))v^z\right)+2\partial_i(a^{ij}v)\partial_j \phi^z+a^{ij}v \partial^2_{ij}\phi^z.
		\end{align*}
		Let $\psi$ be a smooth function on $\Rd$ such that $\psi(x)=1$ if $|x|\le \rho$ and $\psi(x)=0$ if $|x|\ge 2 \rho$. Define $\psi^z(x)=\psi(x-z)$.
		Applying \cref{lem.multiplication,lem.asmall}
		\begin{multline*}
		   \Vert\partial_{ij}((a^{ij}-a^{ij}(z))v^z) \Vert_{L_{-1,p'}(\mR^d)}
		   \les\|(a^{ij}-a^{ij}(z))v^z\|_{L_{1,p'}(\Rd)}
		   \\\lesssim  \Vert (a^{ij}-a^{ij}(z))\psi^z \Vert_{L_{1,p_0}(\mR^d)}\Vert v^z \Vert_{L_{1,p'}(\mR^d)}
		   \le o_\rho\|v^z\|_{L_{1,p'}(\Rd)}
		\end{multline*}
		and
		\begin{align*}
			\|a^{ij}v \partial^2_{ij}\phi^z\|_{L_{-1,p'}(\Rd)}
			\les\|a^{ij}\psi^z\|_{L_{1,p_0}(\Rd)}\|v \partial^2_{ij}\phi^z\|_{L_{-1,p'}(\Rd)}
			\les \|v \partial^2_{ij}\phi^z\|_{L_{-1,p'}(\Rd)}.
		\end{align*}
		Similar to \eqref{tmp.uz1}, we have
		\begin{align*}
			\left(\int_\Rd\|v^z\|_{L_{1,p}(\Rd)}^{p'}\right)^{1/p'}\le
			\|v\|_{L_{1,p'}(\Rd)}+C_\rho\|v\|_{L_p'(\Rd)}.
		\end{align*}
		This yields
		\begin{multline*}
			\left(\int_\Rd \Vert\partial_{ij}((a^{ij}-a^{ij}(z))v^z) +a^{ij}v \partial^2_{ij}\phi^z \Vert_{L_{-1,p'}(\mR^d)}^{p'}dz\right)^{1/p'}
			\\\les C_\rho \|v\|_{L_{p'}(\Rd)}+o_\rho\|v\|_{L_{1,p'}(\Rd)}.
		\end{multline*}		
		Applying \cref{lem.localize}, we have
		\begin{align*}
			\left(\int_\Rd\|g^z+2 \partial_i(a^{ij}v)\partial_j \phi^z\|_{L_{-1,p'}(\Rd)}^{p'}dz \right)^{{1/p'}}
			&\les \|g\|_{L_{-1,p'}(\Rd)}+C_\rho\|\partial_i(a^{ij}v)\|_{L_{-1,p'}(\Rd)}
			\\&\les \|g\|_{L_{-1,p'}(\Rd)}+C_\rho\|v\|_{L_{p'}(\Rd)}.
		\end{align*}
		These estimates imply that
		\begin{align*}
			\left(\int_\Rd\|G^z\|^{p'}_{L_{-1,p'}(\Rd)}dz\right)^{1/p'}\les \|g\|_{L_{-1,p'}(\Rd)}+C_\rho\|v\|_{L_{p'}(\Rd)}+ o_\rho\|v\|_{L_{1,p'}(\Rd)}.
		\end{align*}
		Using the interpolation inequality
		\begin{align*}
			\|v\|_{L_{p'}(\Rd)}\le C_\rho\|v\|_{L_{-1,p'}(\Rd)}+o_\rho\|v\|_{L_{1,p'}(\Rd)},
		\end{align*}
		we obtain from the previous estimate that
		\begin{align*}
			\left(\int_\Rd\|G^z\|^{p'}_{L_{-1,p'}(\Rd)}dz\right)^{1/p'}\les \|g\|_{L_{-1,p'}(\Rd)}+C_\rho\|v\|_{L_{-1,p'}(\Rd)}+ o_\rho\|v\|_{L_{1,p'}(\Rd)}.
		\end{align*}
		One can now follow steps 2,3 of the proof of \cref{lem.u} to obtain \eqref{est.vpq}.	
	\end{proof}
	\begin{proof}[\bf Proof of \cref{thm.uv}]
	Concerning equation \eqref{eqn.u}, by the method of continuity  it suffices to show  \eqref{est.upq} whenever $u$ is a solution to \eqref{eqn.u}. The case $q\geq p$ has been treated in \cref{lem.u}.  Consider the case $q<p$, which is equivalent to $q'>p'$. For each $g\in\LL^{q'}_{-1,p'}([0,1])$, let $v$ be the solution to \eqref{lem.v}, which exists uniquely by \cref{lem.v}.
	We take $\phi=v$ in \eqref{eqn.uweak} and use the equation \eqref{eqn.v} for  $v$ to see that
	\begin{align*}
	\int_0^1 \wei{u_s,g_s}_{L_{1,p}(\Rd)\times L_{-1,p'}(\Rd)}ds=  \int_0^1 \wei{v_t,f_t}_{L_{1,p'}(\Rd)\times L_{-1,p}(\Rd)}dt.
	\end{align*}
	Applying H\"older inequality and \eqref{est.vpq}, we have
	\begin{align*}
		\left|\int_0^1 \wei{v_t,f_t}_{L_{1,p'}(\Rd)\times L_{-1,p}(\Rd)}dt\right|
		&\le\|v\|_{\LL^{q'}_{1,p'}([0,1])} \|f\|_{\LL^{q'}_{-1,p'}([0,1])}
		\\&\les\|g\|_{\LL^{q'}_{-1,p'}([0,1])} \|f\|_{\LL^{q'}_{-1,p'}([0,1])},
	\end{align*}
	and hence
	\begin{align*}
		\left|\int_0^1 \wei{u_s,g_s}_{L_{1,p}(\Rd)\times L_{-1,p'}(\Rd)}ds\right|
		\les\|g\|_{\LL^{q'}_{-1,p'}([0,1])}\|f\|_{\LL^{q'}_{-1,p'}([0,1])}.
	\end{align*}
	Since $g$ is arbitrary, this implies \eqref{est.upq}. The result for equation \eqref{eqn.v} follows from similar arguments.
	\end{proof}


		In the remaining, 
		we consider the parabolic differential equation
		\begin{align}\label{eqn.pde}
			\partial_su+\frac12 a^{ij}\partial^2_{ij}u =f, \quad u(t,\cdot)=0
		\end{align}
		where $f\in\LL^q_{-1,p}([0,1])$, $t\in(0,1]$ is fixed. 
		Whenever the dependence on $t$ plays a role, we write $u^t_s(x)$ for the solution to \eqref{eqn.pde} evaluated at $(s,x)$, $s\le t$, $x\in\Rd$. 
		We quantify the dependence on the terminal time of various quantities related to $u^t_s$ under \cref{con.A1}. These estimates are used in \cref{sec:regularizing_properties_of_the_continuum_paths}  where we particularly take $(a^{ij}):=\sigma\sigma^*$ for a $\sigma$ satisfying \cref{con.A} with $q_0=\infty$. 
		\begin{theorem}\label{thm.pde-1} Assuming \cref{con.A1}. Let $q\in(2,\infty)$ and $p\in(1,\infty)$ be such that
		$\frac1p+\frac1{p_0}<1$. Then for every $\nu\in[0,1]$, every $0\le s\le t\le 1$ and $f\in {\LL^q_{-\nu,p}([0,1])}$,  we have 
			\begin{align}
				\|u^t_s\|_{L_p(\Rd)}\le N(t-s)^{1-\frac \nu2-\frac1q}\|f\|_{\LL^q_{-\nu,p}([s,t])}.
				\label{est.u-nu}
			\end{align}
		\end{theorem}
		\begin{lemma}\label{lem.ust}
			Let $q\in(2,\infty)$ and $p\in(1,\infty)$. Let $t\in[0,1]$.  If $u(r,x)=0$ for $r\in[t,1]$,  then for every $s\in[0,t]$ 
			\begin{gather}\label{u-1p-est}
				\|u(s)\|_{L_p(\Rd)}\les (t-s)^{\frac12-\frac1q}\|\partial_su+\frac12 \Delta u\|_{\LL^q_{-1,p}([s,t])} 
				 \text{ whenever }u\in \LL^q_{1,p}([0,1])
				 \\\shortintertext{and}
				 \label{up-est}
				\|u(s)\|_{L_p(\Rd)}\les (t-s)^{1-\frac1q}\|\partial_su+\frac12 \Delta u\|_{\LL^q_{p}([s,t])} \text{ whenever }u\in \LL^q_{2,p}([0,1]).
			\end{gather}	
		\end{lemma}
		\begin{proof}
		By approximation, we can assume that $u$ is a smooth function on $[0,1]\times\Rd$ with compact support. 
		Put $g:=\partial_su+\frac12 \Delta u$.	Then by Duhamel's formula $u(s,x)=\int_s^tP_{s,r}g(r,x)dr$.
		Applying Minkowski inequality and \cite[Theorem 5.30]{MR3086433}, we have
		\begin{align*}
			\|u_s\|_{L_p(\Rd)}\le\int_s^t\|P_{s,r}g_r\|_{L_p(\Rd)}dr
			\les\int_s^t(r-s)^{-\frac12}\|g_r\|_{L_{-,1,p}(\Rd)}dr.
		\end{align*}
		Using H\"older inequality, we have $\|u_s\|_{L_p(\Rd)}\les(t-s)^{\frac12-\frac1q}\|g\|_{\LL^q_{-1,p}([s,t])}$.
		This shows \eqref{u-1p-est}.
		Inequality \eqref{up-est} is obtained in the same way.
		\end{proof}
		
		\begin{proof}[\bf Proof of \cref{thm.pde-1}]
			By interpolation, it suffices to show that
			\begin{align}
				&\|u^t_s\|_{L_p(\Rd)}\le N(t-s)^{\frac 12-\frac1q}\|f\|_{\LL^q_{-1,p}([s,t])}\quad \text{for } f\in\LL^q_{-1,p}([0,1]),
				\label{tmp.u-1}
				\\&\|u^t_s\|_{L_p(\Rd)}\le N(t-s)^{1-\frac1q}\|f\|_{\LL^q_{p}([s,t])}\quad\text{for } f\in\LL^q_{p}([0,1]).
				\label{tmp.u0}
			\end{align}
			
			From \cref{thm.uv}, we have $\|u^t\|_{\LL^q_{1,p}([s,t])}\les\|f\|_{\LL^q_{-1,p}([s,t])}$ for every $0\le s\le t\le1$.
			From \eqref{eqn.pde}, we have $\partial_s u+\frac12 \Delta u=f+\frac12(\Delta-a^{ij}\partial^2_{ij})u$.
			It follows from \cref{lem.ust} that
			\begin{align*}
			  \|u^t_s\|_{L_p(\Rd)}\les (t-s)^{\frac12-\frac1q}(\|f\|_{\LL^q_{-1,p}([s,t])}+\|(\delta^{ij}-a^{ij})\partial^2_{ij}u^t\|_{\LL^q_{-1,p}([s,t])} ),
			\end{align*}
			where $\delta^{ij}=1$ if $i=j$ and $\delta^{ij}=0$ otherwise. From \cref{lem.multiplication}(iii) and the hypotheses, we have 
			\begin{align*}
				\|(\delta^{ij}-a^{ij})\partial^2_{ij}u^t\|_{\LL^q_{-1,p}([s,t])}
				&\les(\|a^{ij}\|_{\LL^\infty_\infty([s,t])}+\|\nabla a^{ij}\|_{\LL^\infty_{p_0}([s,t])})\|\partial^2_{ij}u^t\|_{\LL^q_{-1,p}([s,t])}
				\\&\les\|u^t\|_{\LL^q_{1,p}([s,t])}.
			\end{align*}
			Combining the previous estimates, we obtain \eqref{tmp.u-1}.

			Inequality \eqref{tmp.u0} is shown analogously. Indeed, using \cref{lem.ust}, we have
			\begin{align*}
			  \|u^t_s\|_{L_p(\Rd)}
			  &\les (t-s)^{1-\frac1q}(\|f\|_{\LL^q_{p}([s,t])}+\|(\delta^{ij}-a^{ij})\partial^2_{ij}u^t\|_{\LL^q_{p}([s,t])} )
			  \\&\les (t-s)^{1-\frac1q}(\|f\|_{\LL^q_{p}([s,t])}+\|\partial^2_{ij}u^t\|_{\LL^q_{p}([s,t])} ).
			\end{align*}
 			It is known (\cite[Theorem 2.1]{LX}) that $\|u^t\|_{\LL^q_{2,p}([s,t])}\les\|f\|_{\LL^q_{p}([s,t])}$.
			These estimates imply \eqref{tmp.u0}.
		\end{proof}
\end{appendix}
%
%


\section*{Data availability} 
	Data sharing not applicable to this article as no datasets were generated or analyzed during the current study.
\section*{Acknowledgment}
	The authors thank Guohuan Zhao for fruitful discussions and Lucio Galeati for pointing out an error in an earlier version of this work.  The authors would like to thank the editors, associated editors and the referees for their constructive suggestions. 
\section*{Funding} 
 	The first author was supported by the Humboldt fellowship and the second author was supported by the DFG via Research Unit FOR 2402 while they were at TU Berlin where this work was carried out.
\bibliographystyle{alpha}
\bibliography{reference}
\end{document}